\documentclass[12dpt,a4paper,DIV = 13,BCOR=0mm,pdftex]{scrartcl}

\usepackage[french]{babel}
\usepackage{pdfsync}
\usepackage[latin1]{inputenc}
\usepackage[colorlinks,linkcolor=blue!60!green!80!red,pdftex]{hyperref}

\usepackage{nccmath} 
\usepackage{amsfonts}
\usepackage{amsmath}
\usepackage{amsthm}
\usepackage{amssymb}
\usepackage{graphicx}
\usepackage{color}
\usepackage{tikz}
\usetikzlibrary{snakes}

\numberwithin{equation}{section}

\theoremstyle{plain}
\newtheorem{theorem}{Théorème}[section]
\newtheorem*{theorem*}{Théorème}
\newtheorem{lemma}[theorem]{Lemme}
\newtheorem*{lemma*}{Lemme}
\newtheorem{proposition}[theorem]{Proposition}
\newtheorem*{proposition*}{Proposition}
\newtheorem{corollaire}[theorem]{Corollaire}
\newtheorem*{corollaire*}{Corollaire}

\theoremstyle{remark}
\newtheorem{remark}[theorem]{Remarque}
\newtheorem*{remark*}{Remarque}
\newtheorem{exemple}[theorem]{Exemple}

\theoremstyle{definition}
\newtheorem{Def}[theorem]{Définition}
\newtheorem*{Def*}{Définition}
\newtheorem{notation}[theorem]{Notation}
\newtheorem*{notation*}{Notation}
\newtheorem{question}[theorem]{Question}

\newcommand{\corn}{{\,\begin{picture}(5,5)(0,0) \put(0,0){\line(1,0){6}}\put(0,0){\line(0,1){6}}\end{picture}\,}}
\newcommand {\ie}{c.\mbox-à\mbox-d. }

\newenvironment{keywords}{
       \list{}{\advance\topsep by0.35cm\relax\small
       \leftmargin=1cm
       \labelwidth=0.35cm
       \listparindent=0.35cm
       \itemindent\listparindent
       \rightmargin\leftmargin}\item[\hskip\labelsep
                                     \bfseries Mots-clés:]}
{\endlist}

\def\G{\mathbb G}
\def\H{\mathcal{H}}
\def\S{\mathcal{S}}
\def\diam{\operatorname{diam}}
\def\dist{\operatorname{dist}}

\def\h{\mathbb{H}^1}
\def\f{F^{-1}(0)}
\def\r{\mathbb{R}}
\def\Ker{\operatorname{Ker}}
\def\Im{\operatorname{Im}}
\def\Var{\underline{\operatorname{Var}}}
\def\a{\mathcal{T}}
\def\Z{\mathcal{Z}}
\def\Lift{\operatorname{Z}}

\begin{document}

\title{Rugosité des lignes de niveau des applications différentiables sur le groupe d'Heisenberg}

\author{KOZHEVNIKOV Artem\\
\'{E}cole Normale Supérieure, Paris, France\\
email: artem.kozhevnikov@ens.fr}

\date{4 janvier 2011}

\maketitle
\begin{abstract} 
Nous étudions les propriétés métriques des lignes de niveau des applications dif\-fé\-ren\-tiables (au sens intrinsèque) définies sur le premier groupe d'Heisenberg $(\h,d_\infty)$, muni de sa structure sous-riemannienne standard. Nous présentons notamment l'analyse exhaustive du cas d'une application $F\in C^1_H(\h,\r^2)$ dont la différentielle horizontale $D_hF$ est surjective. On découvre, en particulier, que les lignes de niveau d'une telle application peuvent être de nature rugueuse et ne font pas partie des "sous-variétés intrinsèques régulières".  
\end{abstract}


\newpage   
\setcounter{page}{2}
\tableofcontents
\newpage 
\section{Introduction}
\subsection{Motivation et résultats principaux} 
Cet article est consacré à l'étude des propriétés métriques des \textit{sous-variétés intrinsèques} dans le cadre de la géométrie sous-riemannienne. 
Nous traitons ici l'exemple de la structure sous-riemannienne non-triviale la plus simple, \ie celui du premier groupe d'Heisenberg $\h$. 
Pour nous, une sous-variété intrinsèque signifie un ensemble qui coïncide localement avec une ligne de niveau d'une application différentiable intrinsèquement dont la différentielle est surjective, par analogie avec le théorème classique des fonctions implicites. 

Le théorème classique des fonctions implicites dit que localement l'ensemble de niveau $G^{-1}(0)$ d'une application $G\in C^1(\r^n, \r^m)$, $G(0)=0$ et $dG(0)\langle \r^n\rangle =\r^m$, peut être représenté comme un graphe (d'une application $C^1$) associé à une décomposition de l'espace tangent en $\r^n=\Ker dG(0) \oplus \r^m$, où la différentielle $dG(0)$ restreinte au deuxième facteur $\r^m$ est un isomorphisme.  
Si on se place maintenant dans le cadre de la différentiabilité horizontale, on voit apparaître un nouveau phénomène algébrique propre à  la \textit{géométrie non-commutative}. Etant donné un morphisme de groupe surjectif $L: \G\to \mathbb{H}$, il n'est pas toujours possible de trouver "le deuxième facteur", \ie d'avoir une décomposition $\G=(\Ker L)\cdot \mathbb {H}$ en produit (semi-direct) de deux sous-groupes. Autrement dit, un morphisme $L: \G\to \mathbb{H}$ dans la suite exacte de groupes
 $$\{0\} \longrightarrow \Ker L \longrightarrow \G \xrightarrow{\ L\ } \mathbb{H} \to \{0\}$$ 
 n'est pas toujours scindé (à la différence du cas des groupes abéliens $\r^n$ et $\r^m$). 
 
Si, tout de même, la différentielle horizontale $D_hF(0) : \G\to \mathbb{H}$ entre deux groupes de Carnot est scindée, alors le théorème des fonctions implicites a lieu dans sa forme presque classique \cite{magnani} (voir théorème \ref{imptheor} pour la formulation précise). Les ensembles de niveau correspondants s'appellent dans ce cas-là sous-variétés intrinsèques \textit{régulières}  selon la terminologie de \cite{regularsubmanifoldsheisenberg}. Les variétés intrinsèques régulières ont été surtout beaucoup étudiées en codimension $1$ (hypersurfaces régulières), où leur régularité du point de vue de la géométrie sous-riemannienne a été établie  \cite{ structureperimeter, Heishypesurfaces, regularsubmanifoldsheisenberg}. 
 
Dans le travail présent nous considérons l'exemple modèle le plus simple d'une sous-variété intrinsèque pour laquelle la différentielle horizontale en question n'est pas scindée et le théorème \ref{imptheor} ne s'applique pas. Cet exemple est donné par une ligne de niveau de l'application $F\in C^1_H(\h,\r^2)$ avec $D_h F$ surjective. 
Les résultats que nous obtenons sont partiellement résumés dans 

\begin{theorem}\label{t: prinsipal} Soit $F\in C^1_H(\h,\r^2)$ telle que $F(0)=0$ et $D_h F(0)$ est surjective. Alors il existe un voisinage compact $U$ de l'élément neutre $0\in \h$ tel que les propriétés suivantes sont vérifiées :
\begin{enumerate}
\item L'ensemble $\f$ est $\varepsilon$\textsf{-Reifenberg plat} dans $U$ par rapport au centre de $\h$ (dit sous-groupe vertical) avec $\varepsilon \to 0$ uniformément lorsque l'échelle se raffine (définition \ref{d: reifenbegplat}, lemme \ref{platequation}). C'est en fait une caractérisation locale \textsf{exacte} de $\f$ au voisinage de $0$ (proposition \ref{p: reifWhitneysuff });
\item\label{mleo} L'ensemble $\Gamma:=\f\cap U$ est un arc simple (appelé \textsf{courbe verticale}) (corollaire  \ref{reifenbergcurve} du théorème \ref{t: parametrtheorem});
\item La dimension sous-riemannienne de $\Gamma$ égale $2$ : $\dim_h \Gamma = 2$ (corollaire \ref{airegamma} du lemme \ref{dimensionestime});
\item La dimension euclidienne de $\Gamma$ peut prendre toute valeur dans l'intervalle $[1,2]$ (lemme \ref{dimriem});
\item La "\textsf{formule de l'aire}" pour la mesure de Hausdorff sous-riemannienne de $\Gamma$ est donnée par 
\begin{equation}\label{e: aireenonce} 
 \H^2_\infty(\Gamma)=\frac{1}{2} \S^2_\infty(\Gamma)=\liminf\limits_{\|\sigma\|\to 0} \sum_\sigma d_\infty(A_i,A_{i+1})^2,
\end{equation}
 où $\sigma= \{A_0<A_1<\ldots< A_n\}$ désigne une subdivision ordonnée de $\Gamma$ ($A_0$ et $A_n$ sont les extrémités de $\Gamma$), et $\|\sigma\|= \max\limits_i d_\infty(A_i,A_{i+1})$ (corollaire \ref{airegamma} de la proposition \ref{generaire});  
\item Si, en outre, $F\in C^{1,\alpha}_H(\h,\r^2)$, $\alpha>0$, (ou même sous des hypothèses plus faibles sur la régularité de $F$), alors 
la courbe verticale $\Gamma$ est fortement régulière au sens d'Ahlfors, et la formule de l'aire se réécrit (quitte à choisir la bonne orientation) 
\begin{equation}\label{e: aireregulierenonce}
  \H^2_\infty(\Gamma)=\frac{1}{2} \S^2_\infty(\Gamma) = \int\limits_{\Gamma} \,dz+2\int\limits_{\Gamma}x\,dy-2\int\limits_{\Gamma} y\,dx,
\end{equation}
où il s'agit des \textsf{intégrales de Stieltjes} (lemme \ref{aireregulier}, proposition \ref{fortregular}, lemme \ref{regularderivees});
\item \label{irreg} Il existe, néanmoins, des exemples de courbes verticales $\Gamma_{1,2}$ "\textsf{rugueuses}" telles que 
$$\H_\infty^2(\Gamma_1)=\infty, \quad \H_\infty^2(\Gamma_2)=0 \quad (\text{exemples \ref{mesureinfinie} et \ref{mesurenulle}}).$$
  Nous obtenons également la \textsf{formule de la coaire}  (théorème \ref{theoremcoaire}) pour les applications $F\in C^{1,\alpha}_H(\h,\r^2)$, $\alpha>0$.
\end{enumerate}
\end{theorem}
Notons que la propriété \ref{mleo} a été récemment indépendamment obtenue dans \cite{magnleo}. 

Ainsi,  la propriété \ref{irreg} dit qu'une sous-variété intrinsèque $\Gamma$, qui est "irrégulière" pour la raison terminologique, peut, en effet, être irrégulière du point de vue métrique. La question naturelle dans ce contexte est de comprendre si ce phénomène se produit de façon systématique. Autrement dit, est-ce que dans la situation, où la différentielle horizontale (d'une application de classe $C^1_H$ entre deux groupes de Carnot) est surjective mais n'est pas scindée, des lignes de niveau correspondant sont toujours de la bonne dimension mais peuvent être irrégulières (au sens d'Ahlfors, par exemple)? Le cas où le groupe de Carnot image est non abélien sera certainement plus intrigant. D'une part, cela nécessitera probablement un certain analogue du théorème de prolongement de Whitney pour les applications de classe $C^1_H$ dont l'image est un groupe non abélien. D'autre part, il se peut que pour une raison de rigidité  structurelle, toute application de $C^1_H$ avec la différentielle horizontale surjective soit forcement plus régulière qu'on n'y penserait (un morphisme, par exemple). On peut se demander s'il existe des groupes de Carnot $\G$ et $\Bbb H$ tels que l'espace  $C^1_H(\G, \Bbb H)$, quotienté par l'action évidente des dilatations et translations, possède des points isolés.

\subsection{Description et structure de l'article}

Nous commençons par introduire la notion de groupe de Carnot dans la sous-section \ref{ss: notations}. Nous concentrons notre attention sur une structure métrique intrinsèque qui est donnée sur les groupes de Carnot par une \emph{distance homogène invariante à gauche} (définition \ref{d: normhomogene}). Remarquons qu'une telle distance sur un groupe de Carnot non-abélien n'est pas (bilipschitz) équivalente à une métrique riemannienne,  même localement. Un exemple privilégié d'une distance intrinsèque est celui d'une métrique dite de \emph{Carnot-Carathéodory}. Elle est définie entre deux points comme l'infimum de longueurs de courbes horizontales (absolument continues) reliant ces points, ce qui est naturellement motivé par la théorie du contrôle optimal. 

Dans notre étude, le choix concret d'une distance intrinsèque ne joue aucun rôle important. C'est, d'une part, parce que deux choix différents donnent toujours deux distances (bilipschitz) équivalentes et, d'autre part, si on veut passer d'une distance à une autre dans les formules \eqref{e: aireenonce} \eqref{e: aireregulierenonce}  il suffit seulement de tenir compte de leur rapport le long de l'axe vertical (remarque \ref{rapportdeuxdist}).  

Dans la sous-section \ref{ss: differenthorizont} nous présentons le concept de la différentiabilité horizontale (définition \ref{d: derivationhorizontale}) et ainsi que certains résultats la concernant qui nous seront utiles dans la suite. Ainsi, nous utiliserons le théorème de \textit{prolongement de Whitney} (théorèmes \ref{Whitneytheorem}, \ref{Whitneyholder}) pour caractériser de manière réversible les courbes verticales (condition \eqref{whitneycond} et remarque \ref{Whit}) et pour en construire des exemples avec des propriétés désirées par la suite. 

Dans la section \ref{hsurfaces} nous présentons un survol de certains aspects de la théorie des hypersurfaces $\h$-régulières qui est  basé sur les travaux \cite{Heishypesurfaces, regularsubmanifoldsheisenberg, submanifold}. On rappelle d'abord le théorème des fonctions implicites  (le cas de la différentielle horizontale scindée; théorème \ref{imptheor}) qui nous met en contexte général où émergent les surfaces intrinsèques régulières. Parmi elles, c'est notamment le cas des hypersurfaces (en $\operatorname{codim}=1$) qui a été le plus étudié (théorème \ref{imphypersurf}, définition \ref{defhypersurf}). Ceci est dû non seulement à sa simplicité en comparaison avec les cas de codimension plus grande, mais aussi à l'importance des hypersurfaces dans un théorème de structure des ensembles de périmètre fini (voir \cite{structureperimeter} pour le cas des groupes de profondeur $m=2$; \cite{ambrosio2008rectifiability} pour un progrès récent dans le cas général). 

Dans la sous-section \ref{regulsurf} nous donnons, en particulier, une nouvelle démonstration élémentaire du fait qu'une surface $\h$-régulière est feuilletée par des courbes horizontales de classe $C^1$ orthogonales à la normale horizontale (lemme  \ref{l: reguler}). Ceci est un ingrédient indispensable pour les considérations de la partie \ref{apptopol}.  
 A la différence de celle dans \cite{bigolin}, notre approche permet d'obtenir ce fait directement sans passer par un résultat (très remarquable, d'ailleurs) de Dafermos  \cite{dafermoscontinuous} sur des équations quasi-linéaires.

Le lemme \ref{metriqueR} dans la sous-section \ref{geomdescrp} fournit la description géométrique de la forme approximative de l'intersection des boules métriques avec une surface $\h$-régulière.  

Dans la section \ref{courbesvetricales} nous nous lançons dans l'étude des lignes de niveau d'une application $F\in C^1_H(\h, \r^2)$, $F(0)=0$, avec $D_hF$ surjective. 

En premier lieu, nous donnons un critère métrique (dit "condition de Whitney" \eqref{whitneycond}, remarque \ref{Whit})  pour qu'un ensemble compact $E$ soit contenu dans une ligne de niveau de $F$. Ce critère dit, grosso modo, que l'ensemble $E$ en tout point doit être tangent  (au sens de dilatations homogènes) à l'axe vertical $Oz$. 
Il est important de noter que les valeurs de la différentielle horizontale $D_hF$ prises sur $\f$ n'ont aucune conséquence sur l'ensemble $\f$ (de façon informelle : "la différentielle horizontale ne capte pas le comportement vertical des ensembles de niveau"). En effet, pourvu que $D_hF$ soit surjective sur $\f$, on peut toujours trouver une autre application $\tilde F \in C^1_H(\h, \r^2)$ telle que (au moins localement) $\f\subset \tilde F^{-1}(0)$ et, par exemple, $D_h\tilde F\corn \f= \operatorname{Id}\circ\, \pi$ soit l'identité sur le plan horizontal.  En particulier, la régularité de $D_hF\corn \f$ ne dit absolument rien sur la régularité de $\f$ (à la différence de la régularité de $D_hF$ au voisinage de $\f$). C'est un phénomène nouveau même pour la géométrie sous-riemannienne, car, la régularité d'une surface $\h$-régulière se
reflète dans la régularité de sa normale horizontale \cite{bigolinregul}. Cela ne laisse pas d'espoir d'avoir une expression de la mesure $\H^2_\infty\corn \f$ (formule de l'aire) qui ferait intervenir de façon non-triviale les valeurs $D_hF\corn \f$ (comparer avec le cas euclidien ou le cas des hypersurfaces régulières sous-riemanniennes). 

Comme le noyau $\Ker D_h F = Oz$ n'admet pas de groupe complémentaire dans $\h$, il n'existe donc pas de graphe intrinsèque associé (définition \ref{d: grapheintrinseque}). Par conséquent, même le paramétrage local de l'ensemble $\f$ est un problème nouveau non-trivial. La stratégie qui consisterait  à couper $\f$ par des surfaces lisses feuilletant l'espace (qu'on se donne à l'avance) ne peut réussir à la différence du cas où la différentielle est scindée.  En effet, on peut montrer que pour toute surface $S$ de classe $C^1$, on peut trouver un exemple de $F$ telle que localement $S$ contient plusieurs points de $\f$. Pour ce faire, il suffit de se placer au voisinage d'un point où la distribution horizontale n'est pas tangente à la surface et d'utiliser la condition de Whitney. 

Pour surmonter cette difficulté on fait appel à la théorie du paramétrage de Reifenberg des ensembles plats (voir \cite{reifenberg, vanishingreifenberg, reifenbergflatmetric} pour certains aspects de cette théorie dans $\r^n$). On démontre d'abord par un argument topologique qui fait intervenir le degré d'une application continue \cite{degree} que la projection de l'ensemble $\f$ sur l'axe $Oz$ est surjective au voisinage de $0$. Ce fait combiné avec la condition de Whitney conduit à ce que $\f$ est $\varepsilon$-Reifenberg plat par rapport à l'axe $Oz$ (définition \ref{d: reifenbegplat}), où $\varepsilon \to 0 $ uniformément lorsque l'échelle devient de plus en plus fine (lemme \ref{l: reifenberg}). Une simple réécriture de la condition de Whitney implique que c'est en fait une caractérisation locale exacte de $\f$ (proposition \ref{p: reifWhitneysuff }). On démontre ensuite (théorème \ref{t: parametrtheorem}) un analogue du théorème du paramétrage de Reifenberg adapté à notre situation. Grâce au fait que la dimension topologique du noyau $\Ker D_hF =Oz$  égale $1$, la preuve s'obtient facilement par un algorithme dyadique qui consiste à chercher à chaque étape un point "quelque part au milieu" entre deux points obtenus à l'étape précédente. 
Nous en déduisons ainsi que $\f$ est localement un arc simple. Ce résultat est avant tout topologique, car le paramétrage obtenu de $\f$ est encore plus "implicite" (vu l'absence de graphe adéquat) que dans le théorème usuel des fonctions implicites. 

Une autre approche du paramétrage de $\f$ consiste à regarder $\f$ comme une intersection de deux surfaces $\h$-régulières qui se coupent transversalement, \ie leur normales horizontales sont linéairement indépendantes (sous-section \ref{s: considtop}, à comparer avec \cite{magnleo}).  Ceci entraîne (grâce au lemme \ref{l: reguler}) que toute courbe horizontale sur une surface rencontre l'autre surface en exactement un point. On aimerait bien alors paramétrer l'ensemble  d'intersection par des courbes horizontales.  
Le problème ici est que les courbes horizontales sur une surface $\h$-régulières sont engendrées par un champ de vecteurs (non-singulier) qui n'est pas lipschitz régulier. Par conséquent, plusieurs de ses trajectoires peuvent passer par un point donné. Néanmoins, il est possible de régler ce problème en choisissant une sous-famille ordonnée de trajectoires (définition \ref{d: flotpent}) qui, munie de la topologie uniforme, est homéomorphe à un intervalle (théorème \ref{t: flotintervalle}). Cette sous-famille est appelée "flot sans pénétration" par analogie avec la mécanique de fluides. Cette méthode de sélection du flot continu (analogue, en fait, à celle dans \cite{funnel}) peut être vue comme un substitut du théorème de redressement d'un champ de vecteurs non-singulier qui est seulement continu. Malheureusement, elle ne s'applique qu'en dimensions $1$ et $2$ essentiellement; son extension à des dimensions plus grandes est probablement impossible. Or, la question qui nous intéresse dans ce contexte, c'est plutôt la topologie d'un ensemble transverse aux trajectoires (voir problème topologique \ref{probtop}), et elle est encore plus délicate. 

Tout au long de la section \ref{s: proprietemetrique} nous discutons les propriétés métriques d'une courbe  verticale $\Gamma=\f\cap U$ (définition \ref{d: courbevertic}). La première observation est que la courbe $\Gamma$ traverse la distribution de plans horizontaux (la structure de contact) toujours dans un seul sens  \eqref{e: contactpositif}. 
Cela avec la condition de Whitney entraîne que la distance intrinsèque au carré $d^2_\infty$ est une quasi-métrique plate sur $\Gamma$ (remarque \ref{r: dcarreplat}, condition \ref{e: kappaplate}).

 Nous passons ensuite quelque temps à regarder dans le contexte abstrait les courbes $\lambda$ munies de quasi-métriques plates $\kappa$ (sous-section \ref{ss: quasiplate}). En particulier, nous déduisons la dimension de Hausdorff de ces courbes ($\dim_\kappa \lambda =1$, proposition \ref{dimensionestime}) et une formule générale pour la mesure de Hausdorff $\H^1_\kappa\corn \lambda$ (proposition \ref{generaire}). Nous donnons également une condition suffisante pour la régularité au sens d'Ahlfors de la mesure $\H^1_\kappa\corn\lambda$ (remarque \ref{ahlfors}). Cette condition s'avère être assez forte : elle garantit, en outre, que la mesure $\H^1_\kappa\corn\lambda$ peut être calculée en terme de \emph{l'intégrale abstraite de Stieltjes} (lemme \ref{aireregulier}, lemme \ref{couture}). L'estimée des sommes de Stieltjes implique que la mesure $\H^1_\kappa$ d'un intervalle sur $\lambda$ est asymptotiquement égale à son diamètre (formules \eqref{youngestime} et proposition \ref{diametre}). 

De cette dernière considération nous déduisons des conséquences pour les courbes verticales (corollaire \ref{airegamma}, proposition \ref{fortregular}). Il est clair aussi que la régularité supplémentaire de $F$ améliore la régularité de ses lignes de niveau (lemme \ref{regularderivees}, corollaire \ref{regularitealpha}).
A titre d'exemple, nous caractérisons les lignes de niveau des applications $F\in C^{1,\alpha}_H(\h,\r^2)$, $\alpha>0$, comme des \emph{relèvements verticaux} (définition \ref{d: relevementvertic}, lemme \ref{l: vertcaract}) de courbes $\frac{1+\alpha}{2}$-hölderiennes planaires. Nous utilisons également l'idée du relèvement vertical des courbes fractales pour construire des courbes verticales de dimension euclidienne $\beta$ pour tout $\beta\in [1,2]$ (lemme \ref{dimriem}).  

En revanche, sans hypothèse de la régularité supplémentaire sur $F$, ses lignes de niveau ne sont pas forcement régulières. Dans la section \ref{exempleirregulier} nous construisons des exemples de courbes verticales $\Gamma_{1,2}$ avec $\H^2_\infty(\Gamma_1)=\infty$ et  $\H^2_\infty(\Gamma_2)=0$. En particulier, elles ne sont pas $2$-Ahlfors régulières. Remarquons que si $\H^2_\infty(\Gamma)=\infty$ pour une courbe verticale $\Gamma$, alors \emph{l'aire de Lévy}, calculée le long de sa projection $\pi(\Gamma)$ sur le plan horizontal, est infinie. L'aire de Lévy signifie ici l'intégrale $\int x\,dy-y\,dx$ comprise au sens de Stieltjes; pour prendre la limite on considère ainsi (notation \ref{n: airelevy}) le filtre de toutes les subdivisions (avec le pas maximal décroissant) sur une courbe et pas seulement celui de subdivisions dyadiques, par exemple (comparer avec le cas du mouvement brownien).  
La construction des exemples irréguliers (lemme \ref{verticalgros}) s'inspire notamment de la théorie de chemins rugueux qui a connu un développement majeur ces dernières années (on renvoie le lecteur à \cite{MR2036784, MR2314753} pour ses fondements). Pour réaliser cette construction nous avons besoin des estimations précises des sommes de Stieltjes associées à l'aire de Lévy pour les courbes $\frac{1}{2}$-hölderiennes ($\frac{1}{2}$ est un exposant critique ici; voir théorème \ref{young}, proposition \ref{p: youngconseque}, remarque \ref{r: souscritiqueLevy}). On obtient ces estimations  \eqref{totalformuleStiel}\eqref{totalformuleStiel1} dans le cas particulier des courbes qui sont données comme des séries de Fourier lacunaires \eqref{ondelettedecomp}\eqref{ondelettedecomp1} (voir aussi \cite{MR1116958}). 

A titre d'application de l'étude des courbes verticales, dans la section \ref{s :coaire} nous démontrons la formule de la coaire \eqref{coarea} pour les applications $F\in C^1_H(\h,\r^2)$ sous une hypothèse de régularité supplémentaire de ses lignes de niveau (théorème \ref{theoremcoaire}). Il est intéressant de comprendre si cette formule est valable  sans hypothèse de régularité supplémentaire. 

\paragraph{Remerciement.}
Je tiens à remercier profondément  Pierre Pansu pour les discussions scientifiques autour du sujet considéré ainsi que pour son soutien constant tout au long de ce travail. Je remercie également Guy David qui m'a fait découvrir des aspects importants de la théorie de paramétrage de Reifenberg. 

\section{Notions de base et notations}
\subsection{Groupe de Carnot}\label{ss: notations}
Un groupe de \emph{Carnot} \cite{FS} est un groupe de Lie nilpotent $\G$, connexe et simplement connexe, dont l'algèbre de Lie $\mathfrak{g}$ est  
\emph{stratifiée},
\ie se décompose en somme directe $\mathfrak{g}= V_1\oplus\ldots\oplus
V_m,\, \dim V_1 \ge 2 $, 
d'espaces vectoriels tels que $
\left[V_1,V_i\right]=V_{i+1}$ pour \,$1 \le i\le m-1 $
et $\left[V_1,V_m\right]=\{0\}$. Les vecteurs appartenant à $V_1$ s'appellent \emph{horizontaux} à la différence des vecteurs n'y appartenant pas appelés \emph{verticaux}.  La distribution horizontale $H\G$ associée est  
le sous-fibré de $T\G$ 
défini par $H_g\G=(\tau_g)_\ast(V_1)$, où $\tau_g(h)=gh$ est une translation à gauche.  L'entier $m$ est appelé la \emph{profondeur} de groupe de Carnot. On rappelle 
 \begin{theorem*} Soit $\G$ un groupe de Lie, connexe, simplement connexe et nilpotent. Alors l'application exponentielle $\exp: \mathfrak{g} \to \G$ est un difféomorphisme global.
 \end{theorem*}
Grâce à la graduation, on définit sur l'algèbre de Lie $\mathfrak{g}$ le groupe multiplicatif d'automorphismes à un paramètre  $\{\delta_t\}_{t>0}$ qu'on appellera les \textit{dilatations} :
$$\text{ pour } X\in V_i \text{ on a } \delta_t(X)=t^i X.$$

 Par le biais de l'application exponentielle $\exp:\mathfrak{g} \to \G$, on  transmet l'action de $\{ \delta_t\}_{t>0}$ sur $\G$ en gardant la même notation. Ainsi, $\Bbb G$ est-il également un espace homogène de dimension topologique $N=\sum\limits_{i=1}^m \dim V_i$ et de dimension homogène $Q=\sum\limits_{i=1}^m i\dim V_i$. La mesure de Lebegue $\mu$ sur $\Bbb R^N$ transportée via $\exp$ est une mesure de Haar biinvariante sur $\G$ et $d(\delta_t\mu)=t^{Q}d\mu$.
 
\begin{Def}\label{d: normhomogene}
Une \emph{norme homogène} $\rho$ est une fonction continue sur $\mathbb{G}$, qui satisfait les propriétés suivantes: 
\begin{itemize}
\item $\rho(x)\ge 0$, $\rho(x)=0$ si et seulement si  $x=e$ (l'élément neutre); 
\item $\rho(x^{-1})=\rho(x)$,  $\rho(\delta_t(x))=t\rho(x)$, $t>0$; 
\item l'inégalité triangulaire généralisée:  $\rho(x y)\le c(\rho(x)+\rho(y))$, $c\ge 1$. 
\end{itemize}
A partir d'une norme homogène $\rho$, on construit la \emph{distance homogène} $d_\rho$ sur
$\mathbb{G}$: on pose $d_\rho(x,y)=\rho(y^{-1}x)$ pour tous $x,y \in \G$. La distance homogène (\ie $d_\rho(\delta_t x,\delta_t y)=t d_\rho(x,y), \, t>0$) est invariante par translations à gauche $d_\rho(zx,zy)=d_\rho(x,y)$.
\end{Def}

\begin{remark*}
Si $c=1$ dans la définition de $\rho$, alors la distance homogène $d_\rho$ est une vraie métrique, tandis qu'en général $d_\rho$ est seulement une quasi-métrique.
\end{remark*}

 \begin{remark*}
Si $d_\rho$ est une métrique sur $\G$, alors  $\diam_\rho B_\rho(x,r)=2r$ \cite{structureperimeter}, où $B_\rho(x,r)=\{y\in \G \mid d_\rho(x,y)<r\}$ est une boule ouverte en distance $d_\rho$. 
\end{remark*}

\begin{remark*}
Deux normes homogènes, $\rho^{\prime}$ et $\rho$, définies sur $\G$ sont toujours équivalentes \cite{FS}:
il existe des constantes $c_1$ et $c_2$ telles que
$0<c_1\le \rho^{\prime}(x)/\rho(x)\le c_2<\infty$ pour tout
$x\in \G \setminus \{e\}$.
\end{remark*}

\begin{proposition*}[\cite{FS}]  Quelle que soit une métrique riemannienne $d_{riem}$ sur $\G$, pour chaque partie compacte $K$ de $\G$ il existe deux constantes positives $C_1$ et $C_2$ telles que pour tous $x,y \in K$ on a l'estimation  suivante :
\begin{equation*} 
C_1d_{riem}(x,y)\le d_\rho(x,y) \le C_2d_{riem}(x,y)^{\frac{1}{m}}.
\end{equation*}
En particulier, les topologies définies par $d_\rho$ et $d$ coïncident. 
\end{proposition*}
\paragraph{Mesure de Hausdorff.}
On se place dans l'espace quasi-métrique $(X, d)$. Par $\diam E=\sup\{d(x,y)\mid x,y\in E\}$ on note le diamètre d'un ensemble $E\subset X$. 
 Soit $0\le k\in \r$; pour $\varepsilon>0$ et un ensemble $E\subset X$ on définit
\begin{align*}
&\displaystyle \H_{\varepsilon}^k(E)=\inf \left\{\sum\limits_{i=1}^\infty (\diam E_i)^k \mid E \subset \bigcup\limits_{i=1}^\infty E_i, \diam E_i\le\varepsilon\right \},\\
& \mathcal{S}_{\varepsilon}^k(E)=\inf \left\{\sum\limits_{i=1}^\infty (\diam E_i)^k \mid E \subset \bigcup\limits_{i=1}^\infty E_i, \diam E_i\le \varepsilon, E_i \text{ est une boule}\right\}.
\end{align*}

\begin{Def*}
On définit \emph{la mesure de Hausdorff  $k$-dimensionnelle} $\H^k$ de $E$ comme $\H^k(E)=\lim\limits_{\varepsilon\to 0+}   \H_{\varepsilon}^k(E) =\sup\limits_{\varepsilon> 0}   \H_{\varepsilon}^k(E)$, ainsi que \emph{la mesure de Hausdorff $k$-dimensionnelle sphérique} $\mathcal{S}^k$  de $E$ comme $\mathcal{S}^k(E)=\lim\limits_{\varepsilon\to 0+}   \mathcal{S}_{\varepsilon}^k(E) =\sup\limits_{\varepsilon> 0}   \mathcal{S}_{\varepsilon}^k(E)$.
\end{Def*}
Les deux mesures, $\H^k$ et $\S^k$, sont des mesures (extérieures) régulières sur $\G$ et comparables entre elles.
 On appelle 
\emph{dimension de Hausdorff} de $E$
le nombre
   $$\dim E = \sup\{k \mid \H^k(E)=\infty \}=\inf\{k \mid \H^k(E)=0 \} .$$ 
\paragraph{Groupes d'Heisenberg.} 
Le $n$-ème groupe  d'Heisenberg $\Bbb H^n$ est un groupe de Carnot de dimension topologique $N=2n+1$ dont l'algèbre de Lie de profondeur $m=2$ : $\mathfrak{h} = V_1 \oplus V_2$. Ici $V_1$ est de dimension $2n$, il est engendré par les vecteurs $X_1, \dots ,X_n, Y_1, \dots , Y_n$, tandis que 
$\dim V_2 = 1$ et $V_2 = \operatorname{span} \{Z\}$. La dimension homogène de $\Bbb H^n$ égale donc $Q=2n+2$. Les seules relations de commutation non-triviales sont données par $[X_j , Y_j ] = -4Z$. Par la formule de Baker-Campbell-Hausdorff on obtient l'opération de groupe sur $\Bbb H^n=\r^{2n+1}=\r^n\times\r^n\times\r$ : 
$$\left(\begin{array}{c}x \\y \\ z\end{array}\right)\left(\begin{array}{c}x' \\y' \\z'\end{array}\right)= \left(\begin{array}{c}x+x' \\y+y' \\z+z' +2(x', y)_{\r^n} - 2(x, y')_{\r^n}\end{array}\right).$$  
L'action des dilatations est donnée par $\delta_t(x,y,z)=(tx,ty,t^2z)$. On note aussi $\pi(x,y,z):=(x,y)$ la projection sur le plan horizontal.
Dans tout ce qui suit nous utiliserons la norme homogène suivante sur $\Bbb H^n$ :$$\rho_\infty(x,y,z)=\max\{\sqrt{\| x\|^2+\| y \|^2},|z|^{1/2}\},$$
où $\|\cdot\|$ signifie la norme euclidienne.   
Par un calcul direct on peut montrer que la distance engendrée, notée $d_\infty$, est en fait une métrique. Les mesures de Hausdorff construites à partir de $d_\infty$ serons notées $\S^k_\infty$ et $\H^k_\infty$. La dimension de Hausdorff correspondante sera notée $\dim_h E$ pour $E\subset \Bbb H^n$. 
 Les champs de vecteurs invariants à gauche sont
 \begin{align*}
& X_i(x,y,z)=\frac{\partial}{\partial x_i} +2y_i\frac{\partial}{\partial z}, \quad Y_i(x,y,z)=\frac{\partial}{\partial y_i} -2x_i\frac{\partial}{\partial z}, \\
&Z(x,y,z)=\frac{\partial}{\partial z}=-\frac{1}{4}[X_i,Y_i],   \quad i=1, \dots, n.
\end{align*}

\subsection{Différentiabilité horizontale}\label{ss: differenthorizont}
\paragraph{Homomorphisme homogène}
\begin{Def*} Soit $\mathbb{G}^1$ et $\mathbb{G}^2$~--- deux groupes de Carnot avec les dilatations respectives $\delta_t^1$ et $\delta_t^2$. L'homomorphisme continu $L:\mathbb{G}^1\to \mathbb{G}^2$ est dit \emph{homogène} (ou \emph{horizontal}), si  $L\circ\delta_t^1=\delta^2_t\circ L$ pour tout $t>0$.
\end{Def*}

A un morphisme continu de groupes de Carnot  $L:\mathbb{G}^1\to \mathbb{G}^2$
correspond  un morphisme 
d'algèbres de Lie $\mathcal{L}=\exp_2^{-1}\circ
L\circ\exp_1 :\mathfrak{g}^1\to \mathfrak{g}^2$. Dans le cas d'un homomorphisme homogène $L$, $\mathcal{L}(V_1^1)\subset V_1^2$.

Le noyau d'un homomorphisme homogène $\Ker L=\mathbb{K}$ est un sous-groupe distingué homogène dans
$\mathbb{G}^1$. De même, $\Ker
\,\mathcal{L}=\mathcal{K}$ est un idéal homogène dans 
$\mathfrak{g}^1$.  L'ensemble est dit homogène s'il est préservé par l'action des dilatations 
$\{\delta_t\}_{t>0}$.  On remarque qu'un sous-espace homogène $W\subset \mathfrak{g}$ admet une décomposition en somme directe: $W=(W\cap V_1)\oplus \ldots \oplus (W\cap V_m)$.

\paragraph{Notion de dérivée horizontale.} Nous introduisons la notion de différen\-tia\-bi\-li\-té bien adaptée à la métrique interne, originalement  
due à P. Pansu.   
\begin{Def}[\cite{Pansu}] \label{d: derivationhorizontale}Soit $f:\mathbb{G}^1\to \mathbb{G}^2$ une application 
entre des ouverts de deux groupes de Carnot, $(\G^1, d^1)$ et $(\G^2, d^2)$, munis 
de distances homogènes. L'application $f$ est appelée \emph{horizontalement différentiable} en $x \in \G^1$, s'il existe un homomorphisme horizontal $L$ tel que 
$$ d^2\big(f(x)^{-1}f(xh), L(h)\big)=o(d^1(h,e)) \text{ lorsque } h \to 0.$$
Si c'est le cas, $L$ est noté par $D_hf(x)$ et  appelé la \emph{différentielle horizontale} de $f$ en $x$.
\end{Def} 
\begin{remark*}
 L'espace des homomorphismes horizontaux $\operatorname{Hom}_h(\G^1, \G^2)$ entre $\G^1$ et $\G^2$ est muni naturellement 
 d'une structure de groupe et de la norme
 $$ \|L\|=\sup\limits_{A\in \G^1\setminus e_1}\frac{d^2(L(A),e_2)}{d^1(A,e_1)}. $$
\end{remark*}
Si $f$ est différentiable en tout point d'un ouvert $\Omega$ et sa différentielle $D_hf(x)$ dépend con\-ti\-nû\-ment de $x$, alors l'application $f$ est dite \emph{horizontalement continûment différentiable} sur  $\Omega$. La classe de 
ces applications est notée $C^1_H(\Omega, \G^2)$, ou en abrégé $C^1_H$, si le contexte détermine sans ambiguïté les espaces de départ et de d'arrivée. 
Comme dans la situation classique, la règle de dérivation 
des fonctions composées est vérifiée pour des applications de classe $C_H^1$, ainsi que 
les autres règles arithmétiques. On remarque aussi qu'une application de classe $C^1_H$ entre deux groupes de Carnot est localement lipschitzienne en distances homogènes.
Pour désigner les applications de $C^1_H$ on utilise également  le terme \emph{horizontale} ou \emph{de contact}.

\paragraph{Critère en termes des dérivées partielles horizontales.}
 Nous présentons l'analogue du théorème classique qui dit que la continuité des dérivées partielles entraîne la différentiabilité continue (voir \cite{vquasi}, par exemple). 

\begin{theorem*} Soit $f:\G^1\to\G^2$ une application entre deux 
groupes de Carnot.  Alors $f\in C^1_H(\Omega,\G^2)$ si et seulement si pour  
tout champ de vecteur $X\in H\h$ horizontal invariant à gauche la dérivée partielle $Xf(x)=\frac{d}{dt}[\exp(tX)(x)]_{t=0}$ est continue sur $\Omega$ et pour tout $x\in \Omega$ le vecteur $Xf(x)$ appartient à la distribution horizontale  de $\G^2$ 
au point $f(x)$: $Xf(x)\in H_{f(x)}\G^2$.
\end{theorem*} 
Dans les coordonnées choisies ci-dessus sur le groupe $\h\cong \r^3$, la différentielle horizontale pour $f\in C^1_H(\h,\r)$ s'écrit 
$$ D_hf(A)(v_x,v_y,v_z)=v_xXf(A)+v_yYf(A)=\nabla_{\h} f(A)\cdot \pi(v), $$
où on note par $\nabla_{\h} f(A):= \big(Xf(A),Yf(A)\big)$ le gradient horizontal de $f$ en $A$.
 
Il sera utile de se rappeler aussi la généralisation du théorème classique de Lagrange.  
\begin{theorem}[de Lagrange \cite{FS}]\label{lagrange}
Pour $f\in C^1_H(\h,\r)$ l'inégalité suivante est vérifiée pour tous $A,B\in \h$.
 \begin{multline*}
|f(A)-f(B)-D_h f(B)(B^{-1}A)|\le \\
\le Cd_\infty(A,B)\big(\|Xf(\cdot)-Xf(B)\|_{\infty, \,B_\infty(B,r)}+ \|Yf(\cdot)-Yf(B)\|_{\infty, \, B_\infty(B,r)} \big),
\end{multline*}
où $B_\infty (B,r)$ désigne une boule en métrique $d_\infty$ de centre $B$ et de rayon $r=cd_\infty(A,B)$, $c,C<\infty$.
\end{theorem}
\begin{Def*}
On dit qu'une application $l:(X,d_X)\to (Y,d_Y)$ entre deux espaces quasi-métriques est Hölder continue 
d'exposant $\beta>0$, et on note $l\in H^\beta(X,Y)$, si  
\begin{equation*}
\|l\|_{H^\beta}:=\sup\limits_{d_X(a,b)>0} d_Y(l(a),l(b))d_X(a,b)^{-\beta}<\infty.
\end{equation*}
On dit que $l\in h_\beta(X,Y)$ si pour tout ensemble compact $K\Subset X$
\begin{equation*}
\sup_{\substack{0<d_X(a,b)<\delta,\\ a,b\in K}} d_Y(l(a),l(b))d_X(a,b)^{-\beta}\to 0 \text{ quand } \delta\to 0.
\end{equation*}
\end{Def*}
\begin{Def*}
On notera 
$C^{1,\alpha}_H(\h;\r)$, $0<\alpha<1$, l'espace des fonctions horizontalement dérivables dont la différentielle horizontale satisfait pour  une certaine constante $M$ et tous $A,B\in \h$
$$ \|D_hf(A)^{-1}D_hf(B)\|\le Md_\infty(A,B)^\alpha.$$
\end{Def*}

\paragraph{Théorèmes de prolongement de Whitney.} Il s'agit de prolonger une application scalaire dif\-fé\-re\-ntiable initialement donnée sur un ensemble fermé 
en une application définie sur l'espace tout entier. Nous formulons ici ce théorème pour le groupe d'Heisenberg $\h$.

\begin{theorem}[\cite{pupychev}]\label{Whitneytheorem} Soit $F\subset \h$ un ensemble fermé. Supposons que $f:F\to \r$ et $k:F \to \operatorname{Hom}_h(\h, \r)$ sont continues. On pose 
$$ R(A,B):=f(A)-f(B)-k(A)(B^{-1}A),$$
et pour un ensemble compact $K\subset F$ on définit
$$ \rho_K(\varepsilon):=\sup\left\{\frac{| R(A,B)|}{d_\infty(A,B)} \mid A,B \in K, 0<d_\infty(A,B)<\varepsilon\right\}.$$
Si $\rho_K(\varepsilon)\to 0$ quand $\varepsilon \to 0+$ quel que soit un ensemble compact $K\subset F$, alors 
il existe une fonction $\tilde f \in C^1_H(\h,\r)$ telle que $\tilde f_{\arrowvert F}=f$ et $D_h\tilde f_{\arrowvert F}=k$. 
\end{theorem}
Nous aurons également besoin du théorème de prolongement de Whitney avec le contrôle sur le module de continuité des dérivées horizontales. Nous utiliserons en particulier ce résultat pour le  module de continuité de Hölder.
\begin{theorem}[\cite{pupychev}]\label{Whitneyholder} Soit $K\subset \h$ un ensemble compact. On se donne $0<\alpha<1$, $f:K\to \r$ et $k:K \to \operatorname{Hom}_h(\h, \r)$ tels que pour tous $A,B\in K$ 
$$\|k(A)^{-1}k(B)\|\le Md_\infty(A,B)^\alpha,\quad |R(A,B)|\le Md_\infty(A,B)^{1+\alpha}.$$ 
Alors il existe une fonction $\tilde f \in C^{1,\alpha}_H(\h,\r)$ telle que $\tilde f_{\arrowvert K}=f$, $D_h\tilde f_{\arrowvert K}=k$ et 
$$\|D_h\tilde f(A)^{-1}D_h\tilde f(B)\|\le CMd_\infty(A,B)^\alpha$$ 
pour tous $A,B\in \h$, où $C$ est une constante universelle. 
\end{theorem}

\section{Les surfaces  $\h$-régulières}\label{hsurfaces}
\subsection{Théorème des fonctions implicites}
 
\begin{Def}[Factorisation des groupes de Carnot]  Soit $\Bbb N \triangleleft \G$ un sous-groupe homogène. Un sous-groupe homogène $\Bbb H \triangleleft \G$ est dit  \emph{complémentaire} à $\Bbb N$ si $\G$ est un produit semi-direct de $\Bbb N$ et  $\Bbb H$, \ie $\Bbb N \cap\Bbb H= e$ et tout élément $g\in\G$ peut s'écrire $g=nh$, où $n\in \Bbb N$ et $h\in \Bbb H$.
\end{Def}    

\begin{Def}[de graphe intrinséque régulier \cite{regularsubmanifoldsheisenberg}]\label{d: grapheintrinseque}
Etant donnée une factorisation $\G=\Bbb N\cdot\Bbb H$ en produit semi-direct des sous-groupes homogènes, et $\Bbb N$ est normal, le \textit{graphe intrinsèque régulier} associé, qui est engendré par une application $\phi : \Omega\subset \Bbb N \to \Bbb H$, est un ensemble   
$$\S_\phi =\{ n\phi(n) \mid n\in \Omega\} \subset \G.$$
\end{Def}

 \begin{theorem} [des fonctions implicites \cite{magnani}]\label{imptheor}
 Soit $\Omega \subset \G^1$ un ensemble ouvert et $f\in C^1_H(\Omega, \G^2)$. Supposons que la différentielle horizontale  $D_hf(a_0):\G^1\to \G^2$, $a_0\in \Omega$, est surjective. On suppose en outre que le noyau $\Bbb N=\Ker D_hf(a_0)$ admet un groupe complémentaire $\Bbb H$ (on a donc une décomposition associée  $a_0=n_0h_0$). Alors il existe un voisinage $U_{\Bbb N}$ de $n_0$ dans $\Bbb N$, un voisinage $U_{\Bbb H}$ de $h_0$ dans $\Bbb H$ tels que localement la ligne de niveau $f^{-1}(f(a_0))$ s'écrit comme un graphe intrinsèque régulier engendré par une unique application continue $\varphi:U_{\Bbb N}\to U_{\Bbb H}$, \ie 
 $$ f^{-1}\big(f(a_0)\big) \cap U_{\Bbb N}U_{\Bbb H}= \{n\varphi(n)\mid n\in U_{\Bbb N}\}.$$ 
De plus, il existe une constante $C$ telle que 
$$d_\rho\big(\varphi(n),\varphi(n') \big) \le C d_\rho\big(n\varphi(n'),n'\varphi(n') \big).$$
En particulier, l'application $\varphi$ est de classe de Hölder d'exposant $\mfrac{1}{m}$ par rapport à la distance $d_\rho$ sur $U_{\Bbb H}$ et à une norme euclidienne $\|\cdot \|$ sur $U_{\Bbb N}$, où $m$ est la profondeur de $\G^1$.
\end{theorem} 

 On voit que dans le cas de la factorisation, $\Bbb H$ est isomorphe à $\G^2 \simeq \G^1/ \Bbb N$. Donc, en particulier, $\Bbb H$ lui-même est un sous-groupe de Carnot de $\G^1$. On constate ainsi qu'en général la condition de la  factorisation ne se réalise que dans des cas assez spéciaux. 
\paragraph{Hypersurfaces sous-riemanniennes.} Le cas le plus simple et le plus étudié auquel le théorème \ref{imptheor} s'applique est celui d'une application scalaire $f:\G\to \r$. Si $D_hf(a_0)$ est surjective, alors il suffit de prendre un élément $a\notin \Ker D_hf(a_0)$ et définir $\Bbb H$ comme étant le sous-groupe engendré par $a$. Autrement dit, il existe un champ de vecteurs horizontal $X$ tel que $Xf(a_0)\not= 0$. Par conséquent, localement l'application $f$ est strictement monotone le long des lignes intégrales du champs $X$.   

\begin{theorem}[des fonctions implicites en $\operatorname{codim} =1$ \cite{Heishypesurfaces}]\label{imphypersurf} Soit $\Omega\subset \G$ un ensemble ouvert, $0\in \Omega$, où  on identifie $\G\cong \r^N$ via les coordonnées exponentielles. Soit  $f\in C_H^1(\G,\r)$ telle que $f(0)=0$ et $X_1f(0)>0$. On note  $I_\delta=\{\xi=(\xi_2,\ldots,\xi_N)\in \r^{N-1}, |\xi_i|<\delta\}$. Il existe un voisinage $U\subset \Omega$ de $0$, $\delta>0$, et une fonction continue $\phi: I_\delta \to \r$ tels que 
$$ f^{-1}(0)\cap U=\big\{ \Phi(\xi):=\exp(\phi(\xi)X_1)(0,\xi) \mid \xi \in I_\delta \big\}. $$
 En outre, il existe $\mathrm{s}:f^{-1}\cap U\to \r$ une fonction borélienne, $ \frac{1}{c} \le\mathrm{s}\le c$ pour $c>1$,  telle que la formule suivante est vérifiée
$$ \int\limits_{I_\delta}\dfrac{\sqrt{\sum\nolimits_{i=1}^{\dim V_1} | X_if\circ \Phi(\xi) |^2}}{X_1f\circ \Phi(\xi)}\,d\mathcal{L}^{N-1}(\xi)=\int\limits_{f^{-1}(0)\cap U} \mathrm{s} \,d\mathcal{S}^{Q-1}_\rho,$$
où les champs de vecteurs horizontaux invariants à gauche $\{X_1,\ldots, X_{\dim V_1} \}$ forment une base de $H\G$.  
\end{theorem}

\begin{Def}[\cite{Heishypesurfaces, regularsubmanifoldsheisenberg}]\label{defhypersurf}
L'ensemble $S\subset \G$ est appelé \emph{hypersurface régulière sous-riemannienne} (ou $\G$-hypersurface) s'il est donné localement au voisinage de tout point comme une ligne de niveau $f^{-1}(0)$ d'une application  $f\in C_H^1(\G,\r)$ dont la différentielle $D_h f$ ne s'annule pas.
On note $\nu_S(p)$ la normale horizontale vers $S$ en point $p\in S$, \ie le vecteur qui est donné localement par   $$\nu_S(p):=-\frac{\bigtriangledown_\G f (p)}{\| \bigtriangledown_\G f (p)\|_p }, \quad \bigtriangledown_\G f (p):=\big(X_1f(p), \ldots,X_{\dim V_1}f(p)\big).$$ 
\end{Def} 
 Ici la norme $\|\cdot\|_p$ est induite par un produit scalaire sur $H_p\G$ qui rend $\{X_1(p),\ldots,X_{\dim V_1}(p)\}$ orthonormaux. Remarquons que selon la définition $\nu_S$ peut être toujours choisie continue sur $S$.
  
\subsection{Paramétrage des surfaces $\h$-régulières}\label{regulsurf}

\paragraph{Notations et premières propriétés.}
Soit $\mathcal{S}$ une surface $\h$-régulière qui est localement donnée par la ligne de niveau $f^{-1}(0)$ d'une application scalaire $f\in C^1_H(\h,\r)$ telle que $Xf(0)\not = 0$ et $f(0)=0$.  En vertu du théorème des fonctions implicites, nous supposons que $\mathcal{S}$ est un $X$-graphe intrinsèque régulier, \ie  
$$\S:=f^{-1}(0)\cap U=\left\{\exp\big(\phi(y,z)X\big)(0,y,z) \mid (y,z)\in \Omega\subset \r^2\right\},$$
où $U$ est un voisinage de $0\in \h$ et l'application scalaire continue $\phi:\Omega \to \r$ satisfait certaines conditions sur un ouvert $\Omega$ que l'on précisera. On note l'application
$$ \Phi(y,z):=\exp\big(\phi(y,z)X\big)(0,y,z)=(0,y,z)(\phi(y,z),0,0)= \big(\phi(y,z),y, z+2\phi(y,z)y\big),$$
qui est en fait l'homéomorphisme de $\Omega$ sur son image $\mathcal{S}$. 
  L'opération de groupe s'écrit pour les points du graphe
$$\Phi(y_1,z_1)^{-1} \Phi(y_2,z_2)= \left(\begin{array}{c}\phi(y_2,z_2)-\phi(y_1,z_1)\\ y_2-y_1\\ z_2-z_1+2(y_2-y_1)\big(\phi(y_1,z_1)+\phi(y_2,z_2)\big)\end{array}\right).$$ 
Pour la distance  d'Heisenberg induite sur le graphe on gardera la même notation: pour  $A=(y_1,z_1), B=(y_2,z_2)\in \Omega$ on aura
$$d_\infty (A,B)=\max\{\sqrt{|\phi(B)-\phi(A)|^2+|y_2-y_1|^2}, |z_2-z_1+2(y_2-y_1)(\phi(B)+\phi(A))|^\frac{1}{2}\}.$$  
On définit également une fonction  $d_\phi:\Omega^2\to [0,\infty]$ : pour  $A, B\in \Omega$ on pose
\begin{equation*}
d_{\phi}(A,B)=\max\{|y_2-y_1|,|z_2-z_1+2(y_2-y_1)(\phi(B)+\phi(A))|^\frac{1}{2}\}.
\end{equation*}
La différentiabilité horizontale continue de $f$ se traduit  sur $\S$ comme 
$$ \phi(y_2,z_2)-\phi(y_1,z_1)-w(B)(y_2-y_1)=o(d_{\infty}(A,B)),$$
où le petit-$o$ est uniforme sur chaque partie compacte de $\Omega$, et $w(B)=-\mfrac{Yf}{Xf}\circ \Phi(B)$ une notation qu'on gardera par la suite.

On peut en déduire 
\begin{lemma}[propriétés de $\phi$ \cite{submanifold}] La fonction $d_\phi$ est une quasi-métrique sur $\Omega$ qui est localement équivalente à la distance induite $d_\infty$, \ie elle est bi-lipschitzienne par rapport à l'autre sur tout $\Omega^{\prime} \Subset\Omega$.
 En outre, l'application $\phi$ satisfait 
 \begin{align}
& |\phi(y_2,z_2)-\phi(y_1,z_1)|=O(d_\phi(A,B) ),\notag\\
& |\phi(y_2,z_2)-\phi(y_1,z_1)|=o(|A-B|^\frac{1}{2}),\label{hoderdemi} \\
& \phi(y_2,z_2)-\phi(y_1,z_1)-w(B)(y_2-y_1)=o\left(d_{\phi}(A,B)\right), \label{Whytney} 
\end{align}
où   les "o" et "O" sont uniformes sur tout $\Omega'\Subset\Omega$.
\end{lemma}

\begin{remark}
L'application $\phi$ "hérite" de la régularité euclidienne de $f$. Par exemple, si $f$ est $\alpha$-Hölder au sens euclidien, alors $\phi$ l'est aussi.  
\end{remark}
\begin{proof} Pour $A,B\in \Omega'\Subset \Omega$ on note $A'=\Phi(A)$, $B'=\Phi(B)$,  et on introduit le point mixte $C'=B\,(\phi(A),0,0)$. 
Ainsi,
$$ \min\limits_{\Omega^{'}} \{ |Xf| \} |\phi(B)-\phi(A)|\le|f(B')-f(C')|=|f(A')-f(C')|\le C_\alpha |A'-C'|^\alpha,$$
et, on conclut avec $|A'-C'|\le C_{\exp} |A-B|$  vu que l'application exponentielle est euclidienne lipschitzienne. 
\end{proof}

\begin{remark*}
En général, les surfaces $\h$-régulières ont un mauvais comportement vis-à-vis de la métrique euclidienne. Citons ainsi la construction dans \cite{kircassano} qui produit l'exemple des surfaces $\h$-régulières $\S$ de dimension euclidienne égale $\dim_E \S = 2.5$. Notons aussi qu'en général d'après le théorème de comparaison de dimensions \cite{comdim} $\dim_E \S\in [2, 2.5]$ vu que $\dim_h \S= 3$.
\end{remark*}

\paragraph{Champ horizontal transporté.}

Maintenant nous allons étudier plus en détail la géométrie de $\mathcal{S}$ à travers l'application $\phi$.  
Tout d'abord, on considère le champ de vecteurs continu sur $\Omega$ qui joue un rôle très important dans notre étude
$$
W^\phi(y,z)=\partial_y-4\phi(y,z)\partial_z.
$$
Si $\phi$ est de classe $C^1$ (autrement dit, $\mathcal{S}$ est une surface régulière au sens euclidien sans points caractéristiques), alors $W^\phi$ est le champ horizontal sur $\mathcal{S}$ tiré en arrière par $\Phi$. 

Considérons $\gamma(t)$ une ligne intégrale du champ de vecteurs $W^\phi$ issue du point $(y_0,z_0)$,  \ie une solution de problème de Cauchy 
\begin{equation}\label{cauchyproblem}
\left\{\begin{array}{l}
\gamma^{\prime}(t)=W^\phi\circ\gamma(t)\\
\gamma(0)=(y_0,z_0),
\end{array}\right.
\quad \Longleftrightarrow\quad
\left\{\begin{array}{l}
 y(t)=y_0+t \\
  z(t)=z_0-4\int\limits_0^t \phi\big(y_0+s, z(s)\big)\,ds
\end{array}\right. .
\end{equation}

D'après le théorème de Peano, une telle courbe (de classe $C^1$) existe pour un intervalle de temps non-vide $I_\gamma=(t_1, t_2)$.
Nous considérons toujours des solutions maximales de \eqref{cauchyproblem}, \ie $\gamma(t)\to \partial \Omega$ lorsque $t\to t_1$ et $t\to t_2$. 
Remarquons qu'une solution de \eqref{cauchyproblem} n'est pas nécessairement unique. 

Grâce aux propriétés  spéciales de $\phi$ on obtient le résultat clé suivant (comparer avec \cite{bigolin}). 
\begin{lemma}[Régularité supplémentaire] \label{l: reguler}Toute courbe $\gamma(\cdot)$ solution de \eqref{cauchyproblem} est de classe $C^2$, ou de façon équivalente, $\phi\circ\gamma(\cdot)$ est de classe $C^1$ et, donc, $\Phi\circ\gamma(\cdot)$ est une courbe horizontale de classe $C^1$ sur $\mathcal{S}$.
En outre, pour tout $t \in I_\gamma$
 \begin{align}
& \big( \phi\circ\gamma\big)'(t)=w(\gamma(t)), \\
& (\Phi\circ\gamma)'(t)= \Big(-\frac{Yf}{Xf}X+Y\Big)\circ\Phi\circ\gamma(t).
\end{align}
\end{lemma}
\begin{proof}Pour $t\in I_\gamma$ on note 
$$\Delta \phi(t):=\phi\circ\gamma(t)-\phi(y_0,z_0);\quad \hat w(t):=w(\gamma(t)).$$
En réécrivant \eqref{Whytney} compte tenu des relations intégrales \eqref{cauchyproblem} pour $t\in[-\delta, \delta]\subset I_\gamma$, $\delta>0$, on obtient 
\begin{align*}
| \Delta \phi(t)| &\le | \hat w(t)t|+| \Delta \phi(t) )-\hat w(t) t|  
\le | \hat w(t)t| +\epsilon_t\big(| t| +| z(t)-z_0+4t\phi(y_0,z_0) +2t\Delta \phi(t)|^\frac{1}{2}\big)\\
&= (| \hat w(t)| + \epsilon_t)| t|+ \epsilon_t\big| 4\int\limits_0^t\Delta \phi(s)\,ds -2t\Delta \phi(t)\big|^\frac{1}{2}&\\
& \le (C+\epsilon_t)| t|+ 2\epsilon_t\bigg( | \int\limits_0^t\Delta \phi(s)\,ds|^\frac{1}{2}+| t\Delta \phi(t)|^\frac{1}{2}\bigg),
\end{align*}
où $C=\max\limits_{t\in[-\delta, \delta]}| \hat w(t)|<\infty$ vu la continuité de $\hat w(t)$, et (quitte à raccourcir $\delta$) $1\ge\epsilon_t\to 0$ lorsque $t\to0$ .
Pour $T\in[0, \delta]$ on introduit la fonction maximale : $M(T)=\max\limits_{t\in[0,T]}| \Delta\phi(t)|$. D'après l'estimation ci-dessus  
\begin{multline*}
M(T)\le \max\limits_{t\in[0,T]}\bigg\{ \big(C+\epsilon(t)\big)t+2\epsilon(t)\Big(|\int\limits_0^t \Delta \phi(s)\,ds|^\frac{1}{2}+|\Delta\phi(t) t|^\frac{1}{2}\Big) \bigg\} \le \big(C+1\big)T+4\big(TM(T))^\frac{1}{2}.
 \end{multline*}
 La résolution de cette inégalité élémentaire sur $M(T)$ donne l'estimation 
$M(T)\le cT, \quad c=c(C)<\infty$,
et en particulier, $\Delta\phi(T)\le  cT$.
Le même raisonnement appliqué au $T<0$ permet de conclure que 
$$ \Delta\phi(t)=O(|t|) \text{ lorsque } t\to 0.$$ 
En réinjectant cela dans \eqref{Whytney} on montre la différentiabilité de $\phi\circ\gamma$ en $0$: 
$$ \Delta \phi(t)-\hat w(t)t=o(|t|).$$
Comme $\gamma(t+t_0)$ est une solution de \eqref{cauchyproblem} issue de $\gamma(t_0)$, la courbe $\phi\circ\gamma$ est continûment dérivable sur tout $I_\gamma$ et 
$$\big( \phi\circ\gamma(t)\big)' =\hat w(t). $$
Maintenant, on voit que la courbe $I_\gamma\ni t\to \Phi \circ \gamma(t)\in\h$ est de classe $C^1$ et horizontale. En particulier, elle est continûment horizontalement différentiable \cite{vquasi}, son  vecteur tangent égale 
$$\left(-\frac{Yf}{Xf}X+Y\right)\circ\Phi\circ\gamma(t),\quad t\in I_\gamma. \qedhere$$
\end{proof}

\begin{notation*} Pour un sous-ensemble $E\Subset \Omega$ on note
\[  \omega^\phi_{1/2}(E):=\sup_{\substack{A,A'\in E;\\ A\not=A'}} \frac{| \phi(A)-\phi(A')|}{\|A -A'\|^\frac{1}{2}}; \quad \omega^w(E):=\sup_{A,A'\in E}| w(A)-w(A')|. \]
\end{notation*}
Les fonctions $\omega^\phi_{1/2}$ et $\omega^w$ sont finies, en outre d'après \eqref{hoderdemi} $\omega^\phi_{1/2}(E)\to 0$  et  $\omega^w(E)\to 0$ uniformément pour les sous-ensembles $E$ de chaque partie compacte de $\Omega$ lorsque $\diam E\to 0$.
\begin{lemma}[Divergence des courbes intégrales]\label{divergence} On considère $\gamma_i(t)=(y_i(t),z_i(t))$, $t\in(-\delta,\delta)$, deux courbes intégrales de $W^\phi$, $i=1,2$. Alors  pour tout $t\in (-\delta,\delta)$ on a que
\begin{equation}
| z_2(t)-z_1(t)|^\frac{1}{2}\le \big\{| z_2(0)-z_1(0)|+C(t)| y_2(0)-y_1(0)|^\frac{1}{2}t\big\}^\frac{1}{2}+C(t)|t|,
\end{equation}
où $C(t)=4\sup\limits_{s\in [0,t]}\omega_{1/2}^\phi\big(\gamma_1(s)\cup \gamma_2(s)\big)$.
\end{lemma}
\begin{proof}
On déduit de l'équation intégrale que pour tout $0\le t<\delta$ 
\begin{multline*}
| z_2(t)-z_1(t)|\le |z_2(0)-z_1(0)|+4\int\limits_0^{t} | \phi(y_2(0)+s,z_2(s))-\phi(y_1(0)+s,z_1(s)) |\,ds\le \\ 
\le |z_2(0)-z_1(0)|+C(t)\Big(|y_2(0)-y_1(0)|^\frac{1}{2}t + \int\limits_0^{t} | z_2(s)-z_1(s) |^\frac{1}{2}\,ds\Big).
\end{multline*}
L'estimation énoncée découle maintenant d'une inégalité non-linéaire de Grönwall  (voir \cite[Th. 2.1]{gronwall}). Le même raisonnement s'applique aux $t$ négatifs. 
\end{proof}

\subsection{Description géométrique de la métrique $d_\infty$ induite sur  $\S$ }\label{geomdescrp}
 Pour deux points notés $A=(y_1,z_1), B=(y_2,z_2)\in \Omega$ on fixe $\gamma_1(t)=(y_1+t, z_1(t))$ une courbe intégrale quelconque de $W^\phi$ issue en $t=0$ du point $A=(y_1,z_1)$.  Supposons que $[0,y_2-y_1]\subset I_{\gamma_1}$; on introduit le point $B'=\gamma_1(y_2-y_1)$, $\hat z_1=z_1(y_2-y_1)$. 
A partir de la norme homogène $\rho$ sur $\h$ nous définissons alors la fonction $d_{g,\rho}(A,B)$:
$$d_{g,\rho}(A,B):=\rho\big(w(A)(y_2-y_1), y_2-y_1, z_2-\hat z_1\big).$$
En prenant, par exemple, la norme homogène $\rho_\infty$ on obtiendra 
$$d_{g,\infty}(A,B)= \max\big\{|y_2-y_1|\sqrt{1+w(A)^2}, |z_2-\hat z_1|^\frac{1}{2}\big\}.$$

Pour $\Omega' \subset \Omega$ notons  par $D(\Omega')\subset \Omega' \times \Omega'$ l'ensemble des points où $d_{g,\rho}$ peut être définie.  Remarquons que pour toute partie compacte de $\Omega'\Subset \Omega$ la fonction $d_{g,\rho}(A,B)$ est définie pourvu que $A\in \Omega'$ et $d_\infty(A,B)\le \delta(\Omega')$, où $\delta(\Omega')=\delta\big(\operatorname{dist}(\Omega',\partial \Omega)\big)>0$. 
La fonction $d_{g,\rho}$ n'est pas symétrique et sa définition dépend du choix d'une courbe $\gamma_1$. Malgré cela $d_{g,\rho}$ est équivalente à la distance homogène $d_\rho\corn \S$, notée $\tilde d_\rho(A,B)=\rho(\Phi(A)^{-1}\Phi(B))$, induite sur $\S$ dans le sens suivant. 
\begin{lemma}\label{metriqueR}
Pour tout $A,B\in D(\Omega')$, $\Omega'\Subset \Omega$,
$$ | d_{g,\rho}(A,B)- \tilde d_\rho(A,B) | =o(d_{g,\rho}(A,B)),$$
où le petit-$o$ est uniforme lorsque $A\to B$.
\end{lemma}
\begin{proof} On notera dans la démonstration qui suit : $$\epsilon^w:=\omega^w\big(\gamma_1[0,y_2-y_1]\big), \quad \epsilon^\phi:=\omega_{1/2}^\phi(B\cup B').$$  
Il est facile de vérifier  que $\epsilon^w\to 0$ et $\epsilon^\phi\to 0$ lorsque $B\to A$ uniformément sur $D(\Omega')$.

D'abord nous obtenons l'estimation de la divergence $\Delta_z$ des composantes verticales de $\tilde d_\rho$ et $d_{g,\rho}$. On a que 
$$\Delta_z:= \big|  (z_2-\hat z_1) - ( z_2-z_1+2(y_2-y_1)(\phi(B)+\phi(A)))\big|=|\hat z_1-z_1+2(y_2-y_1)(\phi(B)+\phi(A))|.$$
On rappelle que $z_1(t)$  satisfait  $z_1(t)=z_1-4\int_0^t \phi(s+y_1,z_1(s)) \,ds$. 
On fait apparaître le terme $\phi(B')=\phi(y_2,z_1(y_2-y_1))$:
\begin{align*}
|\hat z_1-z_1+2(y_2-y_1)(\phi(B)+\phi(A))|&\le 2 |\int\limits_0^{y_2-y_1} 2\phi(s+y_1,z_1(s))-\phi(A)-\phi(B') \,ds |
\\&+2|(y_2-y_1)\big(\phi(B')-\phi(B)\big)|=:2\Delta_1+2\Delta_2.
\end{align*}
Comme $\phi(\gamma_1(\cdot))$ est de classe $C^1$, suivant la formule d'accroissements finis le premier terme est majoré par (avec $\xi_1\in [0,s]$ et $\xi_2\in[s,y_2-y_1]$)
\begin{align*}
\Delta_1&\le |\int\limits_0^{y_2-y_1} w(\xi_1+y_1,z_1(\xi_1))(s-y_1)+w(\xi_2+y_1,z_1(\xi_2))(y_2-s) \,ds |
\\ &\le  \epsilon^\omega\int\limits_0^{y_2-y_1} |s|+|y_2-y_1|\,ds\le 2 \epsilon^w |y_2-y_1|^2.
\end{align*}
Pour le deuxième, 
\begin{equation*}
\Delta_2\le \epsilon^\phi |y_2-y_1||z_2-\hat z_1|^\frac{1}{2}\le 2 \epsilon^\phi\big(|y_2-y_1|^2+|z_2-\hat z_1|\big).
\end{equation*}
Finalement, en additionnant on arrive à $\Delta_z \le 8\left( \epsilon^w+ 2 \epsilon^\phi\right)d_g(A,B)^2$.

Maintenant on fait l'estimation de la divergence des composantes en $x$ :
\begin{align*}
\Delta_x:&=\big|(\phi(A)-\phi(B))-w(A)(y_2-y_1)\big|\le |\int\limits_0^{y_2-y_1} w(\gamma_1(s))\,ds-w(A)(y_2-y_1)
\\&+|\phi(B')-\phi(B)|\le \epsilon^w|y_2-y_1|+\epsilon^\phi|\hat z_2-z_1|^\frac{1}{2}\le 2(\epsilon^\phi+\epsilon^w)d_g(A,B). 
\end{align*}
Avec les estimations de $\Delta_z$ et $\Delta_x$ et compte tenu de la continuité et l'homogénéité de $\rho$, la conclusion s'obtient aisément.
\end{proof}

On remarque que $\diam_{g,\rho}E=\sup\{d_{g,\rho}(A,B)\mid A,B\in E\}$ est défini pour tout sous-ensemble $E$ suffisamment petit dans $\Omega'\Subset \Omega$. Ainsi, nous pouvons définir les mesures de Hausdorff bâties sur $\diam_{g,\rho}$ pour tout ensemble $E\Subset \Omega$. Une conséquence immédiate du lemme \ref{metriqueR} est   
\begin{corollaire}
Les mesures de Hausdorff construites à partir de $d_\rho$ et $d_{g,\rho}$ coïncident (en toute dimension) sur les ensembles compacts de $\S$.
\end{corollaire}

Fixons $\Omega'\Subset \Omega$. Pour $A=(y_0,z_0)\in \Omega'$ et $0<r\le \delta(\Omega')$ on définit l'ensemble  $B_{d_{g,\infty}}(A,r)$ de la façon suivante:
\begin{equation*}
B_{d_{g,\infty}}(A,r)=\big\{(y,z(y))\in \Omega \mid | y-y_0|<\frac{r}{\sqrt{1+w(A)^2}},\, z_-(y)<z(y)<z_+(y)\big\},
\end{equation*}
 où $\big(t+y_0,z_-(t)\big)$ et $\big(t+y_0,z_+(t)\big)$ sont des courbes intégrales de $W^\phi$ issues en $t=0$ respectivement de  $(y_0,z_0-r^2)$ et de $(y_0,z_0+r^2)$.  L'ensemble $B_{d_{g,\infty}}$ joue le rôle d'une boule en "métrique" $d_{g,\infty}$; voir fig. \ref{rmetric}.
  
 \begin{corollaire}
 Il existe $C_r$ défini pour $0<r\le \delta(\Omega')$ tel que  $1\le C_r<C<\infty$, $C_r\to 1$ lorsque $r\to 0$, et pour tout $A\in \Omega'$
 $$ \tilde B_\infty(A,r/C_r)\subset B_{d_{g,\infty}}(A,r)\subset \tilde B_\infty(A,rC_r),$$
où $\tilde B_\infty$ désigne une boule en métrique $d_\infty\corn \S$.
\end{corollaire}

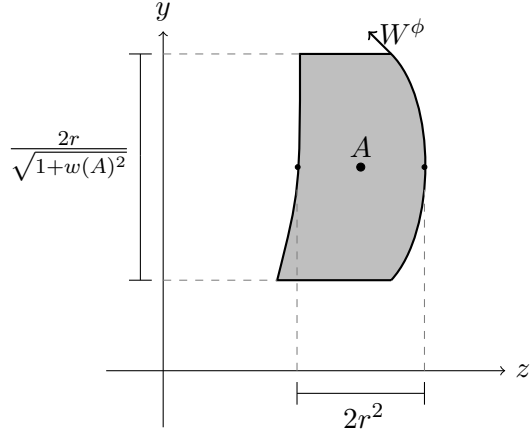
\begin{figure}[h]\centering
\begin{tikzpicture}[scale=3] \label{dessinmetricR}
\draw[->] (-0.25,0) -- (1.5,0) node[right] {$z$};
\draw[->] (0,-0.25) -- (0,1.5) node[above] {$y$};
\filldraw [fill=gray!50,draw=black,thick] (1.0,0.4) .. controls (1.2,0.62) and (1.2,1.2) .. (1.0,1.4) --
(0.6,1.4) --
(0.6,1.4) .. controls (0.6,0.8) .. (0.5,0.4) --
(1.0,0.4);
\filldraw [black] (0.8665,0.9) circle (0.5pt) node[above] {$A$};
\filldraw [black] (1.146,0.9) circle (0.3pt);
\filldraw [black] (0.59,0.9) circle (0.3pt);

\draw [gray,dashed] (0,1.4) -- (0.6,1.4);
\draw [gray,dashed] (0,0.4) -- (0.5,0.4);

\draw [gray,dashed] (0.587,0) -- (0.587,0.8);
\draw [gray,dashed] (1.146,0) -- (1.146,0.8);

\draw (0.88,-0.1) node[below] {$\footnotesize{2r^2}$};
\draw (0.587,-0.1)--(1.146,-0.1);
\draw (0.587,-0.15)--(0.587,-0.05);
\draw (1.146,-0.15)--(1.146,-0.05);

\draw (-0.1,0.4)--(-0.1,0.95) node[left] {$\frac{2r}{\sqrt{1+w(A)^2}}$};
\draw (-0.1,0.95)--(-0.1,1.4);
\draw (-0.15,0.4)--(-0.05,0.4);
\draw (-0.15,1.4)--(-0.05,1.4);
\draw [->,thick] (1.0,1.4) -- (0.9,1.5) node[right]{$W^{\phi}$};
\end{tikzpicture}
\caption{Ensemble $B_{d_{g,\infty}}(A,r)$}\label{rmetric}
\end{figure}
Le résultat suivant est bien connu \cite{Heishypesurfaces} pour les hypersurfaces régulières sur les groupes de Carnot généraux. Nous en donnons la preuve à titre d'application de la technique développée au lemme \ref{metriqueR}.  
\begin{lemma}[Formule de l'aire pour $\S$]\label{aireformsurf}  La mesure sphérique de Hausdorff $\S^3_{d_\rho}\corn \S$ est absolument continue par rapport à la mesure de Lebesgue  $\mathcal{L}^2\corn \Omega$ transportée sur $\S$ via l'application $\Phi$, \ie $\S^3_{d_\rho}\corn\S\ll \Phi_\#\big( \mathcal{L}^2\corn \Omega\big)$. Plus précisément, si $d_\rho$ une métrique sur $\h$, on a que
\begin{equation}\label{airesurface}
\S^3_{d_\rho}\corn \S= \Phi_\#\big( J_\rho(\cdot)\mathcal{L}^2\corn \Omega\big),\ J_\rho(A)^{-1}=
2^{-3}\mathcal{L}^2\{(y,z)\in \r^2\mid \rho(w(A)y,y,z)< 1\}.
\end{equation}
\end{lemma}

\begin{remark*}
Si $d_\rho=d_\infty$, alors $J_\infty(A)=2\sqrt{1+w(A)^2}$.
\end{remark*}

\begin{proof}
Nous utilisons le théorème classique de différentiabilité de la mesure sphérique \cite[2.10.17, 2.10.18]{federer}, adapté au contexte de nos considérations:
\begin{theorem*}
Soit $\mu$ une mesure Borel régulière sur $\h$ et $\alpha>0$. Si 
\begin{equation*}
 \lim\limits_{r\to 0+} \frac{\mu(B_\rho(A,r))}{(\diam_\rho B_\rho(A,r))^\alpha}=s(A) \quad \mu-p.p. \quad \text{alors} \quad\mu=s(\cdot)\S_{d_\rho}^\alpha.
\end{equation*}
\end{theorem*}

Posons $\mu:=\Phi_{\#}(\mathcal{L}^2 \corn\Omega)$. Pour arriver à l'énoncé il suffit alors de montrer l'asymptotique suivante pour tout $A\in \Omega$:
$$\lim\limits_{r\to 0}r^{-3}\mathcal{L}^2\{B\in \Omega \mid \tilde d_{\rho}(A,B)< r\}=\mathcal{L}^2\{(y,z)\in \r^2\mid \rho(w(A)y,y,z)<1\}.$$
Montrons-le en remplaçant $\tilde d_\rho$ par $d_{g,\rho}$ et, par conséquent, d'après le lemme \ref{metriqueR} il sera de même pour $\tilde d_\rho$.
Pour calculer $\mathcal{L}^2\{B\in \Omega \mid d_{g,\rho}(A,B)< r\}$ on utilise le changement de variables (on revient aux notations du lemme \ref{metriqueR}) :
\begin{equation*}(y,z) \stackrel{P}\longrightarrow (y_2,z_2) :\quad
\left\{\begin{array}{l}
y_2=y_1+ry,\\
z_2=z_1+r^2z-4\phi(A)ry-2w(A)r^2y^2.
\end{array}\right.
\end{equation*}
Dans les nouvelles coordonnées on aura
\begin{align*}
z_2-\hat z^1&=r^2z-4\phi(A)ry-2w(A)r^2y^2+4\int\limits_0^{ry} \phi(\gamma_1(s))\,ds \\
&=r^2z+4\int\limits_0^{ry} \phi(\gamma_1(s))-\phi(A)-w(A)s\,ds=r^2z+o((ry)^2).
\end{align*}
On en déduit maintenant le résultat voulu par  
\begin{align*}
&\mathcal{L}^2\{B\in \Omega \mid d_{g,\rho}(A,B)< r\}=\mathcal{L}^2\{(y_2,z_2) \mid \rho\big(w(A)(y_2-y_1),y_2-y_1, z_2-\hat z_1\big)< r\}
\\&=|\det(DP)|\mathcal{L}^2\{(y,z) \mid \rho\big(w(A)ry ,ry, r^2z+o((ry)^2)\big)< r\}
\\&=r^3\mathcal{L}^2\{(y,z) \mid \rho\big(w(A)y ,y, z+o(1))\big)< 1\}=r^3\mathcal{L}^2\{(y,z) \mid \rho\big(w(A)y ,y, z)\big)< 1\}+o(r^3).
\end{align*}
\end{proof}

\section{Courbes verticales dans le groupe d'Heisenberg.}\label{courbesvetricales}
 
Nous considérons $F\in C^1_H(\h,\r^2)$ avec une condition de normalisation $F(0)=0$. Dans tout ce qui suit nous supposons que $D_hF(0)$ est surjective. Le noyau de la différentielle horizontale est donc l'axe vertical (le centre de $\h$) : $\Ker D_hF(0) = Oz=\{x=0, y=0\}$. On note alors $d_hF(A):\r^2\to \r^2$ la partie horizontale de la différentielle $D_hF(A)$, $d_hF(A)(x,y)=D_hF(A)(x,y,0)$. 

A partir de maintenant notre but est d'étudier les propriétés métriques de l'ensemble compact non-vide $\f$  vu localement au voisinage de $0\in \h$.
\begin{remark}
Compte tenu de la non-intégrabilité de la distribution horizontale, il est facile de voir que le noyau $\Ker D_hF(0)=Oz$ n'admet pas de sous-groupe complémentaire. Nous ne sommes donc plus dans le cadre du théorème \ref{imptheor}, et comme nous le verrons aussi par la suite $\f$ n'est pas "une sous-variété régulière sous-riemannienne" au sens de \cite{regularsubmanifoldsheisenberg}.
\end{remark}

\subsection{Considérations préliminaires}

\begin{proposition} Si la différentielle  $D_hF(0)$ de l'application $F\in C_H^1(\h;\r^2)$ est surjective, alors l'application $F$ est surjective  sur un certain voisinage de $F(0)$.  
 \end{proposition}
 \begin{proof}(Méthode de Newton sur les plans horizontaux)
Pour $a\in \r^2$ on définit deux suites récurrentes $\{A_n\}$ et $\{a_n\}$ par les formules
 \[\left\{\begin{aligned} & A_0=0; \\ 
& A_{n+1}=\exp\Big([d_hF \big(A_n \big)]^{-1}(a-a_n)\Big)(A_n), \end{aligned}\right. 
\left\{\begin{aligned} & a_0=F(0); \\ 
& a_{n+1}=F(A_{n+1}). \end{aligned}\right.\]
Pour tout $a$ pris sur un certain voisinage compact de $F(0)$ on vérifie par le principe des applications contraignantes (avec l'estimée \eqref{lagrange}) qu'il existe $A^\ast$ tel que  $A_n\to A^\ast$ et $F(A^\ast)=a$.
 \end{proof}

\paragraph{Condition de Whitney.}
Vu que  la différentielle horizontale $D_hF$ est continue et surjective en $0$, il existe un voisinage compact $U:=\bar B_\infty(0,R)\Subset \h$, $R>0$, de $0$ tel que $D_hF(A)$ est surjective pour tout $A\in U$. Ainsi, la valeur $\|D_hF(A)(B)\|=\|d_hF(\pi(B))\|$ est équivalente à la norme euclidienne $\|\pi(B)\|$ uniformément pour $A\in U$. 

Soit $A,B\in \f\cap U$. Nous écrivons la définition de la différentiabilité horizontale de $F$ lorsque $B\to A$ 
\begin{equation*}
0=F(A)-F(B)=D_hF(A)\big(A^{-1}B\big)+o(d_\infty(A,B)).
\end{equation*}
  On en déduit l'estimation (qu'on appellera de \emph{Whitney}) 
\begin{equation}\label{whitneycond}
\| \pi(B)-\pi(A)\|=o(|z(A^{-1}B)|^{\frac{1}{2}}), \text{ lorsque } B\to A,
 \end{equation}
où  petit-$o$ est uniforme pour $A,B\in \f\cap U$ et $z(\cdot)$ désigne la coordonnée selon $z$. 

\begin{remark}\label{Whit} Si pour toute partie compacte d'un ensemble fermé $E$ l'estimation \eqref{whitneycond} est vérifiée uniformément, alors il existe une fonction $F\in C_H^1(\h,\r^2)$  telle que $E\subset \f$ et $D_hF(A)$ est surjective quel que soit $A\in E$. En effet, il suffit de poser $D_hF(A)(x,y,0)=(x,y)$ pour tout $A\in E$ et appliquer le théorème \ref{Whitneytheorem} de prolongement de Whitney.
\end{remark}

\subsubsection{Calcules de $\H^2_\infty$ de $\f$ pour $F\in C^1(\r^3,\r^2)$.}
   Si $XF$ et $YF$ sont de classe $C^1$, alors $ZF=-\frac{1}{4}[X,Y]F$ est continue. Cela signifie que $F=(f_1,f_2)\in C^1(\r^3,\r^2)$ et d'après le théorème classique des fonctions implicites localement l'ensemble $\f$ est une courbe simple représentée comme 
\begin{equation}\label{lisserepresantaion}
\f\cap V=\big\{ \Gamma(z)=(\gamma(z),z)\mid z\in[-\delta,\delta]\big\},
\end{equation}
où $\gamma$ est de classe $C^1$ et $V$ est un voisinage de $0\in \r^3$.
Quitte à diminuer $\delta$, le paramétrage $\{z\to\Gamma(z)\}$ est bi-lipschitzien de $\big([-\delta,\delta], |\cdot|^\frac{1}{2}\big)$ dans  $(\h, d_\infty)$. En plus, compte tenu de la condition \eqref{whitneycond} la "densité" est égale à
\begin{equation*}
J(z)=\lim\limits_{t\to0+} t^{-1}d_\infty\bigl(\Gamma(z),\Gamma(z+t)\bigr)^2=1+2\lim\limits_{t\to0+} t^{-1}\det(\gamma (z),\gamma(z+t)) =
1+2\gamma'_y\gamma_x-2\gamma'_x\gamma_y.
\end{equation*}
Remarquons que la courbe $\Gamma\in C^1$ n'est jamais tangente à la distribution horizontale car $d_hF$ est injective.  
Par conséquent, la fonction continue $J(\cdot)$ ne s'annule pas. D'après \cite[2.10.10, 2.10.11]{federer}  $\H^2_\infty\corn \Gamma = \Gamma_\#\big(J(\cdot) \H^2_{1/2}\big)$ , où $\H^2_{1/2}$ est une mesure de Hausdorff sur $[-\delta,\delta]$ par rapport à la métrique $|\cdot|^\frac{1}{2}$. Par ailleurs,  $\H^2_{1/2}=\mathcal{L}^1$, et donc 
$$\H^2_\infty(\Gamma)=\int\limits_{-\delta}^{\delta}1+2\gamma'_y\gamma_x-2\gamma'_x\gamma_y\, dz=\int\limits_\Gamma dz-2(y\,dx-x\,dy). $$

Calculons maintenant la densité $J(\cdot)$  plus explicitement. La relation $\frac{d}{dz}F\big(\gamma(z),z\big)=(XF)\gamma'_x + (YF)\gamma'_y +(ZF) J=0$ conduit au système linéaire:
  \[
\left(
\begin{array}{ccc}
Xf_1  & Yf_1      \\
Xf_2  & Yf_2    
\end{array}
\right)
\left(
\begin{array}{ccc}
  \gamma'_x   \\
  \gamma'_y
\end{array}
\right)=
-J
\left(
\begin{array}{ccc}
  Zf_1  \\
  Zf_2
\end{array}
\right)
.\]
L'hypothèse de la surjectivité de $D_hF$ nous permet d'extraire $J$ de ce système: 
 \begin{align*}
\left(\begin{array}{ccc}
  \gamma'_x   \\
  \gamma'_y
\end{array}\right)&=
-J(d_hF)^{-1}\left(\begin{array}{ccc}
  Zf_1  \\
  Zf_2
\end{array}\right),\quad
 J=1-2J \left[ \gamma_y(d_hF)_1^{-1}\left(\begin{array}{ccc}
  Zf_1  \\
  Zf_2
\end{array}\right)- \gamma_x(d_hF)_2^{-1}\left(\begin{array}{ccc}
  Zf_1  \\
  Zf_2
\end{array}\right)\right],\\ 
 \intertext{ et finalement,} 
J&=\frac{1}{1+2(\star)},  \text{ où }\quad \star=\frac{Zf_2(\gamma_x Xf_2+\gamma_y Yf_2)-Zf_1(\gamma_x Xf_1+\gamma_y Yf_1)}{Xf_1Yf_2-Xf_2Yf_1}.
\end{align*}
 
Une courbe lisse simple $\Gamma$ qui n'est jamais tangente à la distribution horizontale, quitte à faire une translation, admet une représentation locale \eqref{lisserepresantaion}.  En appliquant le même raisonnement que ci-dessus, nous pouvons montrer donc que pour une telle courbe $\Gamma$  sa mesure $\H_\infty^2$ est donnée comme l'intégrale le long de $\Gamma$ de la forme de contact (voir \cite{jean}):
 $$ \H^2_\infty(\Gamma)=\int\limits_\Gamma dz-2(y\,dx-x\,dy).$$
 L'orientation de $\Gamma$ est choisie ici de façon à ce que la forme de contact soit positive sur $\Gamma$. C'est une généralisation  de cette formule que nous obtiendrons par la suite. Notons également l'interprétation géométrique remarquable de la mesure $\H_\infty^2(\Gamma)$. Par le théorème de Green, elle est égale à la différence entre l'accroissement de la coordonnée $z$ le long de $\Gamma$ et quatre fois l'aire (orientée) balayée par la projection $\pi(\Gamma)$ dans le plan horizontal (voir fig. \ref{intergeom}). 

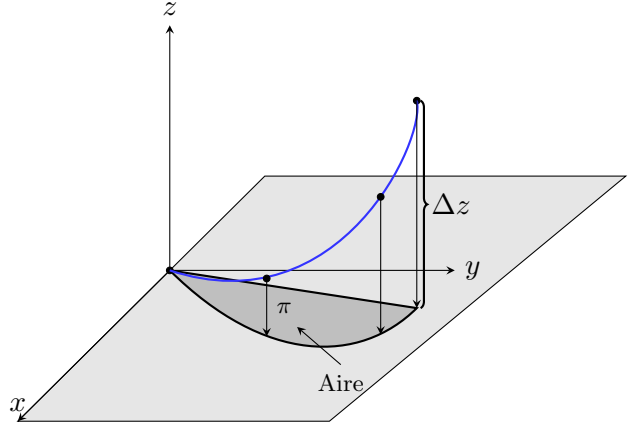
\begin{figure}[h]\centering
\begin{tikzpicture}[scale=2.5, >=stealth]
\filldraw [fill=gray!20] (-0.8, -0.8) -- (0.5,0.5) -- (2.4,0.5) -- (0.84,-0.8) -- cycle;
\draw[->] (0,0) -- (1.5,0) node[right] {$y$};
\draw[->] (0,0) -- (0,1.3) node[above] {$z$};
\draw[->] (0,0) -- (-0.8,-0.8) node[above] {$x$};
\filldraw [fill=gray!50,draw=black,thick] (0,0) .. controls (0.5,-0.5) and (1.0,-0.5) .. (1.3,-0.2) -- (0,0);
\filldraw [black] (0,0) circle (0.5pt);
\filldraw [black] (1.3,0.9) circle (0.5pt);
\draw [draw=blue!80,thick] (0,0) .. controls (0.9,-0.3) and (1.35,0.7) .. (1.3,0.9);
\draw [->](1.3,0.9) -- (1.3,-0.2);
\filldraw [black] (1.11,0.39) circle (0.5pt);
\draw [->](1.11,0.39) -- (1.11,-0.34);
\filldraw [black] (0.51,-0.043) circle (0.5pt);
\draw [->](0.51,-0.045) -- node[right,font=\footnotesize]{$\pi$}(0.51,-0.35);
\draw [<-](0.67,-0.3) -- (0.9,-0.5) node[below,font=\footnotesize]{Aire} ;
\draw[snake=brace,thick] (1.32,0.9) -- node[right]{$\Delta z$} (1.32,-0.2);
\end{tikzpicture}
\caption{Interprétation géométrique de $\H^2_\infty$} \label{intergeom}
\end{figure}

\subsection{Ensembles verticaux Reifenberg plats} 
Ici nous démontrons qu'un ensemble de niveau $\f$ est exactement l'ensemble $\varepsilon$-plat  au sens de Reifenberg avec $\varepsilon\to 0$ quand l'échelle diminue, pour lequel l'axe vertical $Oz$ joue le rôle du plan approximatif en tout point et à toute échelle. Nous en déduisons par la suite que $\f$ est localement une courbe simple. 

 \begin{proposition}\label{injplanhorizontal}
Il existe $r>0$ tel que quels que soient $A\in \f\cap U$ et $B\in h_{A,r}\cap \f\cap U$, 
$$\|F(B)\|\ge c\| \pi(B)-\pi(A)\|, $$
 où $h_{A,r}=\{A+tX(A)+sY(A) \mid (t,s) \in \bar D(0,r)\subset\r^2 \}$ est la boule fermée de rayon $r$ sur le plan horizontal  attaché à $A$ et $c>0$ est une constante. En particulier, pour tout $A\in \f\cap U$ l'ensemble $h_{A,r} \cap \f\cap U$ ne contient que le seul point $A$.
 \end{proposition}
 \begin{proof}
 Par définition, lorsque $h_{A,1}\ni B \to A$
 $$ F(A)-F(B)=D_hF(A)\big(A^{-1}B\big)+o\big(\rho_\infty(A^{-1}B)\big), $$
 où le petit-$o$ est uniforme pour $A\in \f\cap U$ et $B\in h_{A,1}$. Vu que la différentielle continue $d_hF$ est inversible, l'inégalité suivante est vérifiée 
$$\| D_hF(A)\big(A^{-1}B\big)\|\ge 2c\| \pi(B)-\pi(A)\|$$
pour tout $A\in U$ et tout $B\in h_{A,1}$, $c>0$. Comme $A^{-1}B$ est horizontal (\ie $z(A^{-1}B)=0$), il existe $r>0$ tel que $\| o\big(\rho_\infty(A^{-1}B) \big)\|\le c\| \pi(B)-\pi(A)\|$ pour tout $A\in \f\cap U$ et tout $B\in h_{A,r}$. Il en découle le résultat énoncé. 
\end{proof}

\begin{lemma}[projection de $\f$ sur l'axe vertical] \label{sujOz} Il existe $r_0>0$ et $\delta>0$ tels que pour tout $z\in[-\delta,\delta]$ on peut trouver un point $\gamma(z)\in \r^2$ (pas forcement unique), $\| \gamma(z)\|<r_0$, tel que le point $\big(\gamma(z),z\big)$ appartient à $\f$. 
\end{lemma} 
\begin{proof}
Notons $ \tilde F_z(x,y)=F(x,y,z)$ --- la restriction de $F$ sur les plans $z=const$. Pour tout $z$ l'application $ \tilde F_z\in C^0(\r^2;\r^2)$, de plus,
$\tilde F_z$ est différentiable (au sens usuel) en $0\in \r^2$ et $d\tilde F_z(0) =D_hF(0,z)$ est inversible. 

La proposition \ref{injplanhorizontal} garantit l'existence du rayon $r_0>0$ tel que pour tout $B$ dans un disque fermé  $\bar D(0,r_0)\subset \r^2$ l'estimation $ \|\tilde F_0(B)\| \ge C\| B\|$, $C>0$, est vérifiée.  En particulier, l'image du cercle $\partial D(0,r_0)$ par $\tilde F_0$ ne contient pas l'origine. Compte tenu de la continuité de $F$, il existe $\delta>0$ suffisamment petit tel que $0\not \in\tilde F_z(\partial D(0, r_0))$ pour tout $z\in I=[-\delta,\delta]$.  Il est clair que chaque $\tilde F_z \in C^0\big( D(0,r_0),\r^2\big)$ est homotope à $\tilde F_0 \in C^0\big(D(0,r_0),\r^2\big)$ par l'application $F$. On rappelle que le degré $\deg\big(\tilde F_z, D(0,r_0), 0\big)$ est un invariant homotopique quel que soit $z\in I$, car  $0\not \in \tilde F_z\big(\partial D(0,r_0)\big)$ pour $z\in I$. Vérifions par la définition de degré d'application continue que le degré $\deg\big(\tilde F_0, D(0,r_0), 0\big)$ appartient à $\{-1;1\}$. On choisit une approximation standard $F_\epsilon=F\ast h_\epsilon -F\ast h_\epsilon(0) \in C^\infty$, $F_\epsilon(0)=0$, de $F$. La convergence uniforme  $F_\epsilon\to F$, $XF_\epsilon \to XF$ et $YF_\epsilon \to YF$, a lieu sur tout compact (voir \cite{Heishypesurfaces}, par exemple). Comme $d\tilde F_0(0)$ est injective, la différentielle $d(\widetilde F_\epsilon)_0(0)$ l'est aussi pour $\epsilon$ assez petit. De même, si $\epsilon<\epsilon_0$ alors $ \|(\widetilde F_\epsilon)_0(B)\|\ge \frac{C}{2}\|B\|$  pour tout $B\in \bar D(0,r_0)$.  Comme $F^{-1}_\epsilon(0)\cap \bar D(0,r_0)=0$  on a bien   $\deg\big(\tilde F_0, D(0,r_0), 0\big)= \operatorname{sign} \det d (\widetilde F_\epsilon)_0(0)\not=0$.
 Comme  pour $z\in I$ le degré $\deg\big(\tilde F_z, D(0,r_0), 0\big)$ est différent de zéro,  pour chaque $z\in I$ il existe un point $\gamma(z)\in D(0,r_0)$ tel que $\tilde F_z(\gamma(z))=F(\gamma(z),z)=0$ (voir \cite{degree}).
\end{proof}

\begin{exemple}[Non-unicité de l'intersection $\f$ avec $z=const$]\label{nonunique}
Con\-si\-dé\-rons l'ensemble compact $\mathcal{A}=\{0\} \cup \{A_n\}_{n=1}^\infty \cup \{B_k\}_{k=1}^\infty$, où  $A_n=(\frac{1}{n},0,\frac{1}{n})$ et $B_k=(\frac{1}{k},-\frac{1}{k^2},\frac{1}{k})$. On vérifiera par des calculs élémentaires les hypothèses de la remarque \ref{Whit}, ce qui donnera l'existence d'une application $F\in C^1_H(\h,\r^2)$ avec $D_hF(0)$ surjective telle que $\mathcal{A}\subset \f$ . 
\end{exemple}
\begin{proof} 
Comme $\mathcal{A}$ ne possède que le point limite $0$,  on n'a à vérifier que ce qui suit lorsque $n,k\to \infty$ 
$$ \|\pi(B_k)-\pi(A_n)\|=o\bigl(|z(A_n^{-1}B_k)|^{\frac{1}{2}}\bigr) \Longleftrightarrow  \bigl(\mfrac{1}{k}-\mfrac{1}{n}\bigr)^2+\mfrac{1}{k^4}  =o\left(|\mfrac{1}{n}-\mfrac{1}{k}+2\mfrac{1}{k^2}\mfrac{1}{n}|\right).$$
Pour $N$ fixé assez grand, et on cherche à estimer en fonction $k\ge N$
$$\sup\limits_{n\ge N}\frac{(\frac{1}{k}-\frac{1}{n})^2+\frac{1}{k^4}}{|\frac{1}{n}-\frac{1}{k}+2\frac{1}{k^2}\frac{1}{n}|}= \sup\limits_{n\ge N}( b(k,n)+a(k,n)),$$ où on note 
$a(k,n):=nk^{-2}| k(k-n)+2 |^{-1}$ et $b(k,n):=(k-n)^2n^{-1}| k(k-n)+2 |^{-1}$.

Il est facile de voir que $ a(k,n)\le a(k,k)=\frac{k}{2k^2}=\frac{1}{2k}\le \frac{1}{2N}$. 
Pour estimer $b(k,n)$ on pose $g(x)=\frac{(x-k)^2}{x}$ et $h(x)=k(k-x)+2$; on a 
$$\left(\mfrac{g(x)}{h(x)} \right)'= -\mfrac{(x-k)(x(k^2-2)-(k^3+2k))}{x^2(kx-2-k^2)}.$$ 
Les points critiques de $\frac{g(x)}{h(x)}$ sont les suivants (dans l'ordre croissant):
\[
\left\{\begin{aligned}   & x_1=0<N, & x_2=k\ge N,\\
& x_3=k+\mfrac{2}{k}, & x_4=\mfrac{(k^2+2)}{k^2-2}<k+1. 
 \end{aligned}\right. 
\]
En particulier, on voit que $\frac{g(x)}{h(x)}$ est décroissante sur $(0,k)$ et croissante sur $(k+1,\infty)$. On en déduit que $\sup\limits_{n\ge N} b(k,n)\le \max\{b(k,N), b(k, \infty)\}= \max\{b(k,N), \frac{1}{k}\}$. Il suffit donc d'estimer $\sup\limits_{k\ge N} b(k,N)$. Des calculs du même type montrent que la fonction $\frac{(x-N)^2}{N(x(x-N)+2)}$ est croissante pour $x\ge N$ et, donc, $\sup\limits_{k\ge N}b(k,N)=b(\infty,N)=N^{-1}$.
\end{proof}

\begin{Def*} On définit  la distance de Hausdorff $\dist_H$ entre deux sous-ensembles $E_1,E_2\subset \h$  par
\begin{equation*}
\dist_H(E_1,E_2)=\max\big\{\sup\limits_{A\in E_2} d_\infty(A,E_1),\sup\limits_{B\in E_1} d_\infty(B,E_2) \big\}.
\end{equation*}
\end{Def*}
\begin{notation*}
On note  $E_{A,r}:=E\cap \bar B_\infty(A,r)$ pour $E\subset \h$; notons également  $Z_{A}:=\{B\in \h \mid \pi(B)=\pi(A)\}$ et $Z_{A,r}:=Z_{A}\cap \bar B_\infty(A,r)$.  
\end{notation*}

\begin{remark*}
$d_\infty(B,Z_A)=d_\infty(B, \Pi_A(B))= \|\pi(B)-\pi(A)\|$, où  $\Pi_A(B)=A\big(0, z(A^{-1}B)\big)$.
\end{remark*}

\begin{lemma}\label{l: reifenberg}
Il existe un voisinage $\tilde U$ de $0\in \h$ tel que $E=\f$ vérifie 
\begin{equation}\label{platequation}
\dist_H\big(E_{A,r}, Z_{A,r}\big)\le r \varepsilon(r) \text{ pour tout } A\in E\cap \tilde U,
\end{equation}
et $\varepsilon(r)\to 0$ quand $r\to 0$.
\end{lemma}
\begin{proof}
Considérons d'abord $0\in \f$. Si $A\in \f\cap B_\infty(0,r)$, alors $d_\infty(A,Z_{0,r})=\|\pi(A)\|=o(|z(A)|^\frac{1}{2})=o(r)$, lorsque $r\to 0$ d'après la condition \eqref{whitneycond}. Si $A=(0,0,z)\in Z_{0,r}$, alors pour $r\le r_0$ suffisamment petit d'après le lemme \ref{sujOz} on peut trouver $\tilde A=(\gamma(z),z)\in \f$, et donc d'après \eqref{whitneycond} $d_\infty(A,\tilde A)=\|\gamma(z)\|=o(r)$ Ainsi, $\dist_H\big(E_{0,r}, Z_{0,r}\big)=o(r)$.

En suivant la démonstration du lemme \ref{sujOz} on démontre que grâce à la continuité de $D_hF$ on peut choisir un voisinage fermé $\tilde U\subset B_\infty(0,R)$ de $0$ qui vérifie ce qui suit : pour tout $A\in \f \cap \tilde U$ le lemme \ref{sujOz} appliqué à la fonction translatée $F\circ \tau_{A^{-1}}$ a lieu avec $r_0>0$ et $\delta>0$ qui ne dépendent pas de $A$. On remarque maintenant que l'argument ci-dessus s'applique à tout 
$A\in \f\cap\tilde U$ ce qui donne $\dist_H\big(E_{A,r}, Z_{A,r}\big)=o(r)$, où le petit-$o$ est uniforme exactement comme celui de \eqref{whitneycond}.
\end{proof}

\begin{Def*} \emph{Le cône tangent homogène} d'un ensemble $E\subset \h$ en $0$ est un ensemble 
$$ \operatorname{Tan}(E,0)=\big\{A=\lim\limits_{n\to \infty}\delta_{r_n}(A_n) \in \h, \text{ où } A_n\in E \text{ et }r_n \to \infty \big\}.$$
En un point $a\in \h$ quelconque le cône homogène est donné par $\operatorname{Tan}(E,a)=\tau_a\operatorname{Tan}(\tau_{a^{-1}}E,0)$. 
\end{Def*}
Une conséquence facile de la définition et du lemme \ref{l: reifenberg} est
\begin{proposition} Le cône tangent homogène de $\f$ coïncide avec le noyau de la différentielle horizontale :
$$ \operatorname{Tan}\big(\f,0\big)=\Ker D_hF(0)=Oz.$$
\end{proposition}

\begin{proposition}\label{p: reifWhitneysuff }
Supposons qu'un ensemble compact $E\subset \h$ satisfait
\begin{equation*}
d_\infty(B,Z_{A,r})\le\varepsilon(r) r \text{ pour tout } B\in E_{A,r} \text{ et tout } A\in E\cap U,
\end{equation*}
où $\varepsilon(r)\to 0$ quand $r\to 0$ et $U$ est compact.  Alors il existe une fonction $F\in C^1_H(\h, \r^2)$ telle que $F\corn (E\cap U) =0$ et $d_hF\corn (E\cap U)=\operatorname{Id}$. 
\end{proposition}
\begin{proof}
Il suffit de vérifier la condition de Whitney \eqref{whitneycond} pour un compact $E\cap U$ (voir la remarque \ref{Whit}). On prend $A,B\in E\cap U$ assez proche de sorte que  $\varepsilon(r)\le 1$ pour $r:=d_\infty(A,B)$. Il est immédiat que $\|\pi(A)-\pi(B)\|=d_\infty(B,Z_{A,r})\le \varepsilon(r)r$, et donc $d_\infty(A,B)=|z(A^{-1}B)|^{\frac{1}{2}}$, d'où la conclusion. 
\end{proof}

\begin{Def*}\emph{Le cône vertical} pointé en $A\in \h$ d'ouverture $\varepsilon\in (0,1)$ et de rayon $r>0$ est l'ensemble $\mathcal{C}_{r,\varepsilon}(A)=\{ B\in \h \mid d_\infty(B, Z_{A})\le \varepsilon d_\infty(A,B) \}\cap B_\infty(A,r)$. 

Le cône vertical se décompose naturellement en $\mathcal{C}_{r,\varepsilon}(A)=\mathcal{C}^+_{r, \varepsilon}(A)\cup \mathcal{C}^-_{r, \varepsilon}(A)$, où $\mathcal{C}^{\pm}_{r, \varepsilon}(A)=\mathcal{C}_{r,\varepsilon}(A)\cap \{B\in \h \mid z(A^{-1}B)\gtreqqless 0\}$.
\end{Def*}

\begin{figure}[h]\centering
\begin{tikzpicture}[scale=2, >=stealth]
\draw (-1,-0.4) -- ( 1,-0.4)--(1,0.4)--(-1,0.4)--(-1,-0.4);
\draw[fill=gray!30] (0,0) parabola (0.8,0.4)--(1,0.4)--(-1,0.4)--(-0.8,0.4) parabola bend (0,0) (0,0);
\draw[fill=gray!30] (0,0) parabola (0.8,-0.4)--(1,-0.4)--(-1,-0.4)--(-0.8,-0.4) parabola bend (0,0) (0,0);
\draw[->] (-1.4,0) -- (1.4,0) node[right] {$x,y$};
\draw[->] (0,-1) -- (0,1) node[above] {$z$};
\end{tikzpicture}
\caption{Cône vertical $\mathcal{C}_{r,\varepsilon}(0)$  (la zone grise)}
\end{figure}
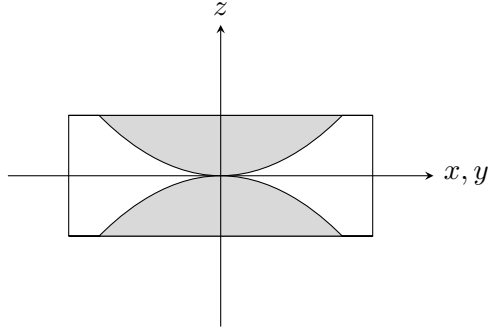

\begin{proposition}\label{diamcone} Soit $O\in \mathcal{C}^{-}_{R, \varepsilon}(A)$ ou de façon équivalente $A\in \mathcal{C}^{+}_{R, \varepsilon}(O)$, $R>0$ quelconque. Alors $\diam_\infty \mathcal{C}^{-}_{R, \varepsilon}(A)\cap  \mathcal{C}^{+}_{R, \varepsilon}(O)\le 2 d_\infty(O,A)(\varepsilon^2+\sqrt{1+\varepsilon^4})$.
\end{proposition}
\begin{proof} Quitte à faire une translation, on peut supposer que $O=0$ et donc $\varepsilon^2 z(A)\ge \|\pi(A)\|^2$. Soit $B\in \mathcal{C}^{-}_{R, \varepsilon}(A)\cap  \mathcal{C}^{+}_{R, \varepsilon}(O)$. Cela veut dire en particulier que 
$ \varepsilon^2 z(B)\ge \|\pi(B)\|^2$ et $\varepsilon^2 \big(z(A)-z(B)+2\det\big(\pi(A),\pi(B)\big)\ge \|\pi(B)-\pi(A)\|^2$. On estime donc
$$
d_\infty(O,B)^2=z(B)\le z(A)+2\det\big(\pi(A),\pi(B)\big)\le z(A)+2\varepsilon^2 \sqrt{z(A)z(B)},
$$
d'où en résolvant une inégalité quadratique, on obtient $d_\infty(O,B)\le z(A)(\varepsilon^2+\sqrt{1+\varepsilon^4})=d_\infty(O,A)(\varepsilon^2+\sqrt{1+\varepsilon^4}).$ L'inégalité triangulaire donne l'énoncé.
\end{proof}
On peut vérifier également que $\diam_\infty \mathcal{C}^{\pm}_{r, \varepsilon}(A)=r\sqrt{1+\varepsilon^4}$.

\begin{Def}\label{d: reifenbegplat}
Un ensemble fermé $E\subset \h$ est dit\textit{ $\varepsilon$-Reifenberg plat} par rapport au sous-groupe vertical $Oz$ dans $U\subset \h$ (à partir de l'échelle $r_0>0$)  si
\begin{equation}\label{parReif}
\dist_H\big(E_{A,r}, Z_{A,r}\big)\le r \varepsilon \text{ pour tout } A\in E\cap U \text{ et tout } 0<r\le r_0.  
\end{equation}
\end{Def}

\begin{theorem}[du paramétrage des ensembles verticaux Reifenberg plats]\label{t: parametrtheorem}
Supposons qu'un ensemble compact $E\subset \h$ est $\varepsilon$-Reifenberg plat par rapport à $Oz$ dans $\bar B_\infty(0,R)$ avec $0\le \varepsilon\le \varepsilon_0$ assez petit. Alors (du point de vue topologique) l'ensemble $E\cap  B_\infty(0,R)$ est localement l'image d'une courbe simple. 
\end{theorem}

\begin{remark*}
On ne cherche pas ici la valeur optimale de $\varepsilon_0$. Dans cette démonstration, on peut prendre tout $\varepsilon_0 < (1+2\sqrt{2})^{-\frac{1}{2}}\approx 0.511\ldots$ (on aura bien $\varepsilon_0^2+\sqrt{1+\varepsilon_0^2+\varepsilon_0^4}< \sqrt{2}$).
\end{remark*}

\begin{proof} Prenons un point $O\in E\cap B_\infty(0,R)$. On pose $r=\frac{1}{4}\min\{ r_0,\, d_\infty(O,\partial B_\infty(0,R))\}$.

En faisant une translation, on se ramène au cas où $O=0$ est l'élément neutre. 
D'après la condition \eqref{parReif}, il existe un point $A\in \mathcal{C}^+_{r,\varepsilon}(O)\cap E$ tel que $d_\infty\big(A, (0, 0, r^2)\big)\le \varepsilon r$. On définit $\Gamma_0(t)=tA$, $t\in [0,1]$,  le segment de droite reliant $O$ et $A$.

On note $r_1=\frac{d_\infty(O,A)}{\sqrt{2}}$. On trouve maintenant un point $B\in E$ ("au milieu" entre $O$ et $A$) tel que 
\begin{equation*}
\left\{\begin{aligned} & B\in \mathcal{C}^+_{r_1, \varepsilon}(O)\cap E, \\
& d_\infty\big(B, (0, 0, r_1^2)\big)\le \varepsilon r_1 \end{aligned}\right. \iff
\left\{\begin{aligned} &  \|\pi(B)\|\le \varepsilon r_1,\\ 
&  r_1^2\ge z(B) \ge (1-\varepsilon^2)r_1^2. \end{aligned}\right.
\end{equation*}
On a aussi $\|\pi(B)-\pi(A)\|\le \varepsilon d_\infty(A,B)$, d'où on déduit que 
\begin{equation*}
d_\infty(A,B)^2=|z(A^{-1}B)|\le |z(B)-z(A)|+2\|\pi(B)-\pi(A)\|\|\pi(B)\|\le r_1^2(1+\varepsilon^2)+2\varepsilon^2 r_1d_\infty(A,B),
\end{equation*}
ce qui donne l'estimée $d_\infty(A,B)\le r_1(\varepsilon^2+\sqrt{1+\varepsilon^2+\varepsilon^4})$. Par conséquent, $$\max\{d_\infty(O,B),d_\infty(B,A)\} \le c(\varepsilon) d_\infty(O,A),$$ où $c(\varepsilon):=2^{-\frac{1}{2}}(\varepsilon^2+\sqrt{1+\varepsilon^2+\varepsilon^4})<1$ pour un choix de $\varepsilon\le  \varepsilon_0$.

 On définit alors
\begin{equation*}
\Gamma_1(t)=\begin{cases} 2tB, & t\in [0, \frac{1}{2}],  \\
\hfill (2t-1)(A-B)+B, & t\in [\frac{1}{2},1], \end{cases} 
\end{equation*}
où il s'agit des opérations linéaires usuelles dans $\r^3$. On obtient facilement 
$$ \max\limits_{t\in [0,1]} d_\infty\big(\Gamma_0(t),\Gamma_1(t)\big)\le C\varepsilon d_\infty(O,A),$$
où $C<\infty$ est une contante absolue qui ne dépend que du choix de la norme homogène $\rho_\infty$.

Maintenant nous construisons récursivement une suite de courbes $\Gamma_n: [0,1]\to \h$, $n\ge 0$, en itérant de façon dyadique la procédure de construction de $\Gamma_1$  à partir de $\Gamma_0$. En particulier, la courbe $\Gamma_n$ est linéaire (au sens de $\r^3$) sur tout intervalle $[\frac{k}{2^n},\frac{k+1}{2^n}]$ pour $k=0,\ldots, 2^n-1$; en plus $\Gamma_n\big(\frac{k}{2^{n}}\big)=\Gamma_{n+1}\big(\frac{k}{2^{n}}\big)\in E$ pour $k=0,\ldots 2^n$. On démontre par récurrence que 
\begin{multline*}
\max_{k=0,\ldots, 2^n-1} d_\infty\left(\Gamma_n\big(\frac{k}{2^{n}}\big),\Gamma_n\big(\frac{k+1}{2^{n}}\big)\right)\le \\ 
c(\varepsilon) \max_{k=0,\ldots, 2^{n-1}-1} d_\infty\left(\Gamma_{n-1}\big(\frac{k}{2^{n-1}}\big),\Gamma_{n-1}\big(\frac{k+1}{2^{n-1}}\big)\right)\le c(\varepsilon)^n d_\infty(O,A).
\end{multline*}  
On obtient alors  
$$
\max\limits_{t\in [0,1]} d_\infty\big(\Gamma_n(t),\Gamma_{n+1}(t)\big)\le C\varepsilon c(\varepsilon)^n d_\infty(O,A).
$$
Par conséquent, la suite $\Gamma_n$ converge uniformément sur $[0,1]$ vers sa limite notée $\Gamma:[0,1]\to \h$.

 Montrons maintenant que $\Gamma([0,1])=E\cap U$, où $U=\mathcal{C}^+_{r_0 ,\varepsilon}(O)\cap \mathcal{C}^-_{r_0,\varepsilon}(A)$;  $O$ et $A$ sont les extrémités de $\Gamma$. Premièrement, on voit que toute courbe $\Gamma_n\subset U$ car, par exemple, $U$ est convexe et pour construire des courbes $\Gamma_n$ on ne prend des points que dans $U$. Remarquons qu'à chaque fois qu'on prend un point $B\in E\cap U$ on obtient (grâce comme toujours à \eqref{parReif}) une décomposition $E\cap U =E\cap (U_1\cup U_2)$, où $U_1= \mathcal{C}^+_{r_0 ,\varepsilon}(O)\cap \mathcal{C}^-_{r_0 ,\varepsilon}(B)$ et $U_2=\mathcal{C}^+_{r_0 ,\varepsilon}(B)\cap \mathcal{C}^-_{r_0 ,\varepsilon}(A)$. On montre donc par récurrence que pour tout $n\ge 0$
$$ E\cap U = E\cap \bigcup\limits_{k=0}^{2^n-1} U_{n,k}, \text { où } U_{n,k}=\mathcal{C}^+_{r_0 ,\varepsilon}\bigg(\Gamma_n\big(\frac{k}{2^n}\big)\bigg)\cap \mathcal{C}^-_{r_0 ,\varepsilon}\bigg(\Gamma_n\big(\frac{k+1}{2^n}\big)\bigg).$$
Or, d'après la proposition  \ref{diamcone}, $\diam_\infty U_{n,k} \le \tilde{C} c(\varepsilon)^n d_\infty(O,A)$, $C<\infty$.  Comme $E$ est compact, on en déduit que $\Gamma_n([0,1]) \xrightarrow[n\to \infty]{} E\cap U$ au sens de la convergence de Hausdorff.
Notons ici que la courbe $\Gamma$ est invective.  En effet, si $0\le t_1<t_2\le 1$, alors pour tout $n$ on peut trouver $k_1$ et $k_2$ tels que $\Gamma(t_1)\in U_{n,k_1}$ et $\Gamma(t_2)\in U_{n,k_2}$, or, $U_{n,k_1}\cap U_{n,k_2}=\varnothing$ dès que $|k_1-k_2|\ge 2$ ce qui est toujours vrai pour $n$ assez grand.
 Ainsi, $E\cap U=\Gamma([0,1])$ est l'image d'une courbe simple. 

En appliquant le même raisonnement que ci-dessus on montre qu'il existe un point $\tilde A\in E\cap \mathcal{C}^-_{r, \varepsilon}(O)$, $-r^2\le z(\tilde A) \le -(1-\varepsilon^2)r^2$, tel que $E\cap  \mathcal{C}^-_{r_0, \varepsilon}(O)\cap  \mathcal{C}^+_{r_0, \varepsilon}(\tilde A)$ est aussi un arc simple. Vu que $E\cap \bar B_\infty(O,r)=E\cap \mathcal{C}_{r,\varepsilon}(O)$ pour $0<r\le r_0$, l'énoncé s'obtient facilement.
\end{proof}

\begin{remark}
On peut également voir la démonstration du théorème  \ref{t: parametrtheorem} sous un angle légèrement différent. On déduit d'abord que l'ensemble compact $E$ est localement connexe (car les points dyadiques choisis à l'étape $n$ sur $E$ forment une $\epsilon_n$-chaîne avec $\epsilon_n\to 0$). Puis, on observe que pour tout couple de points $A,B\in E$ assez proches, l'ordre linéaire $A\le B \Longleftrightarrow z(A^{-1}B)\le 0$ est bien défini. Cet ordre linéaire étant compatible avec la topologie, l'ensemble connexe linéairement ordonné $E\cap U$ est homéomorphe à un intervalle. Voir une réalisation de la même idée dans la sous-section \ref{s: considtop}.
\end{remark}

\begin{remark} On peut vérifier que le paramétrage obtenu $\{t\to \Gamma(t)\}$ de l'ensemble $E\cap U$ est bi-Hölder continu, \ie $ c|t_2-t_1|^{\alpha(\varepsilon)}\le d_\infty(\Gamma(t_1),\Gamma(t_2)) \le \tilde c |t_2-t_1|^{\beta(\varepsilon)}$, $0<c\le \tilde c <\infty$, avec $\beta(\varepsilon)\le\frac{1}{2}\le \alpha(\varepsilon)$. Ici nous ne le démontrons pas (et ne l'utiliserons pas par la suite) premièrement parce que les exposants $\alpha$ et $\beta$ obtenus par l'algorithme ci-dessus sont loin des optimaux, et surtout parce qu'il est toujours facile de reparamétrer une courbe $E\cap U$ pour retrouver des meilleurs exposants (à l'image de ce qu'on fait dans la section \ref{s: proprietemetrique}). 
\end{remark}

\begin{corollaire}\label{reifenbergcurve}
Soit $F\in C_H^1(\h,\r^2)$, $F(0)=0$, et la différentielle $D_hf(0)$ est surjective. Il existe alors un voisinage $U$ de $0\in \h$ tel que $U\cap \f$ est l'image d'une courbe simple.
\end{corollaire}

Ainsi, localement l'ensemble de niveau $\f$ est un arc simple tangent (par exemple, au sens de dilatations homogènes) à l'axe vertical, ce qui justifie 
\begin{Def}\label{d: courbevertic} On appellera \emph{courbe verticale} l'ensemble $U\cap \f$ dans le corollaire \ref{reifenbergcurve}, autrement dit, une partie connexe de $\f$ localisée au voisinage de $0$. 
\end{Def}

\subsection{Considérations topologiques}\label{s: considtop}
Ici nous démontrons le corollaire \ref{reifenbergcurve} par une autre méthode qui présente un intérêt in\-dé\-pen\-dant. L'argument topologique que nous utiliserons ci-dessous est proche de celui dans \cite{funnel}.

\subsubsection{Sélection du flot continu en dimension $2$}
\paragraph{Problème de Cauchy.}
On fixe une fonction scalaire continue  $\psi$ définie sur un ouvert $\Omega \supset [-r,r]^2$ de $\r^2$ telle que $\max | \psi(y,z)| < 1$ pour $(y,z)\in [-r,r]^2$. (La valeur $1$ n'a aucune importance pour les résultats qui suivent; il suffirait de redimensionner les carrés en question).
 
On note  $\delta:=\mfrac{r}{2}$ et
on considère l'ensemble de fonctions suivant :
\begin{equation*}
\Z:=\Bigl\{z \in C^1\big([-\delta,\delta]; \r\big) \mid  \forall\, y\in[-\delta ,\delta] : z'(y)=\psi(y,z(y)) \text{ et }   z(0)\in [- \delta, \delta] \Bigr\}.
\end{equation*}

\begin{remark}
D'après le théorème de Peano, pour tout $z_0 \in [-\delta,\delta]$ il existe un élément $z\in \Z$ tel que $z(0)=z_0$.
De même, pour tout couple $(y_0,z_0)\in I_0:=[-\frac{r}{8}, \frac{r}{8}]^2$ il existe un élément de $z\in\Z$ tel que $z(y_0)=z_0$.
\end{remark}

\subparagraph{Topologie uniforme.}

On munit $\Z$ de la distance uniforme sur $[-\delta,\delta]$:
$$ d_u(z_1,z_2)=\max\limits_{y\in[-\delta,\delta] }|z_1(y)-z_2(y) |.$$ 

Une conséquence directe du théorème généralisé de Kneser \cite{generalKneser} est 
\begin{proposition*}
L'espace $(\Z, d_u)$ est compact et connexe.
\end{proposition*}

\subparagraph{Structure d'EPO.}
On munit naturellement $\Z$ d'une structure d'ensemble partiellement ordonné (EPO) en disant :  
$z_1\le z_2$  si et seulement si  $z_1(t)\le z_2(t)\ \forall\, t\in [-\delta,\delta]$. 

\begin{remark}
L'EPO  $(\Z,\le)$ est un treillis. Si $z_1,z_2\in \Z$ alors il est facile de voir que les fonctions
 $$(z_1\lor z_2)(t):=\max\{z_1(t),z_2(t)\}   \text{ et } (z_1\land z_2)(t):=\min\{z_1(t),z_2(t)\}$$  appartiennent également à $\Z$.
\end{remark}
\begin{remark*}
La compacité entraîne que toute suite généralisée monotone (au sens large) d'éléments de $\Z$ converge uniformément vers un élément de $\Z$. 
\end{remark*}
\begin{remark*}
l'EPO  $(\Z,\le)$ est un treillis complet: quel que soit le sous-ensemble non-vide $\Z'\subset \Z$ il existe $\sup \Z'$ et $\inf \Z'$, de plus ils sont dans la fermeture $\overline{\Z'}^{d_u}$. En effet, en prenant un ensemble $\{z_n\}$ au plus dénombrable partout dense dans $\Z'$ on vérifie que
$$ \sup \Z'=\lim\limits_{n\to \infty} (z_1\lor \ldots \lor z_n), \quad \inf \Z'=\lim\limits_{n\to \infty} (z_1\land \ldots \land z_n).$$
\end{remark*}

\paragraph{Flot sans pénétration.}
\begin{Def}\label{d: flotpent}
En vertu du théorème de Hausdorff, on peut choisir un élément maximal $\a$ (par rapport à l'inclusion) parmi les sous-ensembles de $\Z$ linéairement ordonnés. On appellera une telle famille de fonctions \emph{"flot sans pénétration local"}. 
\end{Def}

\begin{lemma}[Continuité du flot sans pénétration]\label{flotcontinu}
L'espace $(\a,d_u)$ est compact non-vide sans points isolés.  En outre, tout point $\a\ni z= \sup \{ \tilde z \in \a\mid \tilde z < z\}= \inf \{ \tilde z \in \a\mid \tilde z > z\}$ (avec la convention $\sup \varnothing=\inf \a$ et $\inf \varnothing=\sup \a$), et, par conséquent,  $(\a,d_u)$ est connexe. 
\end{lemma}
\begin{proof}
Comme $\Z$ est non-vide, $\a$  l'est aussi.

Montrons que $(\a,d_u)$ est fermé. Par l'absurde, soit $\a\ni z_n\to z\in \Z\setminus \a$ lorsque $n\to \infty$. Vu la maximalité de $\a$, il existe un élément $z_0\in \a$ qui est non-comparable avec $z$, \ie qu'ils existent $t_+,t_-\in  [-\delta,\delta]$ tels que $z(t_+)>z_0(t_+)$ et $z(t_-)<z_0(t_-)$. Ainsi,
soit $z_n\le z_0$ et, donc, $d_u(z_n,z)\ge z(t_+)-z_0(t_+)$, soit  $z_n\ge z_0$ et, donc, $d_u(z_n,z)\ge z_0(t_-)-z(t_-)$, d'où la contradiction au fait que $z_n\to z$.

Supposons, par l'absurde, qu'il existe une fonction $\a\ni  z_+>\inf \a$ telle que  $z_+ >z_-:=\sup\{ \tilde z \mid \tilde z <z_+\}\in \a$. On peut donc trouver  $y_0\in [-\delta, \delta]$ et  $z_0 \in[-r,r]$ tels que $ z_+(y_0)>z_0>z_-(y_0)$. D'après le théorème d'existence, on peut trouver une  solution maximale $z_\pm\in C^1(I_{z_\pm}, \r)$ qui vérifie $z_\pm'(y)=\psi(y,z(y))$ et  $z_\pm(y_0)=z_0$. Vu que $z_\pm(y)\to \partial \Omega$ lorsque $y\to \partial I_{z_\pm}$, l'élément $z^\star_\pm:=z_- \lor(z_+\land z_\pm)\in \Z$ est bien défini. Comme $z_+>z^\star_\pm>z_-$, on obtient la contradiction à la maximalité de $\a$. Le même raisonnement s'applique pour démontrer l'égalité avec $\inf$. En particulier, $\a$ n'a pas de points isolés.
 \end{proof}

Du lemme \ref{flotcontinu} et des résultats de la topologie générale 
(un compact connexe muni d'un ordre linéaire compatible avec la topologie est homéomorphe à un intervalle; voir, par exemple, \cite{genertopol}) découle
\begin{theorem}\label{t: flotintervalle}
L'espace $(\a,d_u)$ est homéomorphe à l'intervalle $([0,1],|\cdot|)$ par un ho\-méo\-mor\-phisme préservant l'ordre.
\end{theorem}

\begin{remark}
Il est facile d'en déduire que l'ensemble $(\Z,d_u)$ lui-même est connexe par arcs.
Or, l'ensemble de solutions $$\mathcal{F}=\big\{\lambda\in  C^1([-\delta,\delta],\r^n) \mid \lambda'(t)=\vec{V}\circ \lambda (t),\, t\in [-\delta,\delta], \text{ et } \lambda(0)=0 \big\},$$ où $\vec{V}$ est un camps de vecteurs continu et borné sur $\r^n$, n'est pas connexe par arcs en général  \cite{cross-section}. Notons aussi que $(\mathcal{F},d_u)$ est toujours un compact connexe qui est acyclique en homologie de \v{C}ech \cite{aronszajn}.
\end{remark}

\subsubsection{Application aux lignes de niveau}\label{apptopol}
\paragraph{Deux fonctions scalaires.}   
On considère maintenant de nouveau une application $F=(f,g)\in C^1_H(\h,\r^2)$ telle que $F(0)=0$ et $\det(d_hF)\not = 0$.  On voit cette fois-ci l'ensemble $F^{-1}(0)$ comme l'intersection des deux surfaces $\h$-régulières, $f^{-1}(0)$ et $g^{-1}(0)$.  Sans perte de généralité on peut supposer $Xf\not=0$. Soit $\phi$ une fonction donnée par le théorème des fonctions implicites appliqué à $f$ comme cela a été fait dans la sous-section  \ref{regulsurf} (on garde les notations de la section \ref{hsurfaces}). Compte tenu du caractère local de notre étude, on cherche à décrire $\Phi(\Omega)\cap g^{-1}(0)$;  on note donc $\mathcal{A}^\phi:=\Phi^{-1}(g^{-1}(0))$.  

Voyons comment la condition de \emph{Whitney} \eqref{whitneycond} s'écrit sur $\mathcal{A}^\phi$. On se rappelle que pour tout $\Omega'\Subset \Omega$  l'estimation uniforme pour $A=(y_1,z_1), B=(y_2,z_2)\in\Omega'\cap \mathcal{A}^\phi$ a lieu
$$
|\pi(\Phi(A)) -\pi(\Phi(B))|=o(|z\big(\Phi(A)^{-1}\Phi(B)\big)|^{1/2}),
$$
ou explicitement  
$$
| \phi(B)-\phi(A) |^2 +|y_2-y_1|^2=o( |z_2-  z_1 +2(y_2-y_1)(\phi(A)+\phi(B))|).
 $$
La distance $d_\phi\corn\mathcal{A}^\phi$ est égale à $$d_\phi(A,B)=| z_2-  z_1 +2(y_2-y_1)(\phi(A)+\phi(B))|^{1/2}$$ dès que les points $A,B\in \mathcal{A}^\phi$ sont suffisamment proches. Compte tenu de  \eqref{Whytney}, la condition de Whitney \eqref{whitneycond} est équivalente à
$$ |y_2-y_1|^2=o(|z_2-  z_1 +2(y_2-y_1)(\phi(A)+\phi(B))|),$$
où le petit-$o$ est uniforme pour $A,B\in\Omega'\cap\mathcal{A}^\phi$. 

\paragraph{Application.}
Nous allons appliquer maintenant le théorème \ref{t: flotintervalle} à  $\psi:=-4\phi$. On choisit $r>0$ de sorte que $[-r,r]^2 \subset \Omega$ et $| \phi|<1/4$ sur $[-r,r]^2$. On note également $\a$ un flot sans pénétration obtenu à partir de $\psi$. 

 Montrons qu'au voisinage de l'origine le graphe de tout élément de $\a$ (qui est une courbe intégrale de $W^\phi$) contient exactement un seul point de $\mathcal{A}^\phi$. 
 Soit $\gamma(t)$ une courbe intégrale de $W^\phi$, d'après le lemme \ref{l: reguler} la composition $g\circ \Phi\circ \gamma:[-\delta,\delta]\to \r$ est de classe $C^1$, sa dérivée vaut $$\left(-\frac{Yf}{Xf}Xg+Yg\right)\circ \Phi \circ\gamma=\frac{\det d_h F }{Xf}\circ \Phi \circ\gamma.$$ Par la continuité de $D_h F$ quitte à diminuer $r>0$, on peut supposer  $$\min_{[-r,r]^2} \left|\frac{\det d_h F }{Xf}\circ \Phi\right|\ge v> 0. $$
 Par la continuité de $g$, il existe $r_1\in (0,\delta]$ tel que $\max\limits_{z\in[-r_1,r_1]} |g\circ \Phi(0,z)|< v\delta$. 
 On considère deux fonctions $z_\alpha, z_\beta\in \a$ telles que $z_\alpha(0)=-r_1$ et $z_\beta(0)=r_1$. 
 On note $\a'=\{z\in \a\mid z_\alpha\le z\le z_\beta\}$ et $U=\{(y,z(y)) \mid y\in[-\delta,\delta],  z \in \a' \}$ un voisinage compact de $0$.
 Il est facile de voir (car on a tout fait pour) que 
 si  $z\in\a'$ alors le graphe de $z$ rencontre toujours un seul point de $\mathcal{A}^\phi$.
 
On définit maintenant l'application $P:\a' \to \mathcal{A}^\phi \cap U$ qui associe à chaque élément de $\a'$  l'unique point de $\mathcal{A}^\phi \cap U$ que son graphe contient.  La condition de Whitney sur $\mathcal{A}^\phi$ et la description géométrique de $d_\infty$ (le lemme \ref{metriqueR}) nous donnent finalement 
\begin{proposition} L'application $P:(\a',d_u) \to \big(\mathcal{A}^\phi \cap U, d_\infty)$ est $1/2$-Hölder continue et surjective. On observe qu'en particulier
$$ \H^2_\infty(\mathcal{A}^\phi \cap U)\le C\H^1_{d_u}(\a'), C<\infty.$$
\end{proposition}

Il est évident que $(\a', d_u)$ est homéomorphe à l'intervalle $([0,1], |\cdot|)$. Soit $s:([0,1], |\cdot|,\le)\to (\a', d_u, \le)$ un homéomorphisme monotone. On considère la relation d'équivalence  $\mathcal{R}$ sur $[0,1]$ définie par: $\alpha \mathcal{R} \beta  \Leftrightarrow P\circ s(\alpha) =P\circ s(\beta)$.  Chaque classe d'équivalence contient soit un intervalle fermé soit un point. L'espace topologique quotient $([0,1]/\mathcal{R}, \tau_\mathcal{R})$ est aussi, à son tour, homéomorphe à $([0,1], |\cdot|)$. L'application $P\circ s$ passe au quotient et définit ainsi une application continue bijective $\tilde P: ([0,1], |\cdot|) \to \big(\mathcal{A}^\phi \cap U, d_\infty)$. Ainsi, nous arrivons à la conclusion du corollaire \ref{reifenbergcurve}.

\paragraph{Ensemble transverse à un "entonnoir".} 
A la fin de cette section on voudrait soulever un problème topologique qui, à notre connaissance,  reste ouvert dans sa généralité. Le cas le plus simple a été traité ci-dessus ($n=1$); voir aussi \cite{funnel} pour la généralisation de ce raisonnement et une riche littérature sur la théorie générale des "entonnoirs" (ang. \emph{funnels}) \cite{aronszajn, funnelsections, cross-section, quasimonotone}.

\begin{question}[Problème topologique]\label{probtop} Soit $\vec{V}=(1,V_1,\ldots,V_n)$ un champ de vecteurs dans $\r^{n+1}$ où les fonctions $V_i: \r^{n+1}\to \r$ sont seulement bornées continues. Considérons l'ensemble $\mathcal{A}$ ("entonnoir") des lignes intégrales du champ $\vec{V}$ qui se trouvent en temps $t=0$ sur l'hyperplan $\{x_1=0\}$.   Supposons que l'ensemble fermé $E\subset \{ -\delta \le x_1 \le \delta \}$ est transverse à $\mathcal{A}$ au sens suivant. Toute ligne de $\mathcal{A}$ coupe l'ensemble $E$ en exactement un seul point et de plus l'application qui associe à une ligne de $\mathcal{A}$ l'unique point de $E$ lui appartenant est continue, où $\mathcal{A}$ est muni de la topologie uniforme (sur l'intervalle de temps $x_1\in [-\delta,\delta]$). 

 L'ensemble $E$ est-il localement homéomorphe à une boule $B^n$ de $\r^n$?
\end{question}

\section{Propriétés métriques des courbes verticales}\label{s: proprietemetrique}
Après avoir établi que l'ensemble de niveau $\Gamma:=\f\cap U$ est un arc simple, nous revenons à des considérations dans $\h$ (et plus sur $\mathcal{S}$). Les coordonnées $x$, $y$ et $z$  d'un point $A\in \h$ seront désignées par $x(A)$, $y(A)$ et $z(A)$ respectivement. Etant homéomorphe à un intervalle $\Gamma$ hérite d'une structure d'ensemble linéairement ordonné. Pour tout $A,B\in \Gamma,A\le B,$ on notera par l'intervalle $[A,B]=\{C\in\Gamma \mid A\le C\le B\}$. Nous choisissons l'ordre croissant sur $\Gamma$ ce qui veut dire que 
\begin{equation}\label{e: contactpositif}
\text{ si }A,B\in \Gamma \text{ et } A\le B \Longrightarrow d_\infty(A,B)^2=z(B)-z(A)-2\big(x(B)y(A)-x(A)y(B)\big).
\end{equation}
On peut toujours faire ce choix car si $A,B\in \Gamma$ et $A\not=B$, alors $z(B)-z(A)-2\big(x(B)y(A)-x(A)y(B)\not =0$.

L'élément essentiel de cette section est  
\begin{remark}\label{r: dcarreplat}
On considère trois points $A, B,C\in \Gamma$ tels que $A\le B\le C$. D'après \eqref{whitneycond} on a avec un petit-$o$ uniforme lorsque $A\to C$
\begin{multline}\label{explicitedcarreplate}
d_\infty(A,B)^2+d_\infty(B,C)^2-d_\infty(A,C)^2=2\det\big( \pi(B)-\pi(A), \pi(C)-\pi(B) \big)  \\
=2\Big((y(C)-y(B))(x(B)-x(A))-(y(B)-y(A))(x(C)-x(B))\Big) \\
=o(d_\infty(A,B))o(d_\infty(B,C))=o(d_\infty(A,B)^2+d_\infty(B,C)^2)=o(d_\infty(A,C)^2).
\end{multline}
Nous réécrivons la dernière relation sous la forme 
\begin{equation}\label{dcarreplate}
 |d_\infty(A,B)^2+d_\infty(B,C)^2-d_\infty(A,C)^2|\le m(d_\infty(A,C)^2)d_\infty(A,C)^2
\end{equation}
pour tous $A\le B\le C$ sur $\Gamma$, où $m(t)\searrow 0$ lorsque $t\searrow 0$.
\end{remark}

\begin{remark*} La condition \eqref{dcarreplate} pour une courbe simple $\Gamma \subset \h$ est  en général plus faible que la condition de Whitney \eqref{whitneycond}. Prenons, par exemple, $\Gamma(t)=(t,0, 2t|t|)$, $t\in [-1,1]$. En effet, pour $t_2\ge t_1$ on a $d_\infty\big(\Gamma(t_1),\Gamma(t_2)\big)^2=\max\{(t_2-t_1)^2, 2(t_2|t_2|-t_1|t_1|)\}= 2(t_2|t_2|-t_1|t_1|)$. Vu que $t\to t|t|$ est croissante, 
$ d_\infty\big(\Gamma(t_1),\Gamma(t_2)\big)^2+d_\infty\big(\Gamma(t_2),\Gamma(t_3)\big)^2-d_\infty\big(\Gamma(t_3),\Gamma(t_1)\big)^2\equiv 0$ pour tous $t_1\le t_2 \le t_3$. Néanmoins, $\|\pi(\Gamma(t))-\pi(\Gamma(0))\|=|t|\not=o(d_\infty\big(\Gamma(t),\Gamma(0)\big)=\sqrt{2}|t|$.
\end{remark*}

\subsection{Quasi-métriques plates} \label{ss: quasiplate}
Pour comprendre les propriétés de $\Gamma$, qui découlent de \eqref{dcarreplate}, nous nous plaçons dans un cadre légèrement plus abstrait. Notamment, tout au long de cette sous-section nous supposons que $\kappa$ est une quasi-métrique continue (par rapport à la topologie usuelle) sur $[0,1]$ qui est \emph{plate} dans le sens où elle satisfait la condition suivante :
\begin{equation}\label{e: kappaplate}
|\kappa(A,B)+\kappa(B,C)-\kappa(A,C)|\le m(\kappa(A,C))\kappa(A,C),
\end{equation}
pour tous $0\le A\le B\le C\le 1$, où $m(t)\searrow 0$ lorsque $t\searrow 0$. L'espace quasi-métrique $\lambda:=([0,1],\kappa)$, étant homéomorphe à un intervalle, sera appelé "courbe plate". Bien sûr, nous pensons avant tout au cas où $\lambda= (\Gamma, d^2_\infty\corn \Gamma)$.
\begin{remark*}
Toute courbe dans $\r^n$  qui est $\epsilon$-plate au sens de Reifenberg  avec $\epsilon\to 0$ lorsque l'échelle se raffine, satisfait la condition \eqref{e: kappaplate} avec $\kappa$ une métrique euclidienne induite (voir \cite{vanishingreifenberg,reifenbergflatmetric}). De bons exemples sont donnés par des flocons de neige plats (\ie dont les angles décroissent avec l'échelle); ils peuvent en particulier avoir une mesure $\H^1_\kappa$ infinie. 
\end{remark*}
La première conséquence facile de \eqref{e: kappaplate} est
\begin{proposition}\label{diametre} Si $[A,B]\subset [0,1]$, alors uniformément lorsque $\kappa(A,B)\to 0$,
$$ \diam_\kappa([A,B])=\kappa(A,B)(1+o(1)).$$
\end{proposition}
\begin{proof}
Par compacité pour tous $0\le A\le B\le 1$ on peut trouver $C,D\in [A,B]$, $C\le D$, tels que $\diam_\kappa([A,B])=\kappa(C,D)$. D'après \eqref{e: kappaplate},
\begin{align*}
\kappa(A,B)-\kappa(A,D)-\kappa(D,B)=o(\kappa(C,D)),\\
\kappa(A,D)-\kappa(A,C)-\kappa(C,D)=o(\kappa(C,D)),
\end{align*}
et en additionnant on conclut que
$$ k(A,B)\ge \kappa(A,B)-\kappa(A,C)-\kappa(D,B)=\kappa(C,D)(1-o(1)). \qedhere$$
\end{proof}

\begin{proposition}\label{dimensionestime}
La dimension de Hausdorff de la courbe plate $\lambda$ est égale à $1$.
\end{proposition} 
\begin{proof}
On construit par récurrence une suite de familles d'intervalles fermés $J_k=\{I^k_1,\ldots,I^k_{2^k}\}, k=0,1,2, \dotsc$. On pose $I^0_1=[0,1]$; supposant construit $J_k$ chaque intervalle $I^k_i, i=1,\ldots,2^k,$ donne naissance à deux intervalles de $J_{k+1}$ de la manière  suivante. Si $I^k_i=[a,c]$ alors $I^{k+1}_{2i-1}=[a,b]$ et $I^{k+1}_{2i}=[b,c]$ où le point $c \in [a,b]$ est tel que 
$\kappa(a,b)=\kappa(b,c)$. 
Pour  $I_i^k=[a,b]$ on note $L_{k,i}:=\kappa(a,b)$. On introduit également une distribution de masse $\mu$ sur $[0,1]$, \ie une mesure de probabilité borélienne, en posant  $\mu(I^k_i)=2^{-k}$ pour tout $I^k_i \in J_k, k=0,1,2,\dotsc$.

D'après \eqref{e: kappaplate} pour tout $k\ge 0$ et $i=1, \ldots, 2^k$,
$$ |L_{k,i}- 2L_{k+1,2i-1}|=|L_{k,i}- 2L_{k+1,2i}|\le m_{k,i} L_{k,i}, \quad m_{k,i}=m(L_{k,i}).$$ 
Pour $L_k:=\max\limits_i L_{k,i}$ et $l_k:=\min\limits_{i} L_{k,i}$ on démontre par récurrence sur $k$ que
 \begin{equation}\label{estimation}
 \frac{L_0}{2^n} \prod\limits_{k=1}^{n}(1-m_k)\le l_n\le L_n\le \frac{L_0}{2^n} \prod\limits_{k=1}^{n}(1+m_k), \quad m_k:=\max\limits_i |m_{k,i}|.
\end{equation}
Quitte à considérer un nombre fini d'intervalles plus petits nous pouvons supposer dès le début que $m_k\le m_0<1$, $k=0,1,\dotsc$. On voit qu'alors  
\begin{equation}\label{estimation1}
L_0\left(\mfrac{1-m_0}{2}\right)^n \le l_n\le L_n \le L_0\left(\mfrac{1+m_0}{2}\right)^n,
\end{equation}
 et donc $m_k\to 0$  lorsque $k\to \infty$.  
 
\textbf{Majoration.}
On fixe $\alpha >1$. Quelque soit $n$ entier les intervalles de famille $J_n$ recouvrent $[0,1]$ et leurs diamètres tendent vers $0$. Estimons donc
\begin{multline*}
\H^\alpha_{\kappa,\delta_n}(\lambda)\le\sum\limits_{i=1}^{2^n} \diam_\kappa(I^n_i)^\alpha\le  c_n \sum\limits_{i=1}^{2^n} L_{n,i}^\alpha \le c_n 2^n  L_n L_n^{\alpha-1} \\ \le
L_0 c_n L_n^{\alpha-1}\prod\limits_{k=1}^{n}(1+m_k)\le c_nL_0^{\alpha}\prod\limits_{k=1}^{n}(1+m_k)\bigg(\frac{1+m_0}{2}\bigg)^{\alpha-1},
\end{multline*}
 où $c_n\to 1$ et $\delta_n\to 0$ quand $n\to \infty$.  A partir d'un certain $n_0$ on a bien $(1+m_k)\big(\frac{1+m_0}{2}\big)^{\alpha-1}<1$, $k\ge n_0$. Par conséquent, $\H^\alpha_\kappa(\lambda)=0$ et donc $\dim\lambda \le 1$.
 
\textbf{Minoration.}
On utilisera le principe de distribution de masse \cite{flaconer1, federer}.
Soit $\alpha<1$ quelconque. Pour avoir $\dim\lambda \ge 1$ il suffit donc de vérifier que pour tout $[A,B]\subset[0,1]$ on a   $$\lim\limits_{\kappa(A,B)\to 0} \mu([A,B])\kappa(A,B)^{-\alpha}=0.$$
 Avec les estimations \eqref{estimation} et  \eqref{estimation1} on démontre sans peine que 
\begin{equation*}
 \limsup\limits_{[a,b]\in J_k,k\to \infty}\mu([a,b])\kappa(a,b)^{-\alpha}=0.
\end{equation*}
Pour finir il reste à remarquer que tout intervalle de $[A,B]\subset [0,1]$ contient un certain intervalle $I^k_i\in J_k$ suffisamment large avec $4\mu(I^k_i)\ge\mu([A,B])$ (voir fig. \ref{largeintervalle}).
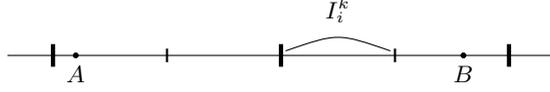
\begin{figure}[h]\centering
\begin{tikzpicture}[scale=3] 
\draw (-1.2,0) -- (1.2,0);
\draw [ultra thick] (-1,-0.05)--(-1,0.05);
\draw [ultra thick] (0,-0.05)--(0,0.05);
\draw [ultra thick] (1,-0.05)--(1,0.05);
\draw [thick] (-0.5,-0.03)--(-0.5,0.03);
\draw [thick] (0.5,-0.03)--(0.5,0.03);
\draw (0.02,0.02).. controls (0.25, 0.1) .. (0.48,0.02);
\draw (0.25,0.1) node[above]{\footnotesize $I_i^k$};
\filldraw (0.8,0)  circle (0.3pt)  node[below] {\footnotesize $B$};
\filldraw (-0.9,0) circle (0.3pt)  node[below] {\footnotesize $A$};
\end{tikzpicture}
\caption{$J_k\ni I\subset [A,B]$ et $4\mu(I^k_i)\ge\mu([A,B])$ }
\label{largeintervalle}
\end{figure}
\end{proof}
\begin{Def}
La courbe $\lambda$ est appelée $p$-\emph{Ahlfors} régulière s'il existe $0<C<\infty$ tel que tout $[A,B]\subset [0,1]$ vérifie  $ C^{-1}\kappa(A,B)^p\le \H^p_\kappa([A,B])\le C\kappa(A,B)^p$.
\end{Def}
\begin{remark}\label{ahlfors}
Suivant le raisonnement de la proposition \ref{dimensionestime} on observe que l'hypothèse $\sum_{k=0}^{\infty} m_k < \infty$ entraîne la $1$-Ahlfors régularité de $\lambda$. En effet, il suffit de remarquer que dans ce cas-là les deux produits dans l'estimation \eqref{estimation} convergent, et par conséquent $0<\tilde C^{-1}<\mu([A,B])\kappa([A,B])^{-1}<\tilde C<\infty$ pour tout intervalle non-vide $[A,B]\subset \lambda$.
\end{remark}

\begin{Def}
Soit $(X,d)$ un espace quasi-métrique et $\lambda:[0,T]\to (X,d)$ une courbe.  Pour $p>0$ on définit $p$-variation de $\lambda$ comme 
$$\operatorname{Var}^p(\lambda)=\sup\left\{\sum\limits_{i=0}^{n-1}d(\lambda(t_i),\lambda(t_{i+1}))^p \,\mid\, 0\le t_0\le \ldots \le t_n\le T\right\}.$$
\end{Def}

Une conséquence immédiate des propriétés métriques de la mesure $\mu$ construite ci-dessus sur $\lambda$ est
\begin{corollaire}
La courbe plate $\lambda$ est de $p$-variation finie $\operatorname{Var}^p\lambda< \infty$ pour tout $p>1$ et de $p$-variation infinie $\operatorname{Var}^p\lambda=\infty$ si $p<1$. De façon équivalente, la courbe plate $\lambda$ admet toujours un paramétrage hölderien $([0,1],|\cdot|)\to ([0,1],\kappa)$ avec tout exposant $\alpha< 1$ et jamais avec $\alpha>1$.  
\end{corollaire}
  
 \begin{Def}\label{infvar}
Pour la courbe $\lambda$ on introduit la variation inférieure
\begin{equation*}
\Var(\lambda):=\liminf\limits_{\delta \to 0} \sum\limits_{i=0}^n \diam_\kappa([A_i,A_{i+1}])=\liminf\limits_{\delta \to 0} \sum\limits_{i=0}^n \kappa(A_i,A_{i+1}), 
\end{equation*}
où l'infimum est pris sur toutes les subdivisions $\sigma=\{0=A_0< A_1<\ldots<A_n< A_{n+1}=1\}$ de $[0,1]$ avec
 $\|\sigma\|:=\max\limits_{i=0,\ldots,n} |A_i-A_{i+1}|<\delta$ (la deuxième égalité est vraie à cause de la proposition \ref{diametre}).
\end{Def}

\begin{proposition}[formule de l'aire] \label{generaire}
Pour la courbe plate $\lambda$, il est vérifié $ \H^1_\kappa(\lambda)=\Var(\lambda)$.
\end{proposition}

\begin{proof}
Il est clair que $\H^1_\kappa(\lambda)\le \Var(\lambda)$ car l'infimum pour $\H^1_\kappa$ est pris sur un ensemble plus large.
Supposons donc que $\H^1_\kappa(\lambda)<\infty$ et montrons l'inégalité inverse. 

Soit $[0,1]\subset \bigcup\limits_{i=0}^N E_i$, $0<\diam_\kappa(E_i)<\delta$, $E_i$ est ouvert, un recouvrement fini (vu la compacité) de $\lambda$. Pour $E_i\not=\varnothing$ on définit $A_i=\inf\{E_i\}$ et $B_i=\sup\{E_i\}$ et on voit que $\Gamma\subset \bigcup\limits_i [A_i,B_i]$. On peut alors trouver une suite de points $0=C_0< C_1<\ldots<C_n< C_{n+1}=1$ tel que tout l'intervalle $[C_i,C_{i+1}]$ soit contenu dans un certain $[A_k,A_{k+1}]$ et que deux intervalles successifs n'appartiennent pas à un même $[A_k,A_{k+1}]$. 
D'après la proposition \ref{diametre} on a que $\diam_\kappa([A_i,B_i])\le \kappa(A_i,B_i)(1+\varepsilon(\delta)) \le \diam_\kappa(E_i)(1+\varepsilon(\delta))$, où $\varepsilon(\delta)\to 0$ lorsque $\delta\to 0$.
Finalement on obtient que
\begin{equation*}
 \sum\limits_{i=0}^n \kappa(C_i,C_{i+1})\le  \sum\limits_k \diam_\kappa([A_k,A_{k+1}]) \le (1+\varepsilon(\delta)) \sum_{k} \diam_\kappa(E_k),
\end{equation*}
d'où on déduit que $\Var(\lambda)\le \H^1_\kappa(\lambda)$. 
\end{proof}

\begin{lemma}[formule de l'aire: cas régulier]\label{aireregulier}
Supposons que $\sum_{i=0}^\infty m(2^{-i})< \infty$. Alors la courbe plate $\lambda$ est $1$-Ahlfors régulière, et en outre
\begin{align}\label{youngestime}
&\nonumber \H^1_\kappa(\lambda)=\lim\limits_{\|\sigma\|\to 0} \sum_{k=0}^n \kappa(t_k, t_{k+1}), \quad \sigma=\{0=t_0<t_1<\ldots<t_{n+1}=1\}, \\
& | \H^1_\kappa([s,t])-\kappa(s,t)|\le C\kappa(s,t) M(\kappa(s,t)),\quad M(l):=\sum_{i=0}^\infty m(Cl2^{-i}),  [s,t]\subset \lambda, C<\infty.
\end{align}
\end{lemma}

\begin{remark*}
L'hypothèse $\sum_{i=0}^\infty m(2^{-i})< \infty$ est un analogue (uniformisé) de la sommabilité de nombres de Jones dans le problème de voyageur de commerce \cite{jones}. En général, en l'absence de cette hypothèse on peut avoir $\H^1_\kappa(\lambda)=0$ ou $\H^1_\kappa(\lambda)=\infty$.
\end{remark*}

\begin{proof}
Comme on a supposé $m$ monotone, $\sum_k m(2^{-k})<\infty$ entraîne que $\sum_k m(r^{-k})<\infty$ pour tout $1<r<\infty$.
D'après la remarque \ref{ahlfors}, $\lambda$ est $1$-Ahlfors régulière et admet donc un paramétrage bi-lipschitzien $([0,T],|\cdot|)\longrightarrow ([0,1],\kappa)$. De plus, on peut toujours choisir ce paramétrage de sorte que \eqref{e: kappaplate} se réécrive 
\begin{equation*}
|\kappa(s,t)+\kappa(h,t)-\kappa(s,h)|\le C m(|s-t|)|s-t|, \quad h\in [s,t]\subset [0,T],\ C<\infty. 
\end{equation*}
Compte tenue de la formule de l'aire (proposition \ref{generaire}) pour obtenir le résultat énoncé il suffit d'appliquer maintenant 
\begin{lemma}[de la couture \cite{feyel}] \label{couture}
Soit $\mu:[0,1]^2\to \r$ une fonction continue telle que 
\begin{equation*}
|\mu(a,b)+\mu(b,c)-\mu(a,c)|\le \omega(|a-c|), \quad b\in [a,c],
\end{equation*}
pour une fonction $\omega(t)\searrow 0$ quand $t\searrow0$ et $\sum_{i=0}^\infty \omega(2^{-i})< \infty$. Alors il existe une unique (à une constante additive près) fonction $\nu:[0,1]\to \r$ telle que 
\begin{equation*}
|\nu(b)-\nu(a)-\mu(a,b)|\le \sum_{i=0}^\infty 2^i\omega(|b-a|2^{-i}).
\end{equation*}
 En outre, les sommes de Stieltjes $\sum_{i=0}^N \mu(t_i,t_{i+1})$, où $\sigma=\{a=t_0<t_1<\ldots<t_{N+1}=b\}$ est une subdivision de $[a,b]$, convergent vers $\nu(b)-\nu(a)$ lorsque $\|\sigma\|\to 0$.
\end{lemma}
Remarquons que ce lemme a été démontré dans \cite{feyel} dans le cas où $\omega(t)=Kt^{1+\alpha}$, $\alpha>0$, mais sa démonstration se modifie facilement pour couvrir les hypothèses que nous avons énoncées ci-dessus (voir également \cite{younggeneral, burkill}). 
\end{proof}
\subsubsection{Application aux courbes verticales}
Il est temps de déduire quelques propriétés des lignes de niveau résultant des con\-si\-dé\-ra\-tions sur les quasi-métriques plates. 

Soit, comme au début de la section, $\Gamma=\f\cap U$  une courbe verticale paramétrée de façon croissante pour une application $F\in C_H^1(\h,\r^2)$ avec la différentielle $D_h F$ surjective. On note $x_i, y_i,z_i$ les coordonnées $x,y,z$ respectivement du point $A_i\in \Gamma\Subset \h$. Soit $\sigma=\{A=A_0<A_1<\ldots<A_{N}<A_{N+1}=B\}$ une subdivision sur $\Gamma=[A,B]$. On voit que
\begin{multline*}
\sum_{i=0}^N d_\infty(A_i,A_{i+1})^2=\sum_{i=0}^{N} z_{i+1}-z_i +2 (x_{i}y_{i+1}-y_{i}x_{i+1})=\\
= z(B)-z(A) +2\sum_{i=0}^{N} x_{i}(y_{i+1}-y_{i})- y_{i}(x_{i+1}-x_{i}). 
\end{multline*}
D'après \eqref{dcarreplate}, $(\Gamma, d_\infty^2\corn \Gamma)$ est une courbe plate, et comme la mesure de Hausdorff est invariante par un plongement isométrique, on déduit des propositions \ref{dimensionestime} et \ref{generaire} 
\begin{corollaire}\label{airegamma}
La courbe $\Gamma$ est de dimension sous-riemannienne égale à $2$, $\dim_h \Gamma=2$, et   
\begin{multline}\label{airedcarre} 
\H^2_\infty(\Gamma)=z(B)-z(A)+2\liminf\limits_{\|\sigma\| \to 0} \sum_{i=0}^{N} (x_{i}y_{i+1}-y_{i}x_{i+1})\\ =:\int\limits_\Gamma \,dz+2\liminf\limits_{\|\sigma\| \to 0}\int_\sigma (x\,dy-y\,dx) .
\end{multline}
\end{corollaire}

En revanche, la mesure de Hausdorff sphérique d'un sous-ensemble  dépend en général de l'espace ambiant.
\begin{proposition} Toute courbe verticale $\Gamma$ vérifie  $\H^2_\infty\corn \Gamma=\frac{1}{2}\S^2_\infty \corn \Gamma$.
\end{proposition}
\begin{proof}
 Il suffit de démontrer deux inégalités. D'une part,  $\diam_\infty(B_\infty(C,r)\cap \Gamma)\le \sqrt{2}r(1+o(1))$, où le petit-$o$ lorsque $r\to 0$ est uniforme en $C\in \h$. En effet, pour des points $A,B\in \Gamma\cap B_\infty(C,r)$, $A\le B$, on utilise la condition de Whitney sur  $\Gamma$ pour déduire
\begin{multline*}
d_\infty(A,B)^2= z(B)-z(A)-2\det(\pi(B),\pi(A))=\big(z(B)-z(C)-2\det(\pi(B),\pi(C))\big)\\
-\big(z(A)-z(C)-2\det(\pi(A),\pi(C))\big)- 2 \big(\det(\pi(B),\pi(A))-\det(\pi(B),\pi(C))+\det(\pi(A),\pi(C))\big)\\
\le d_\infty(B,C)^2+d_\infty(A,C)^2+2|\det(\pi(B)-\pi(A), \pi(A)-\pi(C))|\le 2r^2+o(r^2).
\end{multline*}
 D'autre part, tout intervalle $[A,B]\subset \Gamma$ est contenu dans une boule de rayon $r=\mfrac{d_\infty(A,B)}{\sqrt{2}}(1+o(1))$, le petit-$o$ lorsque $A\to B$ est uniforme en $A,B\in \Gamma$. En effet, étant donnés deux points $A,B\in \Gamma$ on prend un point $C\in \Gamma$ tel que $d_\infty(A,C)=d_\infty(B,C):=\tilde r$. On constate que $[A,B]\subset B_\infty(C,r)$, où $r=\max\{\diam_\infty([A,C]), \diam_\infty([B,C])\}$.
Or, $2\tilde r^2=d(A,B)^2(A+o(1))$ d'après \eqref{dcarreplate}, et  $\frac{\tilde r}{r}\to 1$ d'après la proposition \ref{diametre} (uniformément lorsque $A\to B$).
   \end{proof}

\begin{remark}\label{rapportdeuxdist}
Si au lieu de $d_\infty$ on considère une autre quasi-métrique $d_\rho$ engendrée par une norme homogène $\rho$ sur $\h$, on aura  $\H^2_\rho(\Gamma)=c_\rho\H^2_\infty(\Gamma)$, où $c_\rho=\rho(0,0,1)$ est le coefficient de dilatation de $\rho$ par rapport à $\rho_\infty$ sur l'axe vertical. En effet, à cause de la condition de Whitney sur $\Gamma$, $\diam_\rho E=c_\rho\diam_\infty E(1+o(1))$ lorsque $\diam_\infty  E \to 0$ pour $E\subset \Gamma$, d'où la remarque découle facilement. 
\end{remark}
\begin{remark*}
Pour obtenir la constante $\frac{1}{2}$ comme le rapport $\mfrac{\H^2_\infty}{\S^2_\infty}\corn \Gamma$ il suffit également de considérer les deux mesures sur l'axe vertical $Oz$.
\end{remark*}

\subsection{Cas régulier} 

\begin{Def}
On dira que la courbe verticale $\Gamma$ est fortement régulière au sens d'Ahlfors (ou $f\alpha$-régulière), si uniformément pour les intervalles $[C,D]\subset \Gamma$ on a une estimation 
\begin{equation*}
 \H^2_\infty([C,D])= d_\infty(C,D)^2+o(d_\infty(C,D)^2) \text{ lorsque  $C\to D$}.
\end{equation*}
\end{Def}
On  rappelle 
\begin{Def}\label{d: stieltjesinteg}
Soient  $x,y\in C^0([0,T],\r)$ deux fonctions continues.  L'intégrale de Stieltjes $ \int_0^T x\, dy$ est définie comme la limite (lorsqu'elle existe et finie) des sommes $$ \int_\sigma x\,dy:=\sum \limits_{i=0}^{N} x(t_i)\big(y(t_{i+1})-y(t_i)\big),$$
comptées sur toutes les subdivisions $\sigma:=\{0=t_0<t_1<\ldots<t_N<t_{N+1}=T\}$, quand $\|\sigma\|:=\max\limits_i  |t_{i+1}-t_i|\to 0$.
\end{Def}
\begin{notation}\label{n: airelevy} 
Pour $x,y \in C^0([0,T], \r)$ et $t\in [0,T]$ on notera  $\Lift_{x,y}(t):=2\int_0^t (x\,dy -y\,dx)$, où la dernière intégrale Stieltjes "mixte"  est définie comme la limite finie de 
$$ \int_0^t (x\,dy -y\,dx)= \lim_{\|\sigma\|\to 0} \sum_{i=0}^{N} (x_{i}y_{i+1}-y_{i}x_{i+1}).$$
Remarquons que l'existence de $\Lift_{x,y}(t)$ n'entraîne pas en général l'existence de $\int_0^t x\,dy$ ni de $\int_0^t y\,dx$.
\end{notation}

\begin{proposition}\label{fortregular}
Si la courbe $\Gamma\Subset \h$ est $f\alpha$-régulière, alors la formule de l'aire a lieu 
\begin{equation*}
\H^2_\infty(\Gamma)=\int\limits_{\Gamma} \,dz+2\int\limits_{\Gamma}x\,dy-2\int\limits_{\Gamma} y\,dx.
\end{equation*}
La courbe $\Gamma$ admet un paramétrage naturel par $t\rightarrow \Gamma(t)\in \h$ tel que 
\begin{equation*}
 \H^2_\infty\big(\Gamma([0,t])\big)=t \text{ pour tout } t\in [0,\H^2_\infty(\Gamma)].
\end{equation*}
L'application $t\rightarrow \Gamma(t)$ est bi-hölderienne avec l'exposant $1/2$:
$$ C(|t-s|)|t-s|^\frac{1}{2}\ge d_\infty(\Gamma(t),\Gamma(s))\ge C(|t-s|)^{-1}|t-s|^\frac{1}{2},$$
où $C(\delta)\to 1$ quand $\delta\to 0$.
L'application $t\rightarrow \Gamma(t)$ est de la forme
\begin{equation}\label{formecanonique}
t\longrightarrow \Big(x(t),y(t), z(0)+t-2\int\limits_0^t x\,dy+2\int\limits_0^t y\,dx \Big).
\end{equation}
\end{proposition}

\begin{proof}
Le fait que $\Gamma$ soit $f\alpha$-régulière permet de remplacer "$\liminf$" par "$\lim$" dans l'expression de $\Var$ (voir définition \ref{infvar} et proposition \ref{generaire}) et suivant la formule \eqref{airedcarre} on obtient $$ \H^2_\infty(\Gamma)=\int\limits_{\Gamma} \,dz+2\int\limits_{\Gamma}(x\,dy - y\,dx).$$
Il reste à voir que les intégrales $\int_{\Gamma}x\,dy$ et $\int_{\Gamma} y\,dx$ existent séparément. Avec le paramétrage naturel les coordonnées $x$ et $y$ le long de $\Gamma$ appartiennent à $h_{1/2}$ vu la condition de Whitney sur $\Gamma$. D'après la remarque \ref{integrationpartie} l'intégration par parties est donc valable, ce qui implique l'existence de 
$$2\int\limits_{\Gamma} x\,dy = xy\big|_\Gamma + \int\limits_{\Gamma}(x\,dy - y\,dx). $$ 
 Le reste de l'énoncé est une conséquence directe de la $f\alpha$-régularité et de la formule de l'aire.
\end{proof}

En général, la courbe $\Gamma$ n'est pas $f\alpha$-régulière (voir sous-section \ref{exempleirregulier}). Par contre, la régularité supplémentaire  de l'une des deux applications scalaires de $F=(f,g)$ garantit que la courbe $\Gamma$ est $f\alpha$-régulière. 

\begin{lemma} \label{regularderivees}
Supposons que pour $F=(f,g)$ les dérivées horizontales de $f$ satisfont  
\begin{align}
&\label{deriveestimate}\max\big\{\|Xf(A)-Xf(B)\|, \|Yf(A)-Yf(B)\| \big\}\le m(d_\infty(A,B)^2),\\
&\label{derivesommability}\text{où } m:\r_+\to \r_+ \text{ est monotone et }
\sum_{k=0}^\infty m(r^{-k})<\infty,\text{ pour } r>1. 
\end{align}
Alors la courbe verticale $\Gamma\subset \f$ est $f\alpha$-régulière.
\end{lemma}  
\begin{remark*}
Si \eqref{derivesommability} est vérifié pour un $r>1$, alors vu la monotonie de $m$ cela reste vrai pour tout $r>1$. On supposera donc que $r=2$.
\end{remark*}
\begin{proof}
D'après le théorème \ref{lagrange} de Lagrange,  pour $A,B\in \h$,
\begin{equation*}
 |f(A)-f(B)-D_h f(B)(B^{-1}A)|\le c_1d_\infty(A,B)m\big(c_2 d_\infty(A,B)^2\big),\quad 0<c_1,c_2<\infty.
\end{equation*}
 Sans perte de généralité on peut supposer que $Xf\ge c^{-1}>0$ sur $\Gamma$, et donc
pour $A,B\in \Gamma$, cela se traduit par
$$|(x(A)-x(B))+w(B)(y(A)-y(B))|\le cc_1  d_\infty(A,B) m\big(c_2d_\infty(A,B)^2\big), \quad w(B)=-\mfrac{Yf}{Xf}(B).$$
Par conséquent, on obtient que pour tous $A\le B\le C$ sur $\Gamma$
\begin{multline*}
|d_\infty(A,B)^2+d_\infty(B,C)^2-d_\infty(A,C)^2|\\
\le 2|y(C)-y(B)|cc_1d_\infty(A,B) m\big(c_2d_\infty(A,B)^2\big)+2|y(B)-y(A)|cc_1d_\infty(C,B) m\big(c_2d_\infty(C,B)^2\big)\\
 \le \varepsilon\big(d(A,C)\big) m\big(c_3 d_\infty(A,C)^2\big) d_\infty(A,C)^2,\quad 0<c_3<\infty, 
\end{multline*} 
où $\varepsilon(\delta)\searrow 0$ quand $\delta\searrow 0$, d'après la condition de Whitney sur $\Gamma$ et la proposition \eqref{diametre}. 
Maintenant nous sommes en mesure d'appliquer \eqref{youngestime} du lemme \ref{aireregulier} et en déduire que pour tous $A,B\in \Gamma$, 
\begin{equation}\label{regulardcarreestime}
|\H^2_\infty([A,B])-d_\infty(A,B)^2|\le c_4 \varepsilon(d(A,B)) d_\infty(A,B)^2\? \sum_{i=0}^\infty m\big( c_4 d_\infty(A,B)^2 2^{-i}\big),\quad 0<c_4<\infty,
\end{equation}
d'où la régularité forte d'Ahlfors de $\Gamma$.
\end{proof}

\begin{corollaire} \label{regularitealpha}
Si $f\in C^{1,\alpha}_H(\h,\r)$, $\alpha>0$, alors l'estimation suivante a lieu uniformément pour $[A,B]\subset \Gamma$
\begin{equation*}
 \H^2_\infty([A,B])=d_\infty(A,B)^2\big(1+o(d_\infty(A,B)^\alpha)\big).
\end{equation*}
Si $F\in C^{1,\alpha}_H(\h,\r^2)$, $\alpha>0$, alors l'estimation suivante a lieu uniformément pour $[A,B]\subset \Gamma$
\begin{equation*}
 \H^2_\infty([A,B])=d_\infty(A,B)^2\big(1+O(d_\infty(A,B)^{2\alpha})\big).
\end{equation*}
\end{corollaire}
\begin{proof}
Pour $f\in C^{1,\alpha}_H(\h,\r)$, $\alpha>0$, le module de continuité des dérivées horizontales est donné par $m(\delta)\lesssim \delta^\alpha$.  Vu que $\sum_{i=0}^\infty (\delta 2^{-i})^\alpha\le c_\alpha \delta^\alpha$, il suffit d'appliquer l'estimée \eqref{regulardcarreestime}.

Dans le cas où $F\in C^{1,\alpha}_H(\h,\r^2)$, $\alpha>0$, pour tous $A,B\in \Gamma$ d'après le théorème \ref{lagrange} de Lagrange il est vérifié  
\begin{equation}\label{deuxalpha}
\|\pi(A)-\pi(B)\|\le c d_\infty(A,B)^{1+\alpha}, \ c<\infty.
\end{equation}
Cela entraîne l'estimation suivante : pour tous $A\le B\le C$ sur $\Gamma$,
\begin{equation*}
|d_\infty(A,B)^2+d_\infty(B,C)^2-d_\infty(A,C)^2|\le \tilde c d_\infty(A,C)^{2(1+\alpha)},
\end{equation*}
et donc une utilisation usuelle de \eqref{youngestime} conduit à la conclusion. 
\end{proof}

\subsubsection{Relèvement vertical}
\paragraph{Relèvement horizontal des courbes hölderiennes.} 
Une conséquence directe et importante du théorème \ref{young} est
\begin{proposition}\label{p: youngconseque}
Soit $\gamma_x,\gamma_y\in H^\alpha([0,T],\r)$ avec $\alpha>\frac{1}{2}$. Alors la courbe $\gamma=(\gamma_x,\gamma_y,-\Lift_{\gamma_x,\gamma_y}) \in H^\alpha([0,T],\h)\cap H^\alpha([0,T],\r^3)$ est hölderienne à la fois en métrique d'Heisenberg et en métrique euclidienne. 

Réciproquement, soit $\gamma=(\gamma_x,\gamma_y,\gamma_z)\in H^\alpha([0,T],\h)$ une courbe hölderienne en métrique d'Heisenberg, $\alpha>\frac{1}{2}$. Alors il existe une constante $C$ telle que pour tout $t\in [0,T]$,
$$\gamma_z(t)=C-\Lift_{\gamma_x,\gamma_y}(t).$$
\end{proposition}

\begin{remark}\label{r: souscritiqueLevy}
Il existe également un relèvement hölderien (non-unique) de toute courbe dans le plan $(\gamma_x,\gamma_y)\in H^\alpha([0,T],\r^2)$ pour $\alpha<\frac{1}{2}$, \ie une courbe $(\gamma_x,\gamma_y,\gamma_z)\in H^\alpha([0,T],\h)$ (voir \cite[proposition 3]{lyonsextension}).
\end{remark}
	
\begin{question}
Le problème de caractériser la projection $\pi(\lambda)$  pour $\lambda\in H^\frac{1}{2}([0,1],\h)$ reste ouvert à notre connaissance. Soulignons qu'il existe une contrainte supplémentaire par rapport au relèvement hölderien pour $\alpha\not=\frac{1}{2}$. 

Soit une courbe $\lambda\in H^\frac{1}{2}([0,1],\h)$, \ie $d_\infty(\lambda(h),\lambda(l))\le C|h-l|^\frac{1}{2}$ pour tous $h,l\in [0,1]$. Pour un intervalle $[s,t]\subset [0,1]$ on définit 
\begin{align*}
& V_+([t,s])=2\sup\limits_\sigma \int_\sigma (\lambda_y\,d\lambda_x-\lambda_x\,d\lambda_y), 
& V_-([t,s])=2\inf\limits_\sigma \int_\sigma (\lambda_y\,d\lambda_x-\lambda_x\,d\lambda_y), 
\end{align*}
où $\sigma=\{s=t_0<t_1<\ldots<t_{n+1}=t\}$ parcourt toutes les subdivisions de $[s,t]$. Pour tous $h>l$  
\begin{equation*}
-C(h-l)+2\det\big(\pi(\lambda(h)),\pi(\lambda(l))\big)\le \lambda_z(h)-\lambda_z(l)\le C(h-l)+2\det\big(\pi(\lambda(h)),\pi(\lambda(l))\big).
\end{equation*}
On pose $t_{i+1}=h$, $t_i=l$, on additionne suivant une subdivision $\sigma$ de $[s,t]$ et on prend $\inf$ et $\sup$ des sommes pour obtenir
\begin{equation*}
-C(t-s)+V_+([s,t])\le \lambda_z(t)-\lambda_z(s) \le C(t-s)+V_-([s,t]).
\end{equation*}
En particulier, pour tout intervalle $[s,t]\subset [0,1]$, 
\begin{equation}\label{variationaire}
0\le V_+([s,t])-V_-([s,t])\le 2C|t-s|.
\end{equation}
La condition nécessaire \eqref{variationaire} est-elle suffisante pour construire un relèvement de la classe $H^\frac{1}{2}([0,1],\h)$ de la courbe $(\lambda_x, \lambda_y)\in H^\frac{1}{2}([0,1],\r^2)$?
\end{question}

\begin{proposition}\label{relevevert}
Soit une courbe hölderienne $\{t\to (x(t),y(t),z(t))\}\in H^{\frac{1}{2}}([0,1],\h)$ telle que $x,y\in h_{1/2}([0,1],\r)$. Alors il existe une constante $C>0$ et une application $F\in C^1_H(\h,\r^2)$ avec $D_hF$ surjective telles que la courbe $2$-Ahlfors régulière $$\Gamma=\{t\to (x(t),y(t),z(t)+Ct)\}\subset \f.$$
\end{proposition}

\begin{proof} Comme $|z(t)-z(s)-2x(t)y(s)+2x(s)y(t)|\le c|t-s|$ pour $t,s\in [0,1]$, alors en prenant $C>c$ 
$$(C-s)|t-s|\le |z(t)-z(s)-2x(t)y(s)+2x(s)y(t)+C(t-s)|\le (C+c)|t-s|.$$
Vu que $x,y\in h_{1/2}$, on a uniformément $\|\pi(\Gamma(t))-\pi(\Gamma(s))\|^2=o(|t-s|)=o(|\big(\Gamma(t)^{-1}\Gamma(s)\big)_z|)$, et donc la courbe $\Gamma$ remplit la condition de Whitney (voir remarque \ref{Whit}). La régularité d'Ahlfors de $\Gamma$ résulte de l'equivalence bi-lipschitzienne de $(\Gamma,d_\infty)$ à $([0,1],|\cdot|^\frac{1}{2})$. 
\end{proof}

\begin{Def}\label{d: relevementvertic}
Soit $\{t\to (x(t),y(t))\}\in C^0([0,T],\r^2)$ une courbe telle que $\Lift_{x,y}(t)$ existe pour tout $t\in [0,T]$. On appelle \emph{relèvement vertical} de $\{t\to (x(t),y(t))\}$ une courbe de la forme
\begin{equation*}
[0,T]\ni t\to \big(x(t),y(t),z_0-\Lift_{x,y}(t)\big)+\big(0,0,t\big), \quad z_0\in \r. 
\end{equation*}
\end{Def}

\begin{figure}[h]\centering
\includegraphics[width=65mm]{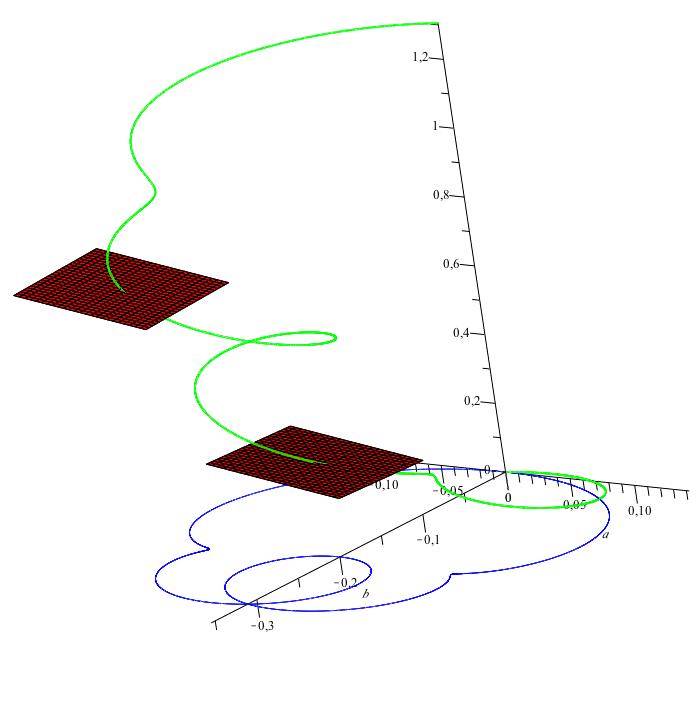}
\includegraphics[width=57mm]{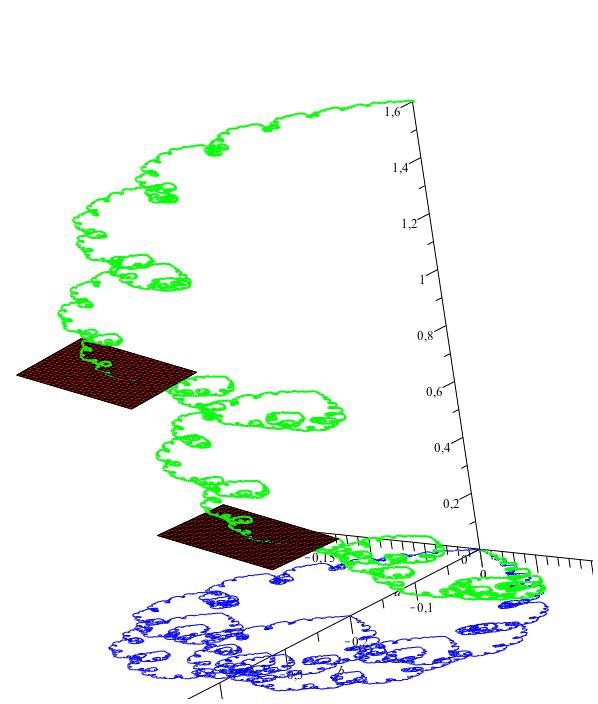}
\caption{\footnotesize Relèvement vertical de la courbe de Weierstrass $t\to \big(\sum\limits_{i=0}^n 2^{-\frac{i}{2}}\sin (2\pi 2^i t),  \sum\limits_{i=0}^n 2^{-\frac{i}{2}}(1-\cos (2\pi 2^i t))\big)$ pour $n=3$ (à gauche) et $n=20$ (à droite). On voit aussi l'intersection du relèvement avec deux plans horizontaux. (Les images sont faites à l'aide de \textbf{Maple}.)}  
\end{figure}

\begin{remark}\label{holderriem}
Si $\gamma\in H^{\alpha}([0,T],\r^2)$, $1\ge\alpha>\frac{1}{2}$, alors son relèvement vertical  appartient également à $H^{\alpha}([0,T],\r^3)$ muni de la norme euclidienne.  En effet, d'après le théorème \ref{young}
\begin{multline*}
 |Z_{\gamma_x,\gamma_y}(t)-Z_{\gamma_x,\gamma_y}(s)+t-s|\le |t-s|+|Z_{\gamma_x,\gamma_y}(t)-Z_{\gamma_x,\gamma_y}(s)-2\big(\gamma_x(t)\gamma_y(s)-\gamma_y(t)\gamma_x(s)\big)|\\
 +2|\gamma_x(t)||\gamma_y(s)-\gamma_y(t)| +2|\gamma_y(t)||\gamma_x(s)-\gamma_x(t)|\\ \le |t-s|+C_\alpha\|\gamma\|^2_\alpha|t-s|^{2\alpha}
 + 4\|\gamma\|_\infty\|\gamma\|_\alpha |t-s|^\alpha\le C(\gamma)|t-s|^\alpha,\quad C(\gamma)\le \infty.
\end{multline*}
\end{remark}
\begin{lemma}\label{l: vertcaract}
Vues localement, les courbes de niveau $\Gamma$ d'une application $F\in C^{1,\alpha}_H(\h,\r^2)$, $\alpha>0$, $D_hF$ surjective, sont exactement les relèvements verticaux des courbes $\gamma\in H^{\frac{1+\alpha}{2}}([0,T],\r^2)$.  
\end{lemma}
\begin{proof}
D'après le corollaire \ref{regularitealpha} une courbe verticale $\Gamma\subset \f$ est fortement Ahlfors régulière. On choisit le paramétrage naturel (par $\H^2_\infty\corn \Gamma$) sur $\Gamma$; selon l'inégalité \eqref{deuxalpha} $\pi(\Gamma)\in H^{\frac{1+\alpha}{2}}$. La courbe $\Gamma$ est donc un relèvement vertical de $\pi(\Gamma)$ d'après  \eqref{formecanonique}.

Soit maintenant $\gamma\in H^{\frac{1+\alpha}{2}}([0,T],\r^2)$ et $\Gamma$ son relèvement vertical. Pour $t,s\in [0,T]$ d'après le théorème \ref{young}
\begin{multline*}
d_\infty(\Gamma(t),\Gamma(s))^2\ge |-\Lift_{\gamma_x,\gamma_y}(t)+\Lift_{\gamma_x,\gamma_y}(s) -2\det(\gamma(t),\gamma(s)) +t-s|\\ \ge |t-s|-C_\alpha\|\gamma\|_\frac{1+\alpha}{2}^2|t-s|^{1+\alpha}\ge c |t-s|\ge c\|\gamma\|_\frac{1+\alpha}{2}^{-1} \|\gamma(t)-\gamma(s)\|^\frac{2}{1+\alpha}, \quad c>0, 
\end{multline*}
pourvu que $T$ soit suffisamment petit (à $\|\gamma\|_\frac{1}{2}$ fixée). Ainsi, pour conclure il suffit de poser $D_hF(A)(x,y,z)=(x,y,0)$ et $F(A)=0$ pour $A\in \Gamma$ et d'appliquer le théorème  \ref{Whitneyholder} de prolongement de Whitney pour les fonctions de $C_H^{1,\alpha}$. \end{proof}

\subsubsection{Dimension euclidienne.}
Considérons comme avant une courbe verticale $\Gamma=\f\cap U(0)$, $F\in C^1_H(\h,\r^2)$ avec $D_hF$ surjective. 
\begin{remark}D'après \cite{comdim} on a $1\le \dim_E\Gamma\le \dim_h\Gamma=2$.
Comme la projection $\pi$ est $1$-lipschitzienne au sens euclidien, $\H^d_E(\Gamma)\ge\H^d_E(\pi(\Gamma))$ et $\dim_E\Gamma \ge \dim_E \pi(\Gamma)$. Si $\Gamma$ est $2$-Ahlfors régulière, alors quitte à changer le paramétrage  $\pi(\Gamma)\in h_{1/2}$ et, par conséquent, $\pi(\Gamma)$ est d'aire nulle dans le plan : $\mathcal{L}^2(\pi(\Gamma))=0$. 
\end{remark}
\begin{lemma} \label{dimriem} La dimension euclidienne des courbes verticales $\Gamma$ peut prendre toute valeur dans l'intervalle $[1,2]$. 
\end{lemma}
\begin{proof}
Pour $1\le\beta<2$ on peut toujours trouver une courbe "quasi-hélix" d'exposant $\beta^{-1}$, \ie $\gamma:[0,1]\to \r^2$ telle que $c^{-1}|t-s|^\frac{1}{\beta}\le \|\gamma(t)-\gamma(s)\|\le c|t-s|^\frac{1}{\beta}$, $0<c<\infty$, $s,t\in [0,1]$. De telles courbes peuvent être constuites comme des courbes de Von Koch autosimilaires, voir  \cite{tricot}, par exemple. On a bien $\dim_E\gamma= \beta$. D'après le théorème \ref{young} la courbe $\{t\to (\gamma_x(t),\gamma_y(t), Z_{\gamma_x,\gamma_y}(t))\}$ appartient à $H^{\beta^{-1}}([0,1],\h)$. En prenant son relèvement vertical (proposition \ref{relevevert}) on obtient une courbe de niveau $\Gamma$ avec $\dim_E\Gamma=\beta$ (remarque \ref{holderriem}). 

Pour la dimension $\beta=2$ voir l'exemple ci-dessous. 
\end{proof}
 
\begin{exemple}
Il existe une courbe de niveau $\Gamma$ $2$-Ahlfors régulière pour laquelle $$\dim_h\Gamma=\dim_E\Gamma=\dim_E\pi(\Gamma)=2.$$
\end{exemple}
\begin{proof}

L'idée est de construire d'abord une courbe $\gamma:[0,1]\to \r^2$ dans le plan de dimension euclidienne $2$ dont le module de continuité $\|\gamma(t)-\gamma(s)\|^2\le |t-s| m(|t-s|)$ satisfait $\sum\limits_{n=0}^\infty m(2^{-n})< \infty$ et d'utiliser par la suite le relèvement vertical pour construire la courbe $\Gamma$ avec les propriétés voulues.  La courbe $\gamma$ dans notre exemple sera donnée comme une courbe de Von Koch, où à chaque étape d'itération le facteur de similitude croît. 

On fixe une suite monotone  de nombres $\{h_n\}\in (0,\frac{1}{2})$ telle que $h_n \searrow 0$ quand $n\to \infty$. On se donne deux points $A_0^0$ et $A^0_1$ avec $l_0=\|A^0_1-A^0_0\|=1$ et on définit par récurrence une suite de courbes $\gamma_n:[0,1]\to \r^2$, $n=1,\ldots,\infty$.   
La courbe $\gamma_n$ est linéaire sur tout intervalle dyadique $I^n_i:=[\frac{i}{2^n},\frac{i+1}{2^n}], i=1,\ldots, 2^n-1,$  et $\gamma(I_i^n)=[A^n_i,A^n_{i+1}]$, où les segments $[A^n_i,A^n_{i+1}]$ sont tous de longueur euclidienne $l_n=2^{-n(\frac{1}{2}+h_n)}$. Les sommets satisfont $A_{2i}^{n+1}=A_i^n, i=0,\ldots, 2^n$ (voir fig. \ref{koch}). 
\begin{figure}[h]\centering
\begin{tikzpicture}[scale=2.5] 
\draw [red,dashed] (0,0)-- (2,0);
\draw[red] (0,0)--(1,0.8)--(2,0);
\draw (0,0) node[left]{\footnotesize $A^1_0$};
\draw (1,0.8) node[above]{\footnotesize $A^1_1$};
\draw (2,0) node[right]{\footnotesize $A^1_2$};
\draw (0.5,0.4) node[above]{\footnotesize $l_1$};
\draw (1.5,0.4) node[above]{\footnotesize $l_1$};
\draw (1,0) node[below]{\footnotesize $l_0$};
\end{tikzpicture}
\begin{tikzpicture}[scale=2.5] 
\draw[red,dashed] (0,0)--(1,0.8)--(2,0);
\draw[red] (0,0)--(0.85, 0.1)--(1,0.8)--(1.15,0.1)--(2,0);
\draw (0,0) node[left]{\footnotesize $A^2_0$};
\draw (0.8,0.1) node[below]{\footnotesize $A^2_1$};
\draw (1,0.8) node[above]{\footnotesize $A^2_2$};
\draw (1.2,0.1) node[below]{\footnotesize $A^2_3$};
\draw (2,0) node[right]{\footnotesize $A^2_4$};
\draw (0.5,0.4) node[above]{\footnotesize $l_1$};
\draw (1.5,0.4) node[above]{\footnotesize $l_1$};
\draw (1.5,0.05) node[above]{\footnotesize $l_2$};
\draw (0.5,0.05) node[above]{\footnotesize $l_2$};
\draw (0.85,0.33) node[above]{\footnotesize $l_2$};
\draw (1.17,0.33) node[above]{\footnotesize $l_2$};
\end{tikzpicture}
\caption{\label{koch} \footnotesize La courbe $\gamma_1$ (à gauche) et $\gamma_2$ (à droite)} 
\end{figure}
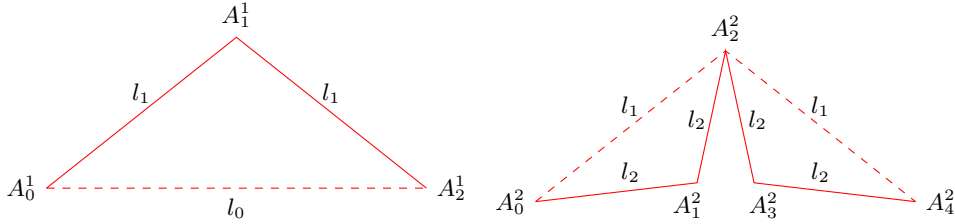

Pour tous points $A^n_i$ et $A^n_j$ tels que $0<|j-i|<2^{n-r}$ on démontre par récurrence sur $n=r+1,\cdots$, que 
$$\|A^n_i-A^n_j\|\le 2 \sum\limits_{k=r+1}^n l_k.$$  
On montre facilement que la suite $\gamma_n$ converge uniformément sur $[0,1]$ vers une courbe limite $\gamma$, et que pour $|t-s|\le 2^{-r}$ on a l'estimation suivante :
\begin{equation*}
\|\gamma(t)-\gamma(s)\|\le 2 \sum\limits_{k=r+1}^\infty l_k.
\end{equation*}
 On en déduit que le module de continuité $m$ de $\gamma$ satisfait
 $m(2^{-r})\le 2^{r}\big(2\sum\limits_{k=r+1}^\infty l_k \big)^2$.
 On remarque qu'en particulier $\gamma\in h_{1/2}$ si $kh_k\to 0$ lorsque $k\to \infty$.
 La série $\sum\limits_{n=0}^\infty m(2^{-n})$ est majorée (à des constantes multiplicatives et additives près) par 
\begin{multline}\label{kochserie}
\sum\limits_{r=0}^\infty m(2^{-r}) \lesssim \sum\limits_{r=0}^\infty2^r\big(\sum\limits_{k=r+1}^\infty l_k \big)^2 \le \sum\limits_{r=0;\, j,k=1}^\infty 2^rl_kl_j \mathbf{1}_{(k\ge r+1,j\ge r+1)} \lesssim \\
\sum\limits_{j,k=1}^\infty 2^{\min\{k,j\}} l_kl_j\le \big(\sum\limits_{k=1}^\infty 2^{\frac{k}{2}}l_k\big)^2.
\end{multline}
 On conclut donc que si la somme $\sum\limits_{n} 2^{-nh_n}<\infty$, alors $\sum\limits_r m(2^{-r})<0$ converge. Or, en utilisant la distribution de masse naturelle sur $\gamma$ il est facile de voir que $\dim_E \gamma=2$ pourvu que $h_n\searrow 0$.
 
 Prenons par exemple $h_n=(n\ln(n+1)^2+2)^{-1}$ et définissons la courbe $\gamma:[0,1]\to \r^2$ comme ci-dessus. Compte tenu de l'estimation \eqref{kochserie}, le lemme \ref{couture}  (avec $ \mu(t,s)=2\big(x(t)y(s)-y(t)x(s)\big)$) assure que la courbe $t\to (\gamma_x(t),\gamma_y(t),Z_{\gamma_x,\gamma_y}(t))$ est bien définie et appartient à $h_{1/2}([0,1],\h)$. La construction de la courbe $\Gamma$ est alors faite par la proposition \ref{relevevert}.
\end{proof}

\subsection{Exemples irréguliers\label{exempleirregulier}}
Dans cette sous-section nous construisons des courbes verticales "rugueuses" $\Gamma_{1,2}$ pour lesquelles $\H^2_\infty(\Gamma_1)=\infty$ ou $\H_\infty^2(\Gamma_2)=0$. 
\begin{remark}
Compte tenu de l'inégalité de la coaire  de \cite{magncoareaineq} pour $\mathcal{L}^2$-presque tout $a\in \r^2$  la mesure $\H^2_\infty(F^{-1}(a)\cap K)<\infty$ est finie pour tout compact $K\in \h$.
\end{remark}

\subsubsection{Séries lacunaires.}\label{ondelettes} Le but de cette sous-section est d'étudier plus en détail l'existence et des estimations de l'intégrale de Stieltjes $\int x\,dy$ au delà de la régularité du lemme \ref{couture}. Pour ce faire on utilisera des fonctions hölderiennes construites comme des séries de Fourier lacunaires. 

Etant donné $\phi_n(t):=(2\pi)^{-1}\exp(-2\pi I 2^n t)$ et $\psi_n(t):=(2\pi)^{-1}\exp(2\pi I 2^n t)$ on considère la série de fonctions: 
\begin{equation}\label{ondelettedecomp}
 f(t)=\sum\limits_{n=0}^\infty 2^{-\frac{n}{2}}\big(a_n\phi_n(t)+b_n\psi_n(t)\big), \quad a_n, b_n\in \mathbb{C}.
\end{equation} Dans tous les cas qui nous intéressent nous supposons que $f\in \bigcap\limits_{\epsilon>0} H^\frac{1-\epsilon}{2}([0,1])$.
Pour $h>0$ on note par $k(h)$ l'unique entier tel que $ 2^{-k(h)-1}\le h< 2^{-k(h)}$.
Pour $h>0$,
\begin{multline*}
 \big| \sum\limits_{n=0}^\infty 2^{-\frac{n}{2}}a_n\big(\phi_n(t+h)-\phi_n(t)\big)\big|\le  (2\pi)^{-1}\sum\limits_{n=0}^{k(h)} 2^{-\frac{n}{2}}|a_n| |\exp(2\pi I2^nh)-1|+  \sum\limits_{n=k(h)+1}^\infty 2^{-\frac{n}{2}}|a_n|\\
 \le \sum\limits_{n=0}^{k(h)} 2^{-\frac{n}{2}}|a_n| h2^n+ 2^{-\frac{k(h)}{2}}\sum\limits_{n=k(h)+1}^\infty 2^{-\frac{n-k(h)}{2}}|a_n|\le h^\frac{1}{2} L_{k(h)}(\{a_n\}),
\end{multline*}
où $L_k(\{a_n\}):=2\big(\sum\limits_{n=0}^{k}|a_n| 2^{-\frac{k-n}{2}}+\sum\limits_{n=k+1}^\infty 2^{-\frac{n-k}{2}}|a_n|\big)$, et donc 
\begin{equation}
|f(t+h)-f(t)|^2\le 2h\big(L_{k(h)}(\{a_n\})^2+L_{k(h)}(\{b_n\})^2\big).
\end{equation}
On voit alors que $L_k(\{a_n\})\le 8\|a_n\|_\infty$, où $\|\cdot\|_\infty$ désigne la norme usuelle de l'espace $l_\infty$. Par conséquent, si $\max\limits_{n\ge 0} (|a_n|+|b_n|)\le C<\infty$ alors $f\in H^{\frac{1}{2}}$ et $\|f\|_{H^\frac{1}{2}}\le 8C$. On remarque également que si  $\lim a_n=\lim b_n=0$, alors $L_k(\{a_n\})+L_k(\{b_n\})\underrel\longrightarrow{k\to\infty}0$ et donc  $f \in h_{1/2}$. 
Avec $f$ considérons encore une fonction $g$ dont la décomposition est donnée par 
\begin{equation}\label{ondelettedecomp1}
g(t)=\sum\limits_{n=0}^\infty 2^{-\frac{n}{2}}\big(c_n\phi_n(t)+d_n\psi_n(t)\big).
\end{equation}

Soit $\sigma=\{0=t_0<t_1<\ldots<t_{N}<t_{N+1}=1\}$ une subdivision de $[0,1]$; on note  $\Delta t_i:=t_{i+1}-t_i$ et $\|\sigma\|:=\max\limits_i |\Delta t_i|$. Notre but est d'estimer la somme de Stieltjes 
$ \int_\sigma f\,dg := \sum\limits_{i=0}^N f(t_i)(g(t_{i+1})-g(t_i))$. 
On décompose $\int_\sigma f\,dg =D_\sigma+S_\sigma$ en deux termes, appelés diagonal ($D$) et secondaire ($S$):
\begin{align*}
&D_\sigma=\sum\limits_{n=0}^\infty 2^{-n}\big(a_n d_n  \int_\sigma \phi_n\,d\psi_n+b_n c_n  \int_\sigma \psi_n\,d\phi_n\big); \\
& S_\sigma=\sum\limits_{n\not=m}2^{-\frac{n+m}{2}}\big(a_n d_m  \int_\sigma \phi_n\,d\psi_m+b_n c_m  \int_\sigma \psi_n\,d\phi_m\big)\\
&+\sum\limits_{n,m=0}^\infty a_n c_m 2^{-\frac{n+m}{2}} \int_\sigma \phi_n\,d\phi_m +b_n d_m 2^{-\frac{n+m}{2}} \int_\sigma \psi_n\,d\psi_m.
\end{align*}

\textbf{Estimation du terme secondaire.}
 Remarquons que  $\int\limits_0^1 \phi_n\,d\psi_m=\int\limits_0^1 \psi_n\,d\phi_m=0$ pour $n\not=m$ non-négatifs, ainsi que $\int\limits_0^1 \psi_n\,d\psi_m=\int\limits_0^1 \phi_n\,d\phi_m=0$ pour tous $n,m\ge 0$. Estimons, par exemple, le terme 
\begin{multline*}
S_{a,c}:=\sum\limits_{n,m=0}^\infty a_n c_m 2^{-\frac{n+m}{2}} \big(\int_\sigma \phi_n\,d\phi_m \pm \int\limits_0^1\phi_n\,d\phi_m\big)=  \\
=\sum\limits_{i=0}^N\sum_{n,m=0}^\infty a_n c_m 2^{-\frac{n+m}{2}} \int\limits_{t_i}^{t_i+\Delta t_i}\big(\phi_n(t_i)-\phi_n(l)\big)\phi_m'(l)\,dl.
\end{multline*}
En fonction de la taille de $\Delta t_i$ on utilisera l'une des quatre bornes possibles qui sont faciles à retrouver:
\begin{equation}\label{quatrestim} \big|\int\limits_{t_i}^{t_i+\Delta t_i}\big(\phi_n(t_i)-\phi_n(l)\big)\phi_m'(l)\,dl\big|\le 
\min \left\{\begin{array}{lcc} 2^{n+m}(\Delta t_i)^2;\\ 2^{n+1}\Delta t_i;\\ 2^m\Delta t_i;\\ 1.  
\end{array}\right.
\end{equation}
 Pour  $0\le i\le N$ fixé, on pose $k_i=k(\Delta t_i)$ et suivant \eqref{quatrestim} on obtient  
\begin{multline*}
\sum_{n,m=0}^\infty |a_n c_m|\, 2^{-\frac{n+m}{2}} \big|\int\limits_{t_i}^{t_i+\Delta t_i}\big(\phi_n(t_i)-\phi_n(l)\big)\phi_m'(l)\,dl\big| \le
 \sum_{n,m\le k_i}|a_nc_m| 2^\frac{n+m}{2}2^{-2k_i}\\
+\sum_{m>k_i\ge n}2|a_nc_m|2^\frac{n-m}{2}2^{-k_i} +\sum_{n>k_i\ge m} |a_nc_m|2^\frac{m-n}{2}2^{-k_i} +\sum_{n,m> k_i}|a_nc_m| \, 2^{-\frac{n+m}{2}}\\
\le 2^{-k_i}\bigg( \sum_{n=0}^{k_i}\frac{|a_{n}|}{2^{(k_i-n)/2}}\sum_{m=0}^{k_i}\frac{|c_{m}|}{2^{(k_i-m)/2}}+2\sum_{n=0}^{k_i}\frac{|a_{n}|}{2^{(k_i-n)/2}}\sum_{m=k_i+1}^{\infty}\frac{|c_{m}|}{2^{(m-k_i)/2}}\\
+\sum_{m=0}^{k_i}\frac{|c_m|}{2^{(k_i-m)/2}}\sum_{n=k_i+1}^{\infty}\frac{|a_{n}|}{2^{(n-k_i)/2}}
+ \sum_{n=k_i+1}^{\infty}\frac{|a_n|}{2^{(n-k_i)/2}}\sum_{m=k_i+1}^{\infty}\frac{|c_m|}{2^{(n-k_i)/2}}\bigg).
\end{multline*} 
 On en déduit que 
$$ |S_{a,c}|\le 2 \sum_{i=0}^N 2^{-k_i}L_{k_i}(\{a_n\}) L_{k_i}(\{c_m\}).$$ 
 Exactement de la même façon nous pouvons majorer les autres termes de $S_\sigma$, et avoir l'estimée suivante : 
\begin{equation*}
|S_\sigma|\le 2\sum_{i=0}^N 2^{-k_i} \big(L_{k_i}(\{a_n\})+L_{k_i}(\{b_n\})\big)\big(L_{k_i}(\{c_m\})+L_{k_i}(\{d_m\})\big).
\end{equation*}

\textbf{Estimation du terme diagonal.}
Pour ce choix de $\phi$ et $\psi$ on calcule
\begin{align*}
&(2\pi)^2 \phi_n(t_i)\big(\psi_n(t_{i+1})-\psi_n(t_i)\big)= \exp(2\pi I 2^n \Delta t_i)-1;
\\& (2\pi)^2 \psi_n(t_i)\big(\phi_n(t_{i+1})-\phi_n(t_i)\big)=\exp(-2\pi I 2^n\Delta t_i)-1.
 \end{align*}
 L'estimée $|\exp(xI)-1-Ix|\le x^2/2$ justifie l'apparition de la partie principale dans $D_\sigma$ :
\begin{multline*}
(2\pi)^2D_\sigma=\sum_{i=0}^N\sum_{n=0}^\infty 2^{-n}\big(a_n d_n( \exp(2\pi I 2^n \Delta t_i)-1)+b_n c_n(\exp(-2\pi I 2^n\Delta t_i)-1) \big)
\\= 2\pi I\sum_{i=0}^N\sum_{n=0}^{k_i}  \Delta t_i(a_n d_n -b_n c_n)+\tilde S_\sigma.
\end{multline*}
Le terme $\tilde S_\sigma$ satisfait l'inégalité 
\begin{multline*}
|\tilde S_\sigma|\le \sum_{i=0}^N\sum_{n=k_i+1}^\infty 2^{-n+1}(|a_nd_n|+|b_nc_n|) + \sum_{i=0}^N\sum_{n=0}^{k_i} 2^{n-1} 2^{-2k_i} (|a_nd_n|+|b_nc_n|) \\
\le \sum_{i=0}^N 2^{-k_i} \Big(\sum_{n= k_i+1}^\infty 2^{-n+k_i+1} (|a_nd_n|+|b_nc_n|)+\sum_{n=0}^{k_i} 2^{n-1-k_i} (|a_nd_n|+|b_nc_n|) \Big) \\  \le \sum_{i=0}^N 2^{-k_i} \big(L_{k_i}(\{a_n\})L_{k_i}(\{d_m\})+L_{k_i}(\{c_m\})L_{k_i}(\{b_n\})\big).
\end{multline*}
Au final nous obtenons que (avec $C\le 10$)
\begin{align}\label{totalformuleStiel}
&\int_\sigma f\,dg = (2\pi)^{-1}I \sum_{i=0}^N\Delta t_i \sum_{n=0}^{k_i}  (a_n d_n -b_n c_n) + R_\sigma(a_n,b_n,c_n,d_n),\\ \label{totalformuleStiel1}
& |R_\sigma(a_n,b_n,c_n,d_n)|\le C\sum_{i=0}^N\Delta t_i \big(L_{k_i}(\{a_n\})+L_{k_i}(\{b_n\})\big)\big(L_{k_i}(\{c_m\})+L_{k_i}(\{d_m\})\big).
\end{align}
 Vu que $\sum_{i=0}^N\Delta t_i=1$, on arrive à l'estimée
\begin{equation}
 |R_\sigma(a_n,b_n,c_n,d_n)|\le C\max\limits_{k\ge k(\|\sigma\|)} \big(L_k(\{a_n\})+L_k(\{b_n\})\big)\big(L_k(\{c_m\})+L_k(\{d_m\})\big).
\end{equation}
En particulier, le reste vérifie $|R_\sigma|\le \tilde C(\|a_n\|_\infty+\|b_n\|_\infty)(\|c_m\|_\infty+\|d_m\|_\infty)$, et si en plus $|a_n|+|b_n|\to 0$ ou $|c_m|+|d_m|\to 0$, alors 
$|R_\sigma|\to 0$ quand $\|\sigma\|\to 0$.   
\begin{proposition}\label{p: stieltjesondel}
Soit $\|a_n\|_\infty+\|b_n\|_\infty<\infty$ et $|c_m|+|d_m|\underrel\longrightarrow{m\to \infty} 0$. Alors, pour $f$ et $g$ dont les décompositions sont données par \eqref{ondelettedecomp} et \eqref{ondelettedecomp1}, 
\begin{equation}\label{ondeconverge}
\lim\limits_{\|\sigma\|\to 0}\int_\sigma f\,dg =(2\pi)^{-1}I \sum_{n=0}^{\infty}(a_n d_n -b_n c_n),
\end{equation}
si et seulement si la dernière série converge (pas forcement absolument).\end{proposition}
\begin{proof}
En supposant que la série converge et vu que $\sum_{i=0}^N \Delta t_i =1$ on a l'égalité suivante 
$$
 \sum_{i=0}^N\Delta t_i \sum_{n=0}^{k_i}  (a_n d_n -b_n c_n)- \sum_{n=0}^{\infty}(a_n d_n -b_n c_n)= \sum_{n=k(\|\sigma\|)}^{\infty}(a_n d_n -b_n c_n)\sum_{i=0}^N\Delta t_i (\mathbf{1}_{\{n\le k_i\}}-1).
$$
La dernier terme tend vers $0$ lorsque $\|\sigma\|\to 0$ car la fonction $n\longrightarrow \sum_{i=0}^N\Delta t_i (\mathbf{1}_{\{n\le k_i\}}-1)$ est monotone et bornée.
\end{proof}
En répartissant bien les poids convexes $\Delta t_i$ on obtient 
\begin{proposition}\label{ondeladherence} Soit $\|a_n\|_\infty+\|b_n\|_\infty<\infty$ et $|c_m|+|d_m|\underrel\longrightarrow{m\to \infty} 0$. Pour la simplicité supposons que les fonctions correspondantes $f\in H^\frac{1}{2}$ et $g\in h_{1/2}$ sont à valeurs réelles, \ie $\bar a_n =b_n$ et $\bar c_n=d_n$ pour tout $n\ge 0$. Alors les valeurs d'adhérence de $\int_\sigma f\,dg$ remplissent lorsque $\|\sigma\| \to 0$ l'intervalle fermé
$$ \overline{ \int_\sigma f\,dg}_{\|\sigma\|\to 0}=\big[ \liminf\limits_{n\to \infty} A_n ;\limsup_{n\to \infty}  A_n\big], \quad A_n:=\pi^{-1}\sum_{k=0}^{n}\Im(a_kd_k).$$
Dans le cas, où $\|a_n\|_\infty+\|b_n\|_\infty<\infty$ et $\|c_m\|_\infty+\|d_m\|_\infty<\infty$
$$\overline{ \int_\sigma f\,dg}_{\|\sigma\|\to 0}\subset \big[ \liminf\limits_{n\to \infty} A_n -l(f,g);\limsup_{n\to \infty}  A_n+l(f,g)\big],$$
où $l(f,g)=C(\|a_n\|_\infty+\|b_n\|_\infty)(\|c_m\|_\infty+\|d_m\|_\infty)$, $0<C<\infty$ est une constante.
\end{proposition}

\subsubsection{Relèvement des courbes rugueuses}

On fixe une fonction $\epsilon:\r_+\to (0,1]$ croissante telle que $\epsilon(t)\searrow 0$ lorsque $t\searrow 0$. On suppose aussi que la fonction définie par $h(t)=\epsilon(t)^{-1}t$ pour $t>0$ et $h(0)=0$ est continue sur $[0,1]$. 
   Soit $\gamma:[0,T]\to \r^2$ une courbe continue dans le plan. Etant donnés $\epsilon$ et $\gamma$ on définit une fonction continue $\Delta z:[0,T]^2\to \r$ : pour $0\le t, s\le T$ on pose
\begin{equation*}
\Delta z(t,s)=h\big(\|\gamma(t)-\gamma(s)\|^2\big)+2\det(\gamma(t),\gamma(s)).
\end{equation*}
 Pour $t\in[0,T] $ on définit un nombre, fini ou infini,
 \begin{equation*}
z(t)=\operatorname{Var}(\Delta z)_{[0,t]}=\sup\left\{\sum\limits_{i=0}^{N}\Delta z(t_{i+1},t_i)\, \mid\, 0=t_0< t_1< \ldots< t_{N+1}=t \right\}.
\end{equation*}
\begin{lemma}[relèvement rugueux]\label{verticalgros} Soient $\epsilon$, $\gamma$, $\Delta z$ et $z$ comme ci-dessus et supposons $z(T)<\infty$. Alors $z:[0,T]\to \r$ est continue et la courbe $\Gamma([0,T])\subset \h$, $\Gamma(t)=(\gamma(t),z(t))$, est une courbe verticale pour une certaine application $F\in C^1_H(\h,\r^2)$, avec $D_hF$ surjective sur $\Gamma([0,T])\subset \f$.   
\end{lemma} 
\begin{proof} La fonction $\Delta z$ étant continue, si sa variation $z(t)$ est bornée, alors elle est aussi continue. Il est clair que $\operatorname{Var}(\Delta z)_{[0,t]}\ge \operatorname{Var}(\Delta z)_{[0,s]}+\operatorname{Var}(\Delta z)_{[s,t]}$ pour tous $0\le s\le t\le T$. On voit donc que 
\begin{equation*}
z(t)-z(s)\ge \operatorname{Var}(\Delta z)_{[s,t]}\ge \Delta z(t,s),
\end{equation*} ce qui donne
$$\epsilon\big(\|\gamma(t)-\gamma(s)\|^2\big)\big(z(t)-z(s)-2\det(\gamma(t),\gamma(s))\big)\ge \|\gamma(t)-\gamma(s)\|^2.$$ On note $A=\Gamma(t)$, $B=\Gamma(s)$; vu que $\epsilon\le1$ est croissante, on arrive à  
$$\epsilon(d_\infty(A,B)^2)d_\infty(A,B)^2\ge \|\pi(A)-\pi(B)\|^2,$$
ce qui est exactement la condition de Whitney pour l'ensemble $\Gamma([0,T])$.
\end{proof}

Soit comme avant $\Gamma=\f\cap U(0)$ une courbe verticale de niveau de l'application $F\in C^1_H(\h,\r^2)$ avec $D_hF$ surjective, paramétrée par $[0,1]$ de façon croissante. 
\begin{remark}\label{relevepetitholder}
On voit que les sommes de Stieltjes $2\sum_{i=0}^N \det\big(\pi(\Gamma(t_{i+1})),\pi(\Gamma(t_i))\big)\le z(\Gamma(1))-z(\Gamma(0))$ sont majorées uniformément (quelle que soit une subdivision $\sigma=\{0=t_0< t_1< \ldots< t_{N+1}=1\}$).
Si maintenant $\gamma\in h_{1/2}([0,1], \r^2)$ et les sommes de Stieltjes $\sum_{i=0}^N \det(\gamma(t_{i+1}),\gamma(t_i))\le C<\infty$ sont majorées uniformément, alors $\gamma$ admet un relevé rugueux comme dans la proposition \ref{verticalgros}. En effet, si $\gamma$ n'est pas réduite à un point il suffit de poser
   $$\epsilon(\delta)=c^{-1}\sup\limits_{\delta_0\le \delta}\big\{ \max\limits_{\|\gamma(t)-\gamma(s)\|^2\ge \delta_0} \delta_0|t-s|^{-1}\big\}, \ \delta\in (0,l];\quad \epsilon(\delta)=\epsilon(l), \  \delta>l; $$
avec $l=\diam(\gamma([0,1]))>0$, et $c>0$  est une constante assez grande.
\end{remark}

\begin{exemple}\label{mesureinfinie}
Il existe une courbe verticale $\Gamma$ pour laquelle $\H^2_\infty(\Gamma)=\infty$.
\end{exemple}
\begin{proof}
On prend deux fonctions $f$ et $g$ à valeurs réelles de la forme de \eqref{ondelettedecomp} et \eqref{ondelettedecomp1} respectivement. On choisit également $\{a_n\}$ et $\{d_m\}$ de sorte que $\lim a_n=\lim d_m=0$ et que la suite $\sum_{n=0}^{N}\Im(a_nd_n)$ tende vers $\infty$ lorsque $N\to \infty$. La courbe $\gamma:=(f,g) \in h_{1/2}([0,1],\r^2)$) et les sommes de Stieltjes $\sum_\sigma \det(\gamma(t_{i+1}),\gamma(t_i))=\int_\sigma g\,df-\int_\sigma f\,dg$ sont majorées uniformément ce qui découle des formules \eqref{totalformuleStiel}\eqref{totalformuleStiel1} et de l'égalité \eqref{integrationparpartie}. D'après la remarque \ref{relevepetitholder}, $\gamma$ admet un relevé rugueux $\Gamma$. En utilisant la formule de l'aire \eqref{airedcarre}, la proposition \ref{ondeladherence} et la remarque \ref{integrationpartie} on obtient que $$\H^2_\infty(\Gamma)=z(\Gamma(1))-z(\Gamma(0))+2\liminf\limits_{\|\sigma\|\to 0} \big(\int_\sigma f\,dg -\int_\sigma g\,df\big)=\infty. \qedhere$$
\end{proof}

\begin{exemple}\label{mesurenulle}
Il existe une courbe verticale $\Gamma$ pour laquelle $\H^2_\infty(\Gamma)=0$.
\end{exemple}
\begin{proof}
On prend deux fonctions $f$ et $g$ à valeurs réelles de la forme de \eqref{ondelettedecomp} et \eqref{ondelettedecomp1} respectivement avec $\lim b_n=\lim c_m=0$. On pose $\gamma:=(f,g) \in h_{1/2}([0,1],\r^2)$. Remarquons que si 
\begin{equation*}
\sum_{i=0}^N h(\|\gamma(t_{i+1})-\gamma(t_i)\|^2)+ 2\det(\gamma(t_{i+1}),\gamma(t_i))\le 2\limsup\limits_{\|\tilde \sigma\|\to 0} \sum_{\tilde \sigma}\det(\gamma(t_{k+1}),\gamma(t_k))<\infty,
\end{equation*}
quelle que soit une subdivision $\sigma=\{0=t_0< t_1< \ldots< t_{N+1}=1\}$  (avec $h(t)=t\epsilon(t)^{-1}$ comme au début de la sous-section), alors $\gamma$ admet un relevé rugueux $\Gamma$ et, d'après la formule de l'aire \eqref{airedcarre}, on aura que $\H^2_\infty(\Gamma)=0$. La remarque \ref{integrationpartie} ($fg\big|_0^1=0$) et la proposition \ref{ondeladherence} entraînent que  
$$ \limsup\limits_{\| \tilde \sigma\|\to 0} \sum_{\tilde \sigma}\det(\gamma(t_{k+1}),\gamma(t_k))= 2\limsup\limits_{\|\tilde\sigma\|\to 0} \int_{\tilde\sigma} g\,df=2\pi^{-1}\limsup\limits_{n\to \infty} \sum_{k=0}^{n}\Im(b_kc_k).$$
Tout en gardant les notations de \ref{ondelettes},
\begin{align*}
&\sum_{i=0}^N\det(\gamma(t_{i+1}),\gamma(t_i))= \int_{\sigma} g\,df-f\,dg =  2\int_{\sigma} g\,df+r_\sigma=2\pi^{-1}\sum_{i=0}^N \Delta t_i \sum_{k=0}^{k_i} \Im(b_kc_k)+r_\sigma+R_\sigma,\\
& \text{ où }|r_\sigma|\le \sum_{i=0}^N \|\gamma(t_{i+1})-\gamma(t_i)\|^2,\quad |R_\sigma|\le C \sum_{i=0}^N \Delta t_i L_{k_i}(\{b_n\})L_{k_i}(\{c_m\}).
\end{align*}
Ainsi, pour construire l'exemple voulu, il suffit d'exhiber une courbe $\gamma=(f,g) \in h_{1/2}$ telle que  
\begin{multline*}
 \|\gamma(t_{i+1})-\gamma(t_i)\|^2 \big(1+ \epsilon(\|\gamma(t_{i+1})-\gamma(t_i)\|^2)^{-1}\big) +C \Delta t_i L_{k_i}(\{b_n\})L_{k_i}(\{c_m\}) \\
 \le 4\pi^{-1} \Delta t_i \limsup\limits_{n\to \infty} \sum_{k=k_i+1}^{n} \Im(b_kc_k)<\infty, \quad i=0,\ldots, N,
\end{multline*}
pour toute subdivision $\sigma$ de $[0,1]$. Compte tenu de l'inégalité  
$$  \|\gamma(t_{i+1})-\gamma(t_i)\|^2\le 4\Delta t_i( L_{k_i}(\{b_n\})^2 +L_{k_i}(\{c_m\})^2),$$
on prend des suites $\{b_n\}$ et $\{c_m\}$ pour lesquelles 
\begin{enumerate}
\item $ 0<\tilde C \big( L_{k}(\{b_n\})^2 +L_{k}(\{c_m\})^2\big)<\limsup\limits_{n\to \infty} \sum_{l=k}^{n} \Im(b_lc_l)<\infty$ pour tout $k\ge 0$; 
\item \begin{equation*}
\lim\limits_{k\to\infty}\dfrac{ L_{k}(\{b_n\})^2 +L_{k}(\{c_m\})^2}{\limsup\limits_{n\to \infty} \sum_{l=k}^{n} \Im(b_lc_l)}=0;
\end{equation*}
\end{enumerate}
où $\tilde C> C+8$ est une constante assez grande. On peut toujours trouver de telles suites (par exemple, en considérant $-Ib_k=c_k=k^{-1}$ pour $k$ assez grand, on aura  $L_{k}(\{b_n\})^2=L_{k}(\{c_m\})^2\lesssim k^{-2}$ et $\sum_{l=k}^{\infty} \Im(b_lc_l)\gtrsim k^{-1}$). Finalement, il n'est pas très difficile de se convaincre qu'avec des suites $\{b_n\}$ et $\{c_m\}$ vérifiant les conditions ci-dessus, on peut toujours construire $\epsilon$ comme au début de la sous-section. En effet, on pose 
$$ \tilde\epsilon(\delta)=2\delta\max\limits_{t,s}  (|t-s|S_{k(|t-s|)})^{-1}, \quad S_k=\limsup\limits_{n\to \infty} \sum_{l=k+1}^{n} \Im(b_lc_l), $$
où $\max $ est pris sur l'ensemble $\{s,t \in[0,1] \mid \|\gamma(t)-\gamma(s)\|^2\ge \delta\}$. Pour $\delta>0$ donné,
  $$\tilde \epsilon(\delta)= \frac{2\delta} {|\bar s-\bar t| S_{k(|\bar s-\bar t| )}}\le \frac{\|\gamma(\bar t)-\gamma(\bar s)\|^2}{ |\bar s-\bar t| S_{k(|\bar s-\bar t| )}}\le \frac{8( L_{k(|\bar s-\bar t|)}(\{b_n\})^2 +L_{k(|\bar s-\bar t|)}(\{c_m\})^2)}{S_{k(|\bar s-\bar t| )}},$$
et donc $\tilde \epsilon(\delta)\to 0$ quand $\delta\to 0$. Or,
\begin{multline*}
 \frac{ \|\gamma(t_{i+1})-\gamma(t_i)\|^2} {\tilde\epsilon(\|\gamma(t_{i+1})-\gamma(t_i)\|^2)}+ \|\gamma(t_{i+1})-\gamma(t_i)\|^2+ C \Delta t_i L_{k_i}(\{b_n\})L_{k_i}(\{c_m\}) \\ \le \frac{\Delta t_i}{2} S_{k_i}+ \Delta t_i \Big(4+\frac{C}{2}\Big)( L_{k_i}(\{b_n\})^2 +L_{k_i}(\{c_m\})^2) \le \Delta t_i S_{k_i}.
\end{multline*}
Il ne reste qu'à modifier légèrement $\tilde \epsilon$ pour finir la construction du $\epsilon$ désiré. 
\end{proof}

 \section{Formule de la coaire}\label{s :coaire}
\begin{theorem}[Formule de la coaire]\label{theoremcoaire}
Soit $F\in C^1_H(\h,\r^2)$ et partout $\det d_h F\not =0$. Supposons aussi que toute courbe de niveau de $F$ soit équi-$f\alpha$-régulière, \ie  quelque soit $R>0$ l'estimation suivante a lieu  
\begin{equation}\label{tancoarea}
\H^2_\infty([A,B])=d_\infty(A,B)^2+o(d_\infty(A,B)^2),
\end{equation}
uniformément pour $a\in \bar B(0,R)\subset \r^2$ et pour tout intervalle $[A,B]\subset B_\infty(0,R)\cap F^{-1}(a)$ de la  courbe verticale $F^{-1}(a)$.

Alors pour tout ensemble borélien $E\subset\h$ la formule de la coaire est vérifiée
\begin{equation}\label{coarea}
\int\limits_{\r^2} \H^2_\infty\big(F^{-1}(a)\cap E\big)\, d\mathcal{L}^2(a)=\mathbf{c}\int\limits_{E} |\det d_h F(A)| \, d\mathcal{H}_\infty^4(A),
\end{equation}
où $d_h F$ est la partie horizontale de la différentielle $D_hF$ et $\mathbf{c}$ une constante géométrique.
\end{theorem}

\begin{remark}
La condition \eqref{tancoarea} est remplie, comme nous l'avons prouvé, pour $F\in C^{1,\alpha}_H$, $\alpha>0$. 
\end{remark}

\begin{remark}
A l'aide de l'inégalité de la coaire \cite{magncoareaineq} qui est générale pour les groupes de Carnot,  nous pouvons également prendre en considération l'ensemble de points singuliers (où $\det d_h F=0$)  dans  la formule de la coaire.  Ainsi la formule de la coaire \eqref{coarea} sera valable pour toute $F\in C^1_H(\h,\r^2)$ pourvu que la condition \eqref{tancoarea} soit remplie uniformément sur chaque partie compacte de l'ensemble non-singulier. C'est en particulier le cas de $F\in C^{1,\alpha}_H$,  $\alpha>0$. 
\end{remark}
 
\begin{remark}
Il sera intéressant d'examiner la validité de la formule de la coaire pour les applications $F\in C^1_H(\h,\r^2)$ (ou bien lipschitziennes) sans hypothèses de régularité supplémentaire.
\end{remark}

L'idée principale de la démonstration consiste à obtenir la formule de la coaire (uniforme) sur des surfaces $\h$-régulières   et à utiliser la formule de la coaire pour des applications scalaires de classe $C^1_H$ par la suite (réduction de dimension).

\begin{proof}
Soit $F=(f,g)$ et $F(0)=0$. Sans perdre de généralité nous supposons que la dérivée $Xf\not=0$ ne s'annule pas sur un certain voisinage de $0\in \h$. Considérons alors la surface $\h$-régulière $\mathcal{S}=f^{-1}(0)\cap U(0)$, paramétrée de la manière standard à l'aide de l'application $\Phi:\Omega\subset \r^2\to \mathcal{S}$, $\Phi(y,z)=\exp(\phi(y,z)X)(0,y,z)$, engendrée par une application scalaire $\phi$. Comme plus haut nous considérons le champ horizontal $W^\phi$ et un flot sans pénétration $\a$ de ce champ.  Du flot $\a$ on extrait un sous-flot $\a'=\big\{z\in \a : z_- \le z\le z_+\big\}$, $z_-,z_+\in \a$, $z_-<0<z_+$, tel que pour un certain $T>0$ toute courbe verticale $(g\circ \Phi)^{-1}(t)$, $t\in [-T,T]$, rencontre toute courbe horizontale de $\a'$ en un point exactement.
En coupant $\a'$, nous définissons $I$ un "rectangle curviligne", où le travail local va se réaliser
$$ I:=\big\{(y,z)\in \Omega \mid |g\circ \Phi(y,z)|\le T,\, z_-(y)\le z\le z_+(y)\big\}.$$
Notons $\Gamma_t=(g\circ \Phi)^{-1}(t)\cap I$ les courbes (de niveau) verticales.
Compte tenu de \eqref{tancoarea}, de la description géométrique de $d_\infty$ sur $\S$ (lemme \ref{metriqueR}) et de l'estimation de la divergence des courbes horizontales (lemme \ref{divergence}), on obtient 
\begin{proposition*}
La fonction $\H^2_\infty(\Gamma_t)$ est continue en $t$.
\end{proposition*}
Fixons $r>0$, qui va tendre vers $0$ par la suite. Soit $N$ entier tel que $Nr\le T <(N+1)r$. Pour tout entier $n\in [-N,N-1]$ considérons la courbe $\Gamma_{nr}$ dont les bouts on note $a_n$ et $b_n$. On constuit une suite de points $\{a_n=a_{n,0}< a_{n,1}<\ldots<a_{n,k_n}\}$ sur $\Gamma_{nr}$, vérifiant la propriété suivante : pour tout $i=0,\ldots, k_n-1$, $ d_\infty(a_{n,i},a_{n,i+1})=r$ et $d_\infty(a_{n,k_n},b_n)\le r$. Maintenant pour tout point $a_{n,i}$ on trouve une courbe horizontale du flot $\gamma_{n,i}$ qui le rencontre, et on note $b_{n,i}$ le point d'intersection de $\gamma_{n,i}$ et de la courbe verticale du niveau "suivant" $\Gamma_{r(n+1)}$.
Pour $n=-N,\ldots, N-1$ et $i=0,\ldots,k_n-1$ on note $I_{n,i}$ le "rectangle curviligne" de sommets aux points $a_{n,i}, b_{n,i},a_{n, i+1},b_{n, i+1}$, qui est délimité par les courbes verticales $\Gamma_{rn}$ et $\Gamma_{r(n+1)}$, et horizontales $\gamma_{n,i}$ et $\gamma_{n, i+1}$ (voir fig. \ref{dessincoaire}). 

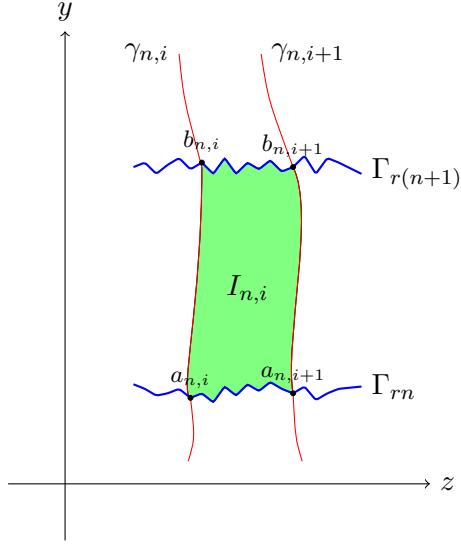
\begin{figure}\centering
\begin{tikzpicture}[scale=3]

\draw[->] (-0.25,0) -- (1.6,0) node[right] {$z$};
\draw[->] (0,-0.25) -- (0,2) node[above] {$y$};

\filldraw [fill=green!50] 
(0.6,1.42) --(0.65,1.37) --(0.7,1.44) -- (0.75,1.37)--(0.8,1.42) --(0.85,1.39)-- (0.9,1.43)--(0.95,1.38)--(1,1.4)    
.. controls (1.1,1.2) and  (0.96,0.62).. 
(1,0.4)--(0.95,0.42)-- (0.9,0.45)--(0.85,0.41)--(0.8,0.44)--(0.75,0.39)--(0.7,0.43)--(0.65,0.36)-- (0.6,0.4)-- (0.55,0.38)      
.. controls (0.5,0.4) and (0.6,0.8) .. (0.6,1.4);
\draw[red] (0.86,1.9)..controls (0.89, 1.7)..(1,1.4) .. controls (1.1,1.2) and  (0.96,0.62) .. (1.0,0.4).. controls (1.01,0.2) .. (1.04,0.1);
\draw[red] (0.5,1.9)..controls (0.52, 1.7)..(0.6,1.4)      
.. controls (0.6,0.8) and (0.5,0.4) ..(0.55,0.38).. controls (0.57,0.2) .. (0.54,0.1);
\pgfsetcornersarced{\pgfpoint{1pt}{1pt}}
\draw [blue,  thick ](0.3,1.4) --(0.35,1.42)-- (0.4,1.37) -- (0.45,1.41) -- (0.5,1.44) -- (0.55,1.39) -- (0.6,1.42) --(0.65,1.37) --(0.7,1.44) -- (0.75,1.37)--(0.8,1.42) --(0.85,1.39)-- (0.9,1.43)--(0.95,1.38)--(1,1.4)--(1.05,1.45)--(1.1,1.37)--(1.15,1.44)--(1.2,1.42)--(1.3,1.37);
\draw [blue,  thick ] (0.3,0.44) --(0.35,0.42)-- (0.4,0.39) -- (0.45,0.4) -- (0.5,0.42) -- (0.55,0.38) -- (0.6,0.4) --(0.65,0.36) --(0.7,0.43) -- (0.75,0.39)--(0.8,0.44) --(0.85,0.41)-- (0.9,0.45)--(0.95,0.42)--(1,0.4)--(1.05,0.43)--(1.1,0.37)--(1.15,0.4)--(1.2,0.42)--(1.3,0.43);
\pgfpathclose
\draw (0.86,1.9) node[right] {$\gamma_{n,i+1}$};
\draw (0.5,1.9) node[left] {$\gamma_{n,i}$};
\draw (0.8,0.75)  node[above] {$I_{n,i}$};
\draw (1.3,1.37) node[right] {$\Gamma_{r(n+1)}$};
\draw (1.3,0.43) node[right] {$\Gamma_{rn}$};
\filldraw (1,1.4) circle (0.3pt)  node[above] {\footnotesize $b_{n,i+1}$};
\filldraw (1,0.4) circle (0.3pt)  node[above] {\footnotesize $a_{n,i+1}$};
\filldraw (0.6,1.42)  circle (0.3pt)  node[above] {\footnotesize $b_{n,i}$};
\filldraw (0.55,0.38) circle (0.3pt)  node[above] {\footnotesize $a_{n,i}$};
\end{tikzpicture}
\caption{Rectangle élémentaire $I_{n,i}$ dans le plan du paramétrage d'une surface $\h$-régulière}
 \label{dessincoaire}
\end{figure}

Comme la composition $g\circ\Phi\circ \gamma_{n,i}$ est de classe $C^1$, l'accroissement de la coordonnée $y$ le long de $\gamma_{n,i}$ est égal à
 $$y(b_{ni})-y(a_{ni})=r\big[(g\circ\Phi\circ \gamma_{n,i})'(a_{n,i})\big]^{-1}+o(r), $$
où le petit-$o$ est uniforme sur $a_{n,i}\in I$.
 Puis,  en se servant de nouveau du lemme \ref{divergence} sur la divergence des courbes horizontales et de la condition de Whitney, nous obtenons
\begin{proposition*}
L'aire euclidienne du "parallélogramme" $I_{n,i}$ est égale à
$$ S(I_{n,i})=r^3 \big|(g\circ\Phi\circ \gamma_{n,i})'\big|^{-1}(a_{n,i})+o(r^3),$$
où le petit-$o$ est uniforme sur $I$.
\end{proposition*}
La condition \eqref{tancoarea} implique l'estimation uniforme lorsque $r\to 0$ suivante
$$\H^2_\infty\big(F^{-1}(0,rn)\cap \Phi(I)\big)=k_nr^2+o(1).$$
La convergence des sommes de Riemann est évidente 
$$\sum\limits_{n=-N}^{N}r\H^2_\infty\big(F^{-1}(0,rn)\cap \Phi(I)\big) \xrightarrow[r\to 0]{} \int\limits_{[-T,T]}\H^2_\infty\big(F^{-1}(0,t)\cap \Phi(I)\big)\, dt.$$
Maintenant on se rappelle la relation suivante (\cite{Heishypesurfaces}, ou le lemme \ref{aireformsurf}): la dérivée de la mesure sphérique $\S^3_\infty$, induite sur $\mathcal{S}$,  par rapport la mesure de Lebesgue $\mathcal{L}^2$ sur $\Omega$ $via$ le paramétrage $\Phi$ est égale à ($c(1)=2$ avec notre normalisation de $\S^3_\infty$)
$$ \frac{d\S^3_\infty \corn \mathcal{S}}{d\mathcal{L}^2}(y,z)=c(1)\sqrt{1+\left(\mfrac{Yf}{Xf}\right)^2}\circ \Phi(y,z).$$

Montrons que les sommes de Riemann $\sum\limits_{n=-N}^{N-1}r k_n r^2 = \sum\limits_{n=-N}^{N}r\H^2_\infty\big(F^{-1}(0,rn)\cap \Phi(I)\big)+o(1)$ le sont également pour l'intégrale sur $\mathcal{S}$ suivante:
$$c(1)^{-1}\int\limits_{\Phi(I)} \left(1+\left(\mfrac{Yf}{Xf}\right)^2\right)^{-1/2}\big|Yg-\mfrac{Yf}{Xf}Xg\big|(a)\, d\S^3_\infty(a).$$
En effet,
\begin{multline*}
\sum\limits_{n=-N}^{N-1}k_n r^2 r = \sum\limits_{n=-N}^{N-1}\sum\limits_{i=0}^{k_n-1} r^3 |(g\circ\Phi\circ \gamma_{n,i})'|^{-1}(a_{n,i}) \, |(g\circ\Phi\circ \gamma_{n,i})'|(a_{n,i}) = \\
=  \sum\limits_{n=-N}^{N-1}\sum\limits_{i=0}^{k_n-1} \big(S(I_{n,i})+o(r^3)\big) \, |(g\circ\Phi\circ \gamma_{n,i})'|(a_{n,i}) =\\
=c(1)^{-1}\sum\limits_{n=-N}^{N-1}\sum\limits_{i=0}^{k_n-1} \big(\S^3_\infty(I_{n,i})+o(r^3)\big) \, \left(1+\left(\mfrac{Yf}{Xf}\right)^2(a_{n,i})\right)^{-1/2} \big|(g\circ\Phi\circ \gamma_{n,i})'\big|(a_{n,i}),
\end{multline*}
 alors que la dérivée $(g\circ\Phi\circ \gamma_{n,i})'=\big(Yg-\mfrac{Yf}{Xf}Xg\big)\circ \Phi\circ \gamma_{n,i}$. Au bout de compte nous avons établi une formule locale:
 \begin{equation*}
c(1)\int\limits_{[-T,T]}\H^2_\infty\big(F^{-1}(0,t)\cap \Phi(I)\big)\, dt=\int\limits_{\Phi(I)} \left(1+\left(\mfrac{Yf}{Xf}\right)^2\right)^{-1/2}\big|Yg-\mfrac{Yf}{Xf}Xg\big|(a)\, d\S^3_\infty(a).
\end{equation*}
 En subdivisant $\Omega$ en des parties par des flots sans pénétration avec les mêmes propriétés que $I$ ci-dessus, nous obtenons la formule de la coaire sur les surfaces $\h$-régulières:
 \begin{equation}\label{surfcoarea}
c(1)\int\limits_{\r}\H^2_\infty\big(F^{-1}(0,t)\cap \S\big)\, dt=\int\limits_{\Phi(\Omega)} \left(1+\left(\mfrac{Yf}{Xf}\right)^2\right)^{-1/2}\big|Yg-\mfrac{Yf}{Xf}Xg\big|(a)\, d\S^3_\infty(a).
\end{equation}
La même formule sera valable (uniformément) pour toutes les surfaces $\h$-régulières $S_s=f^{-1}(s)\cap U$, où $s$ assez petit et $U$ est un voisinage compact de $0$.

Il est temps d'appliquer la formule de la coaire pour des applications scalaires \cite{magncoarea} qu'on explicitera ici dans le cas du groupe d'Heisenberg. Soit $u:(\h,d_\infty)\to (\r,|\cdot|)$
une application lipschitzienne. Alors pour tout ensemble borélien $A\subset \h$ et toute fonction mesurable  $h:A\to \r$ la formule de la coaire suivante a lieu ($\tilde c(1)$ est une constante géométrique)
\begin{equation*}
\int\limits_{A}h(x)\sqrt{(Xu)^2+(Yu)^2}(x)\,d\mathcal{L}^3(x)= \tilde c(1)\int\limits_\r\int\limits_{u^{-1}(s)\cap A} h(x)\,d\S^3_\infty(x)\,ds
\end{equation*}
pourvu que la fonction $h(x)\sqrt{(Xu)^2+(Yu)^2}(x)$ soit sommable.
Dans le cas qui nous intéresse posons $h:=\left(1+\left(\mfrac{Yf}{Xf}\right)^2\right)^{-\frac{1}{2}}\big|Yg-\mfrac{Yf}{Xf}Xg\big|$, $A:=U$,  $u:=f$.  En integrant \eqref{surfcoarea} en $s$ on arrive à
\begin{multline}
\tilde c(1) c(1)\int\limits_{\r^2}\H^2_\infty\big(F^{-1}(b)\cap U\big)\, d\mathcal{L}^2(b)=\tilde c(1) c(1)\int\limits_\r\int\limits_{\r}\H^2_\infty\big(F^{-1}(s,t)\cap U\big)\, dt\,ds=\\
=\tilde c(1)\int\limits_{\r} \int\limits_{f^{-1}(s)\cap U} \left(1+\left(\mfrac{Yf}{Xf}\right)^2\right)^{-\frac{1}{2}}\big|Yg-\mfrac{Yf}{Xf}Xg\big|(a)\, d\S^3_\infty(a)\,ds=\\
=\int\limits_{U} |Xf| |Yg-\mfrac{Yf}{Xf}Xg|(x)\, d\mathcal{L}^3(x)=\int\limits_{U} |Xf Yg- Yf Xg|(x)\, d\mathcal{L}^3(x).
\end{multline}
On rappelle aussi, que la mesure de Lebesgue $\mathcal{L}^3$ sur $\r^3\cong \h$ coïncide à un facteur multiplicatif près avec $\mathcal{H}_\infty^4$.

A partir de là, on procède d'une façon relativement standard. La dernière formule se généralise facilement pour tout ensemble borélien $U\subset \h$, en rassemblant des morceaux locaux, où la formule de la coaire peut déjà être établie.
\end{proof}

\appendix 
\section{Remarques sur l'intégration de Stieltjes}
Ici nous rappelons certains résultats de la théorie de l'intégrale de Stieltjes. 
\begin{theorem*}[\cite{chernyatin}]
Soinet $x,y\in C^0([0,T],\r)$ telles que l'in\-té\-grale $\int_0^T x\, dy$ existe. Alors elle peut être représentée 
\begin{equation*}
\displaystyle \int\limits_0^T x\, dy =\lim\limits_{\tau \to 0} \tau^{-1} \int\limits_0^T \big(y(t+\tau)-y(t)\big)x(t)\, dt.
\end{equation*}
\end{theorem*} 

\begin{theorem*}[\cite{smith}] Si pour une courbe $\gamma=(\gamma_x,\gamma_y)\in C^0([0,1],\r^2)$ l'intégrale de Stieltjes $\int_0^1 \gamma_x\, d\gamma_y$ existe, alors
\begin{equation}\label{enveloppeconvexe}
\limsup\limits_{\delta\to 0}\sum\limits_{i=0}^{n} \mathcal{L}^2\big(\operatorname{EC}\{\gamma([t_i,t_{i+1}])\}\big) =0,
\end{equation}
 où le supremum est pris sur toutes les subdivisions $\{0=t_0<t_1<\ldots<t_n<t_{n+1}=1\}$ avec $\max\limits_i  |t_{i+1}-t_i|\le \delta$, et $\operatorname{EC}\{E\}$ désigne l'enveloppe convexe d'un ensemble $E\subset \r^2$. En particulier, la courbe $\gamma$ doit être d'aire nulle dans le plan : $\mathcal{L}^2\big(\gamma([0,1])\big)=0$. 
\end{theorem*}
 
 \begin{remark} La condition \eqref{enveloppeconvexe} n'est pas suffisante pour l'existence de $\int \gamma_x\, d\gamma_y$. Notons que si $\gamma\in h_{1/2}([0,1],\r^2)$ alors la condition \eqref{enveloppeconvexe} est remplie. Or, on peut trouver un exemple d'une courbe $\gamma\in h_{1/2}$ pour laquelle  $\int \gamma_x\, d\gamma_y$ n'existe pas (voir la sous-section \ref{ondelettes}). 
\end{remark}

\begin{theorem*}[\cite{smith1}] Soit $\gamma\in h_{1/2}([0,1],\r^2)$ une courbe de Jordan (simple et fermée) bordant un ouvert $D\subset \r^2$.  
Alors son intégrale de Stieltjes existe et égale ($\pm$ suivant l'orientation) 
\begin{equation*}
\pm \int\limits_0^1 \gamma_x\, d\gamma_y=\mathcal{L}^2(D).
\end{equation*}
\end{theorem*}

\begin{remark*}
En général, l'intégrale $\int\gamma_x\, d\gamma_y$ n'existe pas même pour les courbes de Jordan (voir \cite{besicovitch}, par exemple).
\end{remark*}

Une autre condition importante suffisante pour l'existence de l'intégrale de Stieltjes est donnée par
\begin{theorem}[\cite{kondurar, young}]\label{young} Soit $x\in H^\alpha([0,T],\r) $ et $y\in H^\beta([0,T],\r)$ avec $\alpha+\beta>1$. Alors l'intégrale
$\int_0^T x\, dy$ existe au sens de Stieltjes et, de plus, pour tout $t\in [0,T]$, 
$$ \big| \int\limits_0^T x\, dy- x(t)(y(T)-y(0))\big| \le C_{\alpha+\beta} \|x\|_{H^\alpha} \|y\|_{H^\beta}\, T^{\alpha+\beta}.$$
\end{theorem}
Ce résultat s'étend également à des modules de continuité pour $f$ et $g$, plus généraux que ceux de type Hölder, voir \cite{younggeneral, burkill} et le lemme \ref{couture}.
\begin{remark}\label{integrationpartie}
Compte tenu de la relation 
\begin{equation}\label{integrationparpartie}
\big|\int_\sigma x\,dy + \int_\sigma y\,dx-xy \Big|_0^1\big|=\big| \sum_{i=0}^N (x(t_{i+1})-x(t_i))(y(t_{i+1})-y(t_i))\big|,  
\end{equation}
on observe que si $x,y\in H^\frac{1}{2}([0,1])$, alors pour toute suite de subdivisions $\sigma_n$, $\|\sigma_n\|\underrel\longrightarrow{n\to \infty} 0$, 
$$\big|\lim\limits_{n\to \infty}\int_{\sigma_n} x\,dy+\lim\limits_{n\to \infty}\int_{\sigma_n} y\,dx - xy \Big|_0^1\big|\le \| x\|_\frac{1}{2}\| y\|_\frac{1}{2},$$ pourvu que les limites en question existent. Si en plus $x\in h_{1/2}$ ou $y\in h_{1/2}$, alors "l'intégration par parties" est valable, \ie la relation
\begin{equation*}
\lim\limits_{n\to \infty}\int_{\sigma_n} x\,dy+\lim\limits_{n\to \infty}\int_{\sigma_n} y\,dx =xy \Big|_0^1,
\end{equation*}
a lieu en supposant seulement qu'une des deux limites existe. 
\end{remark}

\newpage
\addcontentsline{toc}{section}{Références} 
\bibliographystyle{alpha}
\bibliography{H1R2}

\end{document}